\documentclass[11pt]{amsart} \textwidth=14.5cm \oddsidemargin=1cm
\usepackage{stmaryrd}
\usepackage{amsmath}
\usepackage{amsxtra}
\usepackage{amscd}
\usepackage{amsthm}
\usepackage{amsfonts}
\usepackage{amssymb}
\usepackage{arydshln}
\usepackage{eucal}
\usepackage{graphics, color}
\usepackage{hyperref,mathrsfs}
\usepackage[dvipsnames]{xcolor}
\usepackage{tikz}
\usepackage[all]{xy}
\usepackage{hyperref}
\usepackage{color}
\usepackage{mathtools}
\usepackage{bbm} 
\usepackage{pifont}
\usepackage{tikz}
\usetikzlibrary{arrows.meta}
\allowdisplaybreaks

\hypersetup{colorlinks,linkcolor=blue, urlcolor=blue,
citecolor=blue}

\newtheorem{thm}{Theorem} [section]
\newtheorem{lemma}[thm]{Lemma}
\newtheorem{corollary}[thm]{Corollary}
\newtheorem{proposition}[thm]{Proposition}

\theoremstyle{definition}

\numberwithin{equation}{section}

\newtheorem*{thmA}{Theorem A}
\newtheorem*{thmB}{Theorem B}
\newtheorem*{thmC}{Theorem C}
\newtheorem*{thmD}{Theorem D}

\newcommand{\nc}{\newcommand}
\nc{\browntext}[1]{\textcolor{brown}{#1}}
\nc{\greentext}[1]{\textcolor{green}{#1}}
\nc{\redtext}[1]{\textcolor{red}{#1}}
\nc{\bluetext}[1]{\textcolor{blue}{#1}}
\nc{\brown}[1]{\browntext{ #1}}
\nc{\green}[1]{\greentext{ #1}}
\nc{\red}[1]{\redtext{ #1}}
\nc{\blue}[1]{\bluetext{ #1}}

\title[Affine quantum Schur algebras and $\imath$quantum groups]{Affine quantum Schur algebras and $\imath$quantum groups with three parameters}

\author[Li Luo]{Li Luo}
\author[Xirui Yu]{Xirui Yu}
\address{School of Mathematical Sciences, Key Laboratory of MEA (Ministry of Education) \& Shanghai Key Laboratory of PMMP, East China Normal University, Shanghai 200241, China}
\email{lluo@math.ecnu.edu.cn (Luo)\\51255500105@stu.ecnu.edu.cn (Yu)}

\begin{document}


\begin{abstract}
We study the affine quantum Schur algebras corresponding to the affine Hecke algebras of type C with three parameters. Multiplication formulas for semisimple generators are derived for these algebras. We prove that they admit a stabilization property in the sense of Beilinson-Lusztig-MacPherson, by which we construct quasi-split $\imath$quantum groups of affine type AIII with three parameters. Upon specialization of parameters, we recover some known multiplication formulas for affine quantum Schur algebras of types C and D with equal parameters. We take this opportunity to demonstrate in the appendix that different forms of multiplication formulas for affine quantum Schur algebras of types C with equal parameters across the literature are indeed equivalent. 
\end{abstract}

\maketitle


\section{Introduction}
\subsection{Background}
The quantum Schur algebra $\mathbf{S}_{n,d}$ of finite type A, introduced by Dipper and James \cite{DJ89}, is an endomorphism algebra of a sum of permutation modules of the Hecke algebra $\mathbf{H}(\mathfrak{S}_d)$ associated with the symmetric group $\mathfrak{S}_d$. It can also be interpreted as a convolution algebra on pairs of partial flag varieties \cite{BLM90}, thereby generalizing Iwahori-Matsumoto's geometric construction for Hecke algebras (cf. \cite{IM65}). Using this interpretation, Beilinson, Lusztig, and MacPherson (abbr. BLM) derived a stabilization property using some closed multiplication formulas for Chevalley generators of $\mathbf{S}_{n,d}$, which further provides a realization of the general linear quantum group $\mathbf{U}_q(\mathfrak{gl}_n)$ in the projective limit of Schur algebras $\mathbf{S}_{n,d}$ (as $d\to \infty$). This construction of $\mathbf{U}_q(\mathfrak{gl}_n)$ matches well the Schur-Jimbo duality given in \cite{Jim86}, indicating a natural surjective homomorphism $\mathbf{U}_q(\mathfrak{gl}_n)\to \mathbf{S}_{n,d}$, whose restriction $\mathbf{U}_q(\mathfrak{sl}_n)\to \mathbf{S}_{n,d}$ is also surjective.

Several works have also been performed on the affine type A Schur algebra $\widetilde{\mathbb{S}}_{n,d}$. In particular, Lusztig \cite{Lu99} observed that the natural homomorphism from the affine quantum group $\mathbf{U}_q(\widetilde{\mathfrak{sl}}_n)$ to the affine quantum Schur algebra $\widetilde{\mathbb{S}}_{n,d}$ is no longer surjective. In other words, the Chevalley generators do not generate the whole $\widetilde{\mathbb{S}}_{n,d}$, but a proper subalgebra $\widetilde{\mathbf{U}}_{n,d}$ named the Lusztig algebra. Consequently, merely providing multiplication formulas for Chevalley generators of $\widetilde{\mathbb{S}}_{n,d}$ is inadequate to reveal the stabilization procedure from $\widetilde{\mathbb{S}}_{n,d}$ to $\mathbf{U}_q(\widetilde{\mathfrak{gl}}_n)$. This difficulty was addressed by Du and Fu in \cite{DF15} using the multiplication formulas for semisimple generators, rather than those for Chevalley generators. Notably, in Du-Fu's work, some geometric quantities, such as dimensions of certain varieties, were converted into combinatorial quantities related to the length of Weyl group elements, among others. The advantage of this approach is that it allows the method to be formally generalized to many scenarios that defy geometric description in the sense of Iwahori, such as the cases of super version \cite{DG14}, Lusztig's unequal parameters \cite{LL21}, and so on.

Thanks to the influential work \cite{BW18} by Bao and Wang on $\imath$quantum groups, the finite (resp. affine) type B/C counterparts of BLM construction have also been established in \cite{BKLW18} (resp. \cite{FLLLW20,FLLLW23}). Furthermore, there have been some developments in the context of type D (see \cite{FL15, CF24, DLZ25}). In fact, type BCD can be consistently addressed within a framework of unequal parameters, which was first studied in \cite{BWW18} (resp. \cite{FLLLWW20}) on $\imath$Schur duality with multiparameter for finite (resp. affine) type. Since the BLM construction of finite type with two unequal parameters has been dealt with in \cite{LL21}, it is natural to study the case of affine type with three unequal parameters.   

The aim of this paper is to explore the
stabilization phenomenon for affine type C with three unequal parameters. Precisely, we provide multiplication formulas for semisimple generators of the affine Schur algebra $\mathbb{S}^{\mathfrak{c}}_{n,d}$ corresponding to the Hecke algebra $\mathbb{H}$ of affine type C with parameters $(q,q_0,q_1)$, by which the stabilization property is verified. Then we realize a coideal subalgebra $\mathbb{U}^\mathfrak{c}$ of the quantum group $\mathbb{U}=\mathbf{U}_q(\mathfrak{gl}_n)$ or $\mathbf{U}_q(\mathfrak{sl}_n)$ such that $(\mathbb{U}, \mathbb{U}^\mathfrak{c})$ forms a quantum symmetric pair in the sense of \cite{Le99,Ko14}. 

Besides Iwahori-Matsumoto's construction of the affine Hecke algebra associated with an algebraic group $G$, there is also an equivariant K-theoretic description of this affine Hecke algebra via $K^{G^L\times\mathbb{C}\setminus\{0\}}(Z)$, where $Z$ is the Steinberg variety of the Langlands dual $G^L$. The isomorphism of these two different constructions, called the Langlands reciprocity, plays an important role in the local geometric Langlands correspondence. Kato \cite{Ka09} obtained the equivariant K-theoretic realization of affine Hecke algebras of affine type C with three parameters using his exotic setup. In \cite{LXY}, the first author, together with Xu and Yang, applied Kato's method to construct the affine Schur algebra $\mathbb{S}^{\mathfrak{c}}_{n,d}$ with parameters $(q,q_0,q_1)$. So the present paper can be regarded as a companion of \cite{LXY} in the sense of Langlands reciprocity. 

\subsection{Main results}
\subsubsection{Multiplication formula}
The affine Schur algebra $\mathbb{S}^{\mathfrak{c}}_{n,d}$ has a basis $\{e_A~|~A\in \Xi_{n,d}\}$ indexed by a set $\Xi_{n,d}$ of some integer matrices. Each matrix in $\Xi_{n,d}$ determined by a triple $(\lambda,g,\mu)$, where $\lambda$ and $\mu$ are weak compositions of $d$, and $g$ is in the set $\mathscr{D}_{\lambda\mu}$ of minimal length double coset representatives of the Weyl group $W$ of affine type C. The semisimple generators of $\mathbb{S}^{\mathfrak{c}}_{n,d}$ are in the form of $e_A$ with tridiagonal matrices $A\in\Xi_{n,d}$. Similarly as in the affine type A case, the Chevalley generators of $\mathbb{S}^{\mathfrak{c}}_{n,d}$ do not generate the whole $\mathbb{S}^{\mathfrak{c}}_{n,d}$, either. Thus, the first main result is a multiplication formula for semisimple generators of $\mathbb{S}^{\mathfrak{c}}_{n,d}$. For notation, we refer to \S\ref{sec:3}.
\begin{thmA}[Theorem~\ref{1.15}]
Let $A$, $B\in\Xi_{n,d}$ with $B$ tridiagonal and $\mathrm{row}_{\mathfrak{c}}(A)=\mathrm{col}_{\mathfrak{c}}(B)$. Then
\begin{equation*}
e_{B}e_{A}
=\sum_{\scalebox{0.7}{$\substack{T\in\Theta_{B,A} \\ S\in \Gamma_{T}}$}}(q^{-2}-1)^{n(S)}q_0^{\alpha_0}q_1^{\alpha_1}q^{\alpha}\frac{[A^{(T-S)}]_\mathfrak{c}^!}{[A-T_{\theta}]_\mathfrak{c}^![S]^![T-S]^!}\llbracket  S\rrbracket e_{A^{(T-S)}}.
\end{equation*}
\end{thmA}
We emphasize that this multiplication formula is the heart of the present work. The computational and combinatorial details are much more challenging and tedious than in the case of a single parameter.

In the specialization $(q,q_0,q_1)=(q,1,q^2)$, we rediscover the multiplication formulas for affine type C (with a single parameter) shown in \cite{FLLLW23}. These multiplication formulas for Schur algebras of affine type C (with a single parameter) were written in a different form in \cite{FL19}. In the appendix, we take advantage of the opportunity to directly verify the equivalence of the formulas in \cite{FLLLW23} and \cite{FL19}; see Proposition~\ref{prop:compare}.

In the specialization $(q,q_0,q_1)=(q,q_0,q_0)$ (resp. $(q,q,q)$), we get multiplication formulas for affine type B with two parameters $(q,q_0)$ (resp. a single parameter $q$). These formulas are definitely new and are expected to be used in the geometric stabilization procedure by considering the partial flag varieties of affine type B.

In the specialization $(q,q_0,q_1)=(q,1,1)$, we obtain multiplication formulas for affine type D. In \cite{CF24}, multiplication formulas for Chevalley generators are provided. We prove that they are special cases of our formulas.

\subsubsection{Canonical bases at specialization}
There is still a bar involution on the affine Hecke algebra $\mathbb{H}$ for the case of unequal parameters. In order to give a compatible bar on $\mathbb{S}^{\mathfrak{c}}_{n,d}$, we have to prove in Proposition~\ref{bar} that the free $\mathbb{Z}[q^{\pm\frac{1}{2}},q_0^{\pm\frac{1}{2}},q_1^{\pm\frac{1}{2}}]$-module $\mathbb{H}_{\lambda\mu}$ is closed under the bar involution. 
In \cite{LL21}, a finite-type version of Proposition~\ref{bar} was presented, but the proof therein has a gap. We take this opportunity to provide a correct proof for the affine-type version, which applies similarly to the finite-type case in {\em loc. cit}.

The existence of a bar involution on $\mathbb{S}^{\mathfrak{c}}_{n,d}$ allows us to give a standard basis $\{[A]~|~A\in\Xi_{n,d}\}$ of $\mathbb{S}^{\mathfrak{c}}_{n,d}$ satisfying $\overline{[A]}=[A]+$ lower terms.
The multiplication formula given in Theorem A is written in terms of the standard basis in Theorem~\ref{standard}, by which we can construct a canonical basis $\{\{A\}^\mathbf{L}~|~A\in\Xi_{n,d}\}$ of $\mathbb{S}^{\mathfrak{c},\mathbf{L}}_{n,d}$, which is the specialization of $\mathbb{S}^{\mathfrak{c}}_{n,d}$ at $q=\boldsymbol{v}^{-\mathbf{L}(s_1)}$, $q_0=\boldsymbol{v}^{-\mathbf{L}(s_0)+\mathbf{L}(s_d)}$, $q_1=\boldsymbol{v}^{-\mathbf{L}(s_0)-\mathbf{L}(s_d)}$ for a weight function $\mathbf{L}$.

\begin{thmB}[Theorem~\ref{thm:canonicalatspe}]
For any weight function $\mathbf{L}$, the specialization $\mathbb{S}^{\mathfrak{c},\mathbf{L}}_{n,d}$ admits a canonical basis.
\end{thmB}

In the spirit of Beilinson-Lusztig-MacPherson \cite{BLM90}, we employ the multiplication formula in Theorem A to establish a stabilization property of the family $\mathbb{S}^{\mathfrak{c}}_{n,d}$ as $d\to\infty$, which leads to the construction of a stabilization algebra $\dot{\mathbb{K}}^\mathfrak{c}_n$ and further a specialization $\dot{\mathbb{K}}^{\mathfrak{c},\mathbf{L}}_n$. The stabilization procedure allows us to provide a monomial basis and a stably canonical basis of $\dot{\mathbb{K}}^{\mathfrak{c},\mathbf{L}}_n$ from the ones of $\mathbb{S}^{\mathfrak{c},\mathbf{L}}_{n,d}$. Here we call it a stably canonical basis because of ‌the absence of the positivity property.

\begin{thmC}[Corollary~\ref{5.3} \& Theorem~\ref{thm:canL}]
There is an algebra $\dot{\mathbb{K}}_{n}^\mathfrak{c}$ arising from stabilization on the family of affine Schur algebras $\mathbb{S}^{\mathfrak{c}}_{n,d}$ as $d$ varies. Moreover, its specialization $\dot{\mathbb{K}}_{n}^{\mathfrak{c},\mathbf{L}}$ admits a stably canonical basis.
\end{thmC}

\subsubsection{Affine $\imath$quantum groups}
Similarly to the case of equal parameter, we expect that the algebra $\dot{\mathbb{K}}_{n}^\mathfrak{c}$ is the modified version of an affine $\imath$quantum group $\mathbb{K}_{n}^\mathfrak{c}$ with three parameters, which is a coideal subalgebra of the affine quantum $\mathfrak{gl}_n$ (over $\mathbb{Q}(q^{\frac{1}{2}},q_0^{\frac{1}{2}},q_1^{\frac{1}{2}})$). 
However, since we have not introduced a co-multiplication for $\mathbb{S}^{\mathfrak{c}}_{n,d}$ in this paper, we do not intend to verify this expectation by studying the properties of co-multiplication (as done in \cite{FLLLW23}), but leave this for subsequent work. Instead, we give the corresponding $\mathfrak{sl}$-counterpart via direct verification, namely, we directly show that the algebra $\dot{\mathcal{K}}_{n}^\mathfrak{c}$ obtained from the Lusztig algebra $\mathbb{U}^\mathfrak{c}_{n,d}$,  the subalgebra of $\mathbb{S}^\mathfrak{c}_{n,d}$ generated by Chevalley generators, through the stabilization process is precisely the modified version of the affine $\imath$quantum group $\mathbb{U}^\mathfrak{c}(\widetilde{\mathfrak{sl}}_n)$ with three parameters introduced in \cite{FLLLWW20}. 

\begin{thmD}[Theorem \ref{thm:isoUK}]
The algebra $\dot{\mathcal{K}}_{n}^\mathfrak{c}$ is isomorphic to the modified version of $\mathbb{U}^\mathfrak{c}(\widetilde{\mathfrak{sl}}_n)$. 
\end{thmD}

\subsubsection{Variants}
As the same as the case of equal parameter, there are three variants (called type $\imath\jmath$, $\jmath\imath$ and $\imath\imath$) of the algebras $\mathbb{S}_{n,d}^\mathfrak{c}$, $\mathbb{U}_{n,d}^\mathfrak{c}$ and $\dot{\mathbb{K}}_n^\mathfrak{c}$ (here we regard $\mathfrak{c}=\jmath\jmath$). The affine Schur algebras $\mathbb{S}_{n,d}^{\imath\jmath}$, $\mathbb{S}_{n,d}^{\jmath\imath}$ and $\mathbb{S}_{n,d}^{\imath\imath}$ are subalgebras of $\mathbb{S}_{n,d}^\mathfrak{c}$, while their stabilization algebras $\mathbb{K}_{n}^{\imath\jmath}$, $\mathbb{K}_{n}^{\jmath\imath}$ and $\mathbb{K}_{n}^{\imath\imath}$ are subalgebras of $\mathbb{K}_{n}^\mathfrak{c}$. We obtain conclusions about the (stably) canonical bases for these algebras parallel those for $\mathbb{S}_{n,d}^\mathfrak{c}$ of $\mathbb{K}_{n}^\mathfrak{c}$. Moreover, the (stably) canonical bases for them are compatible with the ones in $\mathbb{S}_{n,d}^\mathfrak{c}$ or $\mathbb{K}_{n}^\mathfrak{c}$ under inclusion.

\subsection{Organization}
Here is a layout of the paper. In Section 2, we introduce the affine Schur algebra $\mathbb{S}_{n,d}^\mathfrak{c}$ with three parameters, which has a basis labeled by a set $\Xi_{n,d}$ consisting of certain $\mathbb{Z}\times \mathbb{Z}$ integer matrices. Some quantum combinatorics is studied for the setup of three parameters. Section 3 is the heart of this paper, in which multiplication formulas for semisimple generators of $\mathbb{S}_{n,d}^\mathfrak{c}$ are achieved. The canonical basis at the specialization is established in Section 4. Section 5 is devoted to the construction of the stabilization algebra $\mathbb{K}^\mathfrak{c}_n$.  In Section 6 we prove that the stabilization algebra obtained from Lusztig algebras is just the modified version of the affine $\imath$quantum group $\imath$quantum group $\mathbb{U}^\mathfrak{c}(\widetilde{\mathfrak{sl}}_n)$ with three parameters. The other three variants of type $\imath\jmath$, $\jmath\imath$ and $\imath\imath$ are also studied in Sections 5 and 6. In the appendix, we show that our multiplication formulas specialize to those appearing in \cite{FLLLW23} and \cite{CF24} for affine type C and D with an equal parameter. In particular, we take advantage of the opportunity to directly verify the equivalence of the multiplication formulas provided in \cite{FLLLW23} and \cite{FL19}.

\vspace{0.3cm}
\noindent{\bf Acknowledgments}\quad
LL is partially supported by the National Key R\&D Program of China (No. 2024YFA1013802) and the NSF of China (No. 12371028).

\section{Preliminary}
\subsection{Notation}
Let $\mathbb{N}=\{0,1,2,\ldots\}$ and $\mathbb{Z}$ denote the sets of natural numbers and integers, respectively. Fix $r,d\in\mathbb{N}$ with $d\geqslant2$. Let
\begin{equation*}
n=2r+2,\quad D=2d+2.
\end{equation*}
For any $a\leqslant b\in\mathbb{Z}$, denote $[a..b]:=\{a,a+1,\ldots,b\}$, $(a..b)=\{a+1,a+2,\ldots,b-1\}$ and similar notation $[a..b)$ and $(a..b]$.

\subsection{Affine Weyl group}

Let $W$ be the affine Weyl group of type $\widetilde{C}_{d}$ generated by $S=\{s_0,s_1,\ldots,s_d\}$ with the affine Dynkin diagram
\begin{center}
\begin{tikzpicture}[scale=.4]
\node at (0,-1) {$0$};
\node at (4,-1) {$1$};
\node at (12,-1) {$d-1$};
\node at (16,-1) {$d$.}; 
    
\draw[thick, double distance=2pt, -{Implies}, shorten >=2pt] (0,0) ++(0.5,0) -- +(3,0);
\draw[thick, double distance=2pt, -{Implies}, shorten >=2pt] (15,0) -- (12.5,0);

\foreach \x in {2,4.5}
\draw[thick,xshift=\x cm] (\x,0) ++(0.5,0) -- +(2,0);
       
\node at (8,0) {$\dots$};
    
\foreach \x in {0,2,6,8} 
\draw[thick,xshift=\x cm] (\x, 0) circle (0.3); 

\end{tikzpicture}
\end{center}
We identify $W$ with the permutation group of $\mathbb{Z}$ consisting of the permutations $g$ satisfying
\begin{equation*}
g(i+D)=g(i)+D \quad\mbox{and}\quad g(-i)=-g(i)\quad\mbox{for $i\in\mathbb{Z}$}.
\end{equation*}
In particular,
\begin{equation*}
g(0)=0 \quad \mbox{and}\quad g(d+1)=d+1.
\end{equation*}
Denote by
$(i,j)_{\mathfrak{c}}\in W, (i\neq j)$
the transposition swapping $kD\pm i$ and $kD\pm j$ $(k\in\mathbb{Z})$ and fixing $\mathbb{Z}\setminus \{kD\pm i,kD\pm j~|~k\in\mathbb{Z}\}$ pointwise.
As permutations of $\mathbb{Z}$, the simple generators are identified by
\begin{equation*}
s_0=(1,-1)_{\mathfrak{c}},\quad
s_d=(d,d+2)_{\mathfrak{c}},\quad
s_i=(i,i+1)_{\mathfrak{c}}, \ (i=1,\ldots,d-1).
\end{equation*}

Let $\ell: W\rightarrow\mathbb{N}$ be the length function of $W$. We introduce $\ell_{\mathfrak{c}_0}:~W\rightarrow\mathbb{N}$ (resp. $\ell_{\mathfrak{c}_d}:~W\rightarrow\mathbb{N}$) such that $\ell_{\mathfrak{c}_0}(g)$ (resp. $\ell_{\mathfrak{c}_d}(g)$) equals the total number of $s_0$ (resp. $s_d$) appearing in a reduced expression of $g$. The functions $\ell_{\mathfrak{c}_0}$ and $\ell_{\mathfrak{c}_d}$ are independent of the choice of a reduced form. We set $\ell_{\mathfrak{a}}=\ell-\ell_{\mathfrak{c}_0}-\ell_{\mathfrak{c}_d}$.

\begin{lemma}\label{length}
The length of $g\in W$ is given by
\begin{align*}
\ell(g)
=&\frac{1}{2}\sharp\Bigg\{(i,j)\in[1..d]\times\mathbb{Z}~\Big|~
\begin{array}{c}
i>j\\
g(i)<g(j)
\end{array}\mbox{or}
\begin{array}{c}
i<j\\
g(i)>g(j)
\end{array}
\Bigg\},\\
\ell_{\mathfrak{c}_{0}}(g)
=&\frac{1}{2}\sharp\Bigg\{(i,j)\in\{0\}\times\mathbb{Z}~\Big|~
\begin{array}{c}
i>j\\
g(i)<g(j)
\end{array}\mbox{or}
\begin{array}{c}
i<j\\
g(i)>g(j)
\end{array}
\Bigg\},\\
\ell_{\mathfrak{c}_{d}}(g)
=&\frac{1}{2}\sharp\Bigg\{(i,j)\in\{d+1\}\times\mathbb{Z}~\Big|~
\begin{array}{c}
i>j\\
g(i)<g(j)
\end{array}\mbox{or}
\begin{array}{c}
i<j\\
g(i)>g(j)
\end{array}
\Bigg\}.
\end{align*}
\end{lemma}
\begin{proof}
The formula for $\ell(g)$ is proven in \cite[Lemma 2.1]{FLLLW23}.
By an induction on the length of a reduced expression of $g$, we have
\begin{align*}
\ell_{\mathfrak{c}_{0}}(g)
&=\sharp\{i\in\mathbb{N}~|~g(i)<0\}=\frac{1}{2}\sharp\{i\in\mathbb{Z}~|~i>0,g(i)<0~\text{or}~i<0,g(i)>0\}\\
&=\frac{1}{2}\sharp\Bigg\{(i,j)\in\{0\}\times\mathbb{Z}~\Bigg|~
\begin{array}{c}
i>j\\
g(i)<g(j)
\end{array}\text{or}
\begin{array}{c}
i<j\\
g(i)>g(j)
\end{array}
\Bigg\}.
\end{align*}
Similarly, $\ell_{\mathfrak{c}_{d}}(g)$ can be obtained. 
\end{proof}

Denote the set of weak compositions of $d$ into $r+2$ parts by
\begin{equation*}
\Lambda=\Lambda_{r,d}:=\{\lambda=(\lambda_0,\lambda_1,\ldots,\lambda_{r+1})\in\mathbb{N}^{r+2}~|~\sum_{i=0}^{r+1}\lambda_i=d\}.
\end{equation*}
For $\lambda\in\Lambda$, we shall denote by $W_{\lambda}$ the parabolic subgroup of W generated by $S\setminus \{s_{\lambda_0},s_{\lambda_0,1},\ldots,s_{\lambda_0,r}\}$, where $\lambda_{0,i}=\lambda_0+\lambda_1+\cdots+\lambda_i$ for $0\leqslant i\leqslant r$; note $\lambda_{0,0}=\lambda_0$ and $\lambda_{0,r}=d-\lambda_{r+1}$. 

Let
\begin{equation*}
\mathscr{D}_\lambda:=\{g\in W~|~\ell(wg)=\ell(w)+\ell(g),~\forall w\in W_{\lambda}\}.
\end{equation*}
Then $\mathscr{D}_{\lambda}$ (resp. $\mathscr{D}_{\lambda}^{-1}$) is the set of minimal length right (resp. left) coset representatives of $W_{\lambda}$ in $W$, and 
\begin{equation*}
\mathscr{D}_{\lambda\mu}:=\mathscr{D}_{\lambda}\cap\mathscr{D}_{\mu}^{-1}
\end{equation*}
is the set of minimal length double coset representatives for $W_{\lambda}\backslash W/W_{\mu}$.

In the following, we list some known results (see \cite[Propasition 4.16, Lemma 4.17 \& Theorem 4.18]{DDPW08}).
\begin{proposition}\label{2.2}
Let $\lambda,\mu\in\Lambda$ and $g\in\mathscr{D}_{\lambda\mu}$.
\begin{itemize}
\item[(a)] There exists $\delta\in\Lambda_{r',d}$ for some $r'$ such that $W_{\delta}=g^{-1}W_{\lambda}g\cap W_{\mu}$.
\item[(b)] The map $W_{\lambda}\times(\mathscr{D}_{\delta}\cap W_{\mu})\rightarrow W_{\lambda}gW_{\mu}$ sending $(x,y)$ to $xgy$ is a bijection; moreover, we have $\ell(xgy)=\ell(x)+\ell(g)+\ell(y)$.
\item[(c)] The map $W_{\lambda}\times(\mathscr{D}_{\delta}\cap W_{\mu})\rightarrow W_{\mu}$ sending $(x,y)$ to $xy$ is a bijection; moreover, we have $\ell(xy)=\ell(x)+\ell(y)$.
\end{itemize}
\end{proposition}

\subsection{Affine Hecke algebra}

Denote $$\mathbb{A}=\mathbb{Z}[q^{\pm\frac{1}{2}},q_0^{\pm\frac{1}{2}},q_1^{\pm\frac{1}{2}}].$$ 
Let $\mathbb{H}=\mathbb{H}(W)$ be the affine Hecke algebra of type $\widetilde{C}_d$ with three parameters, which is an $\mathbb{A}$-algebra generated by
\begin{equation*}
T_i~(0\leqslant i\leqslant d),
\end{equation*}
subject to the following relations: for $0\leqslant i,j,k\leqslant d$,
\begin{gather*}
(T_0-q_0^{-1})(T_0+q_1)=0,\quad (T_i-q^{-1})(T_i+q)=0~(i\neq0),\quad (T_d-q_1^{-1})(T_d+q_0^{-1})=0,\\
T_iT_j=T_jT_i\qquad (|i-j|>1),\\
(T_0T_1)^2=(T_1T_0)^2,\quad T_kT_{k-1}T_k=T_{k-1}T_kT_{k-1} \quad (k\neq0,1,d),\quad (T_{d-1}T_{d})^2=(T_{d}T_{d-1})^2.
\end{gather*}

For any $w\in W$ with a reduced form $w=s_{i_1}\cdots s_{i_l}$, set
\begin{equation*}
T_w:=T_{i_1}\cdots T_{i_l} \quad\mbox{and}\quad q_w:=q_{s_{i_1}}\cdots q_{s_{i_l}},
\end{equation*}
where
\begin{equation*}
q_{s_i}:=\left\{
\begin{aligned}
&q_1,&&\text{if}~i=0,\\
&q,&&\text{if}~i\neq0,d,\\
&q_0^{-1},&&\text{if}~i=d.
\end{aligned}
\right.
\end{equation*}
It is known that $T_w$ and $q_w$ are both independent of the choice of a reduced form. 

\subsection{Affine quantum Schur algebra}

For any finite subset $X\subset W$ and $\lambda\in\Lambda$, set
\begin{equation*}
T_{X}:=\sum_{w\in X}q_w^{-1}T_w\quad \text{and}\quad x_{\lambda}:=T_{W_{\lambda}}.
\end{equation*}
\begin{lemma}
For $\lambda,\mu\in\Lambda$ and $g\in\mathscr{D}_{\lambda\mu}$, let $\delta$ be the element satisfying Proposition~\ref{2.2}. Then $T_{W_{\lambda} g W_{\mu}}=q_g^{-1} x_{\lambda}T_gT_{\mathscr{D}_{\delta}\cap W_{\mu}}$.
\end{lemma}
\begin{proof}
It is a direct computation that  
\begin{align*}
T_{W_{\lambda} g W_{\mu}}
&=\sum_{w\in W_{\lambda} g W_{\mu}}q_w^{-1}T_w=\sum_{\scalebox{0.7}{$\substack{x\in W_{\lambda}\\ y\in\mathscr{D}_{\delta}\cap W_{\mu}}$}}q_{xgy}^{-1}T_x T_g T_y\quad\mbox{(by Prop.~\ref{2.2}(b))}\\&=(\sum_{x\in W_{\lambda}}q_{x}^{-1}T_x)(q_{g}^{-1}T_g)(\sum_{y\in\mathscr{D}_{\delta}\cap W_{\mu}}q_{y}^{-1}T_y)=q_g^{-1} x_{\lambda}T_gT_{\mathscr{D}_{\delta}\cap W_{\mu}}.
\end{align*}
\end{proof}

The affine quantum Schur algebra $\mathbb{S}_{n,d}^{\mathfrak{c}}$ is defined as the following $\mathbb{A}$-algebra
\begin{equation*}
\mathbb{S}_{n,d}^{\mathfrak{c}}:=\mathrm{End}_{\mathbb{H}}(\mathop{\oplus}\limits_{\lambda\in\Lambda}x_{\lambda}\mathbb{H})=\mathop{\bigoplus}\limits_{\lambda,\mu\in\Lambda}\mathrm{Hom}_{\mathbb{H}}(x_{\mu}\mathbb{H},x_{\lambda}\mathbb{H}).
\end{equation*}

\begin{lemma}[{\cite[Lemma 3.1]{FLLLWW20}}]\label{2.5}
For $\lambda\in\Lambda$ and $i\in[0..d]\backslash \{\lambda_0,\lambda_{0,1},\cdots,\lambda_{0,r}\}$, we have
\begin{equation*}
x_{\lambda}T_i=\left\{
\begin{aligned}
&q_0^{-1}x_{\lambda},&&\text{if}~i=0,\\
&q^{-1}x_{\lambda},&&\text{if}~i\neq0,d,\\
&q_1^{-1}x_{\lambda},&&\text{if}~i=d.
\end{aligned}
\right.
\end{equation*}
\end{lemma}

For $\lambda,\mu\in\Lambda$ and $g\in\mathscr{D}_{\lambda\mu}$, we define an $\mathbb{H}$-homomorphism 
$$\phi_{\lambda\mu}^g\in\mathrm{Hom}_{\mathbb{H}}(x_{\mu}\mathbb{H},x_{\lambda}\mathbb{H})\quad \mbox{by}\quad
x_{\nu}\mapsto \delta_{\mu\nu}T_{W_{\lambda} g W_{\mu}}=\delta_{\mu\nu}q_g^{-1} x_{\lambda}T_gT_{\mathscr{D}_{\delta}\cap W_{\mu}}.$$
The following lemma is obtained by a standard argument (cf. \cite{Du92}).
\begin{lemma}\label{lem:basis}
The set $\{\phi_{\lambda\mu}^g~|~\lambda$, $\mu\in\Lambda\ \mbox{and}\ g\in\mathscr{D}_{\lambda\mu}\}$ forms an $\mathbb{A}$-basis of $\mathbb{S}_{n,d}^{\mathfrak{c}}$.   
\end{lemma}

\subsection{Matrix description}
Denote
\begin{equation*}
\Theta_n:=\{A=(a_{ij})\in\mathrm{Mat}_{\mathbb{Z}\times\mathbb{Z}}(\mathbb{N})~|~a_{ij}=a_{i+n,j+n},~\forall i,j\in\mathbb{Z}\}
\end{equation*}
and its subset
\begin{equation*}
\Xi_n:=\bigcup\limits_{d\in\mathbb{N}}\Xi_{n,d},
\end{equation*} where
\begin{align*}
\Xi_{n,d}:=\{A=(a_{ij})\in\mathrm{Mat}_{\mathbb{Z}\times\mathbb{Z}}(\mathbb{N})~|~&a_{-i,-j}=a_{ij}=a_{i+n,j+n},~\forall i,j\in\mathbb{Z};\\
&a_{00},a_{r+1,r+1}~\text{are odd};~\sum_{1\leqslant i\leqslant n}\sum_{j\in\mathbb{Z}}a_{ij}=D\}.
\end{align*}
For $A=(a_{ij})\in\Xi_{n,d}$, we shall always denote 
\begin{equation*}
a'_{ij}=\left\{
\begin{array}{ll}
\frac{1}{2}(a_{ii}-1), &\mbox{if $i=j\equiv 0,r+1\ (\mathrm{mod}\ n)$},\\
a_{ij}, &\mbox{otherwise}.
\end{array}
\right.
\end{equation*}
It was shown in \cite[Lemma~2.7]{FLLLW23} that there is a bijection
$$\kappa: \bigsqcup_{\lambda,\mu\in\Lambda}\{\lambda\}\times\mathscr{D}_{\lambda\mu}\times\{\mu\} \rightarrow\Xi_{n,d},\qquad (\lambda,g,\mu)\mapsto A=\kappa(\lambda,g,\mu),$$
where $$\lambda=\mathrm{row}_{\mathfrak{c}}(A):=(\mathrm{row}_{\mathfrak{c}}(A)_0,\ldots,\mathrm{row}_{\mathfrak{c}}(A)_{r+1}),~ \mu=\mathrm{col}_{\mathfrak{c}}(A):=(\mathrm{col}_{\mathfrak{c}}(A)_0,\ldots,\mathrm{col}_{\mathfrak{c}}(A)_{r+1})$$
with \begin{align*}
\mathrm{row}_{\mathfrak{c}}(A)_k&=\left\{
\begin{aligned}
&a'_{00}+\sum_{j\geqslant1}a_{0j},&&\text{if}~k=0,\\
&a'_{r+1,r+1}+\sum_{j\leqslant r}a_{r+1,j},&&\text{if}~k=r+1,\\
&\sum_{j\in\mathbb{Z}}t_{kj},&&\text{if}~1\leqslant k\leqslant r
\end{aligned}
\right.\quad \mbox{and}\\ 
\mathrm{col}_{\mathfrak{c}}(A)_k&=\left\{
\begin{aligned}
&a'_{00}+\sum_{i\geqslant1}a_{i0},&&\text{if}~k=0,\\
&a'_{r+1,r+1}+\sum_{i\leqslant r}a_{i,r+1},&&\text{if}~k=r+1,\\
&\sum_{i\in\mathbb{Z}}t_{ik},&&\text{if}~1\leqslant k\leqslant r,
\end{aligned}
\right.
\end{align*}
and $g$ is determined by $A$ via a ``column reading'' (cf. \cite[\S2.4]{FLLLW23}).
%
%
%
For each $A=\kappa(\lambda,g,\mu)\in\Xi_{n,d}$, we denote
\begin{equation*}
e_A=\phi_{\lambda\mu}^g.
\end{equation*}
Then Lemma~\ref{lem:basis} is rewritten as 
\begin{lemma}
The set $\{e_A~|~A\in\Xi_{n,d}\}$ forms an $\mathbb{A}$-basis of $\mathbb{S}_{n,d}^{\mathfrak{c}}$.   
\end{lemma}

Let $A\in\Xi_{n,d}$. We choose $k_j\geqslant0$ for $0\leqslant j\leqslant r+1$ so that $a_{ij}=0$ unless $|i-j|\leqslant k_j$. We define a weak composition
\begin{equation*}
\delta(A)\in\Lambda_{\tau,d},
\end{equation*}
with $\tau=k_0+\sum_{j=1}^r(2k_j+1)+k_{r+1}$ as follows. The composition $\delta(A)$ starts with the entries $(k_0+1)$ in the $0$-th column, $a'_{00}$, $a_{10}$, $a_{20}$, $\ldots$, $a_{k_00}$, followed by the entries $(2k_j+1)$ in the $j$-th column when $i$ increases the interval $[j-k_j..j+k_j]$ for $j=1,\ldots,r$, and is finally followed by the entries $(k_{r+1}+1)$ in the $(r+1)$-th column, $a_{r+1-k_{r+1},r+1}$, $\ldots$, $a_{r,r+1}$, $a'_{r+1,r+1}$.
This $\delta(A)$ satisfies Proposition{2.2} (a), i.e. 
$$W_{\delta(A)}=g^{-1}W_\lambda g\cap W_\mu,\quad \mbox{for $A=\kappa(\lambda,g,\nu)\in\Xi_{n,d}$}.$$

Set
\begin{equation*}
I^+=(\{0\}\times\mathbb{N})\sqcup([1..r]\times\mathbb{Z})\sqcup(\{r+1\}\times\mathbb{Z}_{\leqslant r+1}).
\end{equation*} 
Each matrix $A=(a_{ij})\in \Xi_n$ is determined by the entries $a_{ij}$, $((i,j)\in I^+)$. Moreover, we denote
\begin{equation*}
I_{\mathfrak{a}}^+=I^+\backslash\{(0,0),(r+1,r+1)\}.
\end{equation*}

\subsection{Quantum combinatorics}

For $m\in\mathbb{Z}$, $n\in\mathbb{N}$, denote
\begin{equation*}
[m]=\frac{q^{-2m}-1}{q^{-2}-1},\quad [n]^!=[n][n-1]\cdots[1],\quad [0]^!=1,
\end{equation*}
\begin{equation*}
\left[\begin{array}{c}
m\\
n
\end{array}\right]
=\frac{[m][m-1]\cdots[m-n+1]}{[n]^!}.
\end{equation*}

For $T=(t_{ij})\in\Theta_n$, we define
\begin{equation*}
[T]^!=\prod\limits_{i=1}^{n}\prod\limits_{j\in\mathbb{Z}}[t_{ij}]^!.
\end{equation*}
For $A=(a_{ij})\in\Xi_n$, define
\begin{equation*}
[A]_{\mathfrak{c}}^!=[a'_{00}]_{\mathfrak{c}_0}^![a'_{r+1,r+1}]_{\mathfrak{c}_1}^!\prod_{(i,j)\in I_a^+}[a_{ij}]^! \quad\mbox{where}\quad [m]_{\mathfrak{c}_l}^!=\prod_{k=1}^m[2k]_{\mathfrak{c}_l}
\end{equation*}
with
\begin{equation*}
[2k]_{\mathfrak{c}_0}=[k](1+q_0^{-1} q_1^{-1} q^{-2(k-1)})\quad\mbox{and}\quad [2k]_{\mathfrak{c}_1}=[k](1+q_0 q_1^{-1} q^{-2(k-1)}).
\end{equation*}

\begin{lemma}\label{1.4}
For any $A\in\Xi_n$, we have 
\begin{equation*}
[A]_{\mathfrak{c}}^!
=\sum_{w\in W_{\delta(A)}}q_0^{-\ell_{\mathfrak{c}_0}(w)+\ell_{\mathfrak{c}_d}(w)}q_1^{-\ell_{\mathfrak{c}_0}(w)-\ell_{\mathfrak{c}_d}(w)}q^{-2\ell_{\mathfrak{a}}(w)}.
\end{equation*}
\end{lemma}

\begin{proof}
Denote the Weyl group of type $A_{m-1}$ (resp. $C_m$) by $\mathfrak{S}_m$ (resp. $W_{C_m}$). We have $W_{\delta(A)}\simeq W_{C_{\delta_0}}\times \mathfrak{S}_{\delta_1}\times \mathfrak{S}_{\delta_2}\times\cdots\times \mathfrak{S}_{\delta_{\tau}}\times W_{C_{\delta_{\tau+1}}}$. 

For each $w\in \mathfrak{S}_{a_{ij}}$ with $(i,j)\in I_{\mathfrak{a}}^+$, we have $\ell_{\mathfrak{c}_0}(w)=\ell_{\mathfrak{c}_d}(w)=0$ and $\ell_{\mathfrak{a}}(w)=\ell(w)$. Hence
\begin{equation*}
\sum_{w\in \mathfrak{S}_{a_{ij}}}q_0^{-\ell_{\mathfrak{c}_0}(w)+\ell_{\mathfrak{c}_d}(w)}q_1^{-\ell_{\mathfrak{c}_0}(w)-\ell_{\mathfrak{c}_d}(w)}q^{-2\ell_{\mathfrak{a}}(w)}
=\sum_{w\in \mathfrak{S}_{a_{ij}}}q^{-2\ell_{\mathfrak{a}}(w)}
=[a_{ij}]^!.
\end{equation*}

For each $w\in W_{C_{\delta_0}}$, the length $\ell_{\mathfrak{c}_d}(w)=0$ while $\ell_{\mathfrak{c}_0}(w)=0$ for each $w\in W_{C_{\delta_1}}$. 
Thus,
\begin{align*}
&\sum_{w\in W_{\delta(A)}}q_0^{-\ell_{\mathfrak{c}_0}(w)+\ell_{\mathfrak{c}_d}(w)}q_1^{-\ell_{\mathfrak{c}_0}(w)-\ell_{\mathfrak{c}_d}(w)}q^{-2\ell_{\mathfrak{a}}(w)}\\
=&\sum_{w\in W_{C_{\delta_0}}}q_0^{-\ell_{\mathfrak{c}_0}(w)}q_1^{-\ell_{\mathfrak{c}_0}(w)}q^{-2\ell_{\mathfrak{a}}(w)}\sum_{w\in W_{C_{\delta_{\tau+1}}}}q_0^{\ell_{\mathfrak{c}_d}(w)}q_1^{-\ell_{\mathfrak{c}_d}(w)}q^{-2\ell_{\mathfrak{a}}(w)}\prod_{(i,j)\in I_{\mathfrak{a}}^+}[a_{ij}]^!.
\end{align*}

It suffices to show 
\begin{align*}
\sum_{w\in W_{C_d}}q_0^{-\ell_{\mathfrak{c}_0}(w)}q_1^{-\ell_{\mathfrak{c}_0}(w)}q^{-2\ell_{\mathfrak{a}}(w)}
=[d]_{\mathfrak{c}_0}^!\quad\mbox{and}\quad
\sum_{w\in W_{C_d}}q_0^{\ell_{\mathfrak{c}_d}(w)}q_1^{-\ell_{\mathfrak{c}_d}(w)}q^{-2\ell_{\mathfrak{a}}(w)}
=[d]_{\mathfrak{c}_1}^!.
\end{align*}
In the following, we only prove the first equation, while the other one is similar. 

Let $\lambda=(d-1,1,0,\ldots,0)\in\Lambda_{r,d}$. We have $W_{\lambda}\simeq W_{C_{d-1}}$, and hence
\begin{align*}
&\sum_{w\in W_{C_d}}q_0^{-\ell_{\mathfrak{c}_0}(w)}q_1^{-\ell_{\mathfrak{c}_0}(w)}q^{-2\ell_{\mathfrak{a}}(w)}\\
=&\sum_{w\in W_{C_{d-1}}}q_0^{-\ell_{\mathfrak{c}_0}(w)}q_1^{-\ell_{\mathfrak{c}_0}(w)}q^{-2\ell_{\mathfrak{a}}(w)}\sum_{w\in\mathscr{D}_{\lambda}}q_0^{-\ell_{\mathfrak{c}_0}(w)}q_1^{-\ell_{\mathfrak{c}_0}(w)}q^{-2\ell_{\mathfrak{a}}(w)}\\
=&\sum_{w\in W_{C_{d-1}}}q_0^{-\ell_{\mathfrak{c}_0}(w)}q_1^{-\ell_{\mathfrak{c}_0}(w)}q^{-2\ell_{\mathfrak{a}}(w)}\cdot [d](1+q_0^{-1} q_1^{-1} q^{-2(d-1)})\\
=&\sum_{w\in W_{C_{d-1}}}q_0^{-\ell_{\mathfrak{c}_0}(w)}q_1^{-\ell_{\mathfrak{c}_0}(w)}q^{-2\ell_{\mathfrak{a}}(w)}\cdot [2d]_{\mathfrak{c}_0}
=[d]_{\mathfrak{c}_0}^!.
\end{align*} This completes the proof.
\end{proof}

\section{Multiplication formula} \label{sec:3}
\subsection{A multiplication formula for Hecke algebras}

\begin{lemma}\label{1.5}
For $A=\kappa(\lambda,g,\mu)\in\Xi_{n,d}$, we have $$q_g^{-1}x_{\lambda}T_gx_{\mu}=[A]_{\mathfrak{c}}^!e_A(x_{\mu}).$$
\end{lemma}

\begin{proof}
By Proposition~\ref{2.2} we have
\begin{equation*}
x_{\mu}
=\sum_{x\in W_{\mu}}q_x^{-1}T_x
=\sum_{\scalebox{0.7}{$\substack{x\in W_{\lambda} \\ y\in\mathscr{D}_{\delta}\cap W_{\mu}}$}}q_{wy}^{-1}T_{wy}
=\sum_{w\in\mathscr{D}_{\delta}\cap W_{\mu}}q_{w}^{-1}T_{w}\sum_{y\in W_{\delta}}q_{y}^{-1}T_{y}
=T_{\mathscr{D}_{\delta}\cap W_{\mu}}x_{\delta},
\end{equation*}
while by Lemma \ref{1.4} we have 
\begin{equation*}
x_{\mu}x_{\delta}
=\sum_{w\in W_{\delta}}q_w^{-1}x_{\mu}T_w
=\sum_{w\in W_{\delta}}q_0^{-\ell_{\mathfrak{c}_0}(w)+\ell_{\mathfrak{c}_d}(w)}q_1^{-\ell_{\mathfrak{c}_0}(w)-\ell_{\mathfrak{c}_d}(w)}q^{-2\ell_{\mathfrak{a}}(w)}x_{\mu}
=[A]_{\mathfrak{c}}^!x_{\mu}.
\end{equation*}
Hence
$q_g^{-1}x_{\lambda}T_g x_{\mu}
=q_g^{-1}x_{\lambda}T_g T_{\mathscr{D}_{\delta}\cap W_{\mu}}x_{\delta}
=e_A(x_{\mu})x_{\delta}
=e_A(x_{\mu}x_{\delta})
=[A]_{\mathfrak{c}}^!e_A(x_{\mu})$.
\end{proof}

\subsection{A rough multiplication formula}
In this section, we always take $B=\kappa(\lambda,g_1,\mu)$, $A=\kappa(\mu,g_2,\nu)\in\Xi_{n,d}$ and write $\delta=\delta(B)$.
\begin{lemma}\label{1.6}
 We have
\begin{equation*}
e_Be_A(x_{\nu})=\frac{1}{[A]_{\mathfrak{c}}^!}q_{g_1}^{-1}x_{\lambda}T_{g_1}T_{(\mathscr{D}_{\delta}\cap W_{\mu})g_2}x_{\nu}.
\end{equation*}
\end{lemma}
\begin{proof}
By Lemma \ref{1.5}, we have
\begin{align*}
e_Be_A(x_{\nu})
=\frac{1}{[A]_{\mathfrak{c}}^!}q_{g_2}^{-1}e_B(x_{\mu}T_{g_2}x_{\nu})
=\frac{1}{[A]_{\mathfrak{c}}^!}q_{g_2}^{-1}e_B(x_{\mu})T_{g_2}x_{\nu}
=\frac{1}{[A]_{\mathfrak{c}}^!}q_{g_1}^{-1}x_{\lambda}T_{g_1}T_{(\mathscr{D}_{\delta}\cap W_{\mu})g_2}x_{\nu}.
\end{align*}    
\end{proof}

For $w\in\mathscr{D}_{\delta}\cap W_{\mu}$, we write
\begin{equation*}
T_{g_1}T_{wg_2}=\sum_{\sigma\in\Delta(w)}c^{(w,\sigma)}T_{g_1\sigma wg_2},
\end{equation*}
where $c^{(w,\sigma)}\in\mathbb{A}=\mathbb{Z}[q^{\pm\frac{1}{2}},q_0^{\pm\frac{1}{2}},q_1^{\pm\frac{1}{2}}]$ and $\Delta(w)$ is a finite subset of $W$.
For $\sigma\in\Delta(w)$, we have
\begin{equation*}
T_{g_1\sigma wg_2}=T_{w_{\lambda}^{(\sigma)}}T_{y^{(w,\sigma)}}T_{w_{\nu}^{(\sigma)}}
\end{equation*}
if we write
$g_1\sigma wg_2=w_{\lambda}^{(\sigma)}y^{(w,\sigma)}w_{\nu}^{(\sigma)}$ for some $y^{(w,\sigma)}\in\mathscr{D}_{\lambda\nu}, w_{\lambda}^{(\sigma)}\in W_{\lambda}, w_{\nu}^{(\sigma)}\in W_{\nu}$.
We further denote
\begin{equation*}
A^{(w,\sigma)}=(a_{ij}^{(w,\sigma)})=\kappa(\lambda,y^{(w,\sigma)},\nu).
\end{equation*}

\begin{proposition}\label{3.3}
Keep $\delta=\delta(B)$. We have
\begin{equation*}
e_Be_A
=\sum_{\scalebox{0.7}{$\substack{w\in\mathscr{D}_{\delta}\cap W_{\mu}\\ \sigma\in\Delta(w)}$}}c^{(w,\sigma)}q_0^{L_0} q_1^{L_1} q^{L}\frac{[A^{(w,\sigma)}]_\mathfrak{c}^!}{[A]_\mathfrak{c}^!}e_{A^{(w,\sigma)}},\quad \mbox{where}
\end{equation*}
\begin{align*}
L_0&=-\ell_{\mathfrak{c}_0}(g_1\sigma wg_2)+\ell_{\mathfrak{c}_d}(w)+\ell_{\mathfrak{c}_d}(g_2)+\ell_{\mathfrak{c}_0}(y^{(w,\sigma)})-\ell_{\mathfrak{c}_d}(y^{(w,\sigma)}),\\
L_1&=-\ell_{\mathfrak{c}_d}(g_1\sigma wg_2)-\ell_{\mathfrak{c}_0}(w)-\ell_{\mathfrak{c}_0}(g_2)+\ell_{\mathfrak{c}_0}(y^{(w,\sigma)})+\ell_{\mathfrak{c}_d}(y^{(w,\sigma)}),\\
L&=-\ell_{\mathfrak{a}}(g_1\sigma wg_2)-\ell_{\mathfrak{a}}(g_1)-\ell_{\mathfrak{a}}(w)-\ell_{\mathfrak{a}}(g_2)+2\ell_{\mathfrak{a}}(y^{(w,\sigma)}).
\end{align*}
\end{proposition}

\begin{proof}
By Lemmas \ref{1.5} \& \ref{1.6}, we have
\begin{align*}
e_Be_A(x_\nu)
=&\frac{1}{[A]_\mathfrak{c}^!}q_{g_1}^{-1}q_{g_2}^{-1}\sum_{\scalebox{0.7}{$\substack{w\in\mathscr{D}_{\delta}\cap W_{\mu}\\ \sigma\in\Delta(w)}$}}c^{(w,\sigma)}q_{w}^{-1}x_{\lambda}T_{g_1\sigma wg_2}x_{\nu}\\
=&\sum_{\scalebox{0.7}{$\substack{w\in\mathscr{D}_{\delta}\cap W_{\mu}\\ \sigma\in\Delta(w)}$}}c^{(w,\sigma)}\frac{1}{[A]_\mathfrak{c}^!}q_{g_1}^{-1}q_{g_2}^{-1}q_{w}^{-1}q_0^{-\ell_{\mathfrak{c}_0}(w_{\lambda}^{(\sigma)})-\ell_{\mathfrak{c}_0}(w_{\nu}^{(\sigma)})}q_1^{-\ell_{\mathfrak{c}_d}(w_{\lambda}^{(\sigma)})-\ell_{\mathfrak{c}_d}(w_{\nu}^{(\sigma)})}\\
&\times q^{-\ell_{\mathfrak{a}}(w_{\lambda}^{(\sigma)})-\ell_{\mathfrak{a}}(w_{\nu}^{(\sigma)})}x_{\lambda}T_{y^{(w,\sigma)}}x_{\nu}\\
=&\sum_{\scalebox{0.7}{$\substack{w\in\mathscr{D}_{\delta}\cap W_{\mu}\\ \sigma\in\Delta(w)}$}}c^{(w,\sigma)}q_0^{-\ell_{\mathfrak{c}_0}(g_1\sigma wg_2)+\ell_{\mathfrak{c}_d}(w)+\ell_{\mathfrak{c}_d}(g_2)+\ell_{\mathfrak{c}_0}(y^{(w,\sigma)})-\ell_{\mathfrak{c}_d}(y^{(w,\sigma)})}\\
&\times q_1^{-\ell_{\mathfrak{c}_d}(g_1\sigma wg_2)-\ell_{\mathfrak{c}_0}(w)-\ell_{\mathfrak{c}_0}(g_2)+\ell_{\mathfrak{c}_0}(y^{(w,\sigma)})+\ell_{\mathfrak{c}_d}(y^{(w,\sigma)})}\\
&\times q^{-\ell_{\mathfrak{a}}(g_1\sigma wg_2)-\ell_{\mathfrak{a}}(g_1)-\ell_{\mathfrak{a}}(w)-\ell_{\mathfrak{a}}(g_2)+2\ell_{\mathfrak{a}}(y^{(w,\sigma)})}\frac{[A^{(w,\sigma)}]_\mathfrak{c}^!}{[A]_\mathfrak{c}^!}e_{A^{(w,\sigma)}}(x_\nu).
\end{align*} So the statement holds.   
\end{proof}

\subsection{Multiplication formulas for semisimple generators}
We assume that $B=(b_{ij})=\kappa(\lambda,g_1,\mu)$ is tridiagonal, i.e. $b_{ij}=0$ unless $|i-j|\leqslant1$. In this case, we can take
\begin{align*}
\delta=\delta(B)=(b'_{00},b_{10};b_{01},b_{11},b_{21};\ldots;b_{r-1,r},b_{r,r},b_{r+1,r};b_{r,r+1},b'_{r+1,r+1})\in\Lambda_{3r+2,d}.
\end{align*}
For any $1\leqslant i\leqslant r+1$, $w\in\mathscr{D}_{\delta}\cap W_{\mu}$ and $g_2\in\mathscr{D}_{\mu\nu}$,  we introduce the following subset of $W$:
\begin{align*}
K_w^{(i)}=\{\text{products of disjoint}~&\text{transpositions of the form}~(j,k)_{\mathfrak{c}}~|\\
&j\in R_{3i-2}^{\delta},~k\in R_{3i-1}^{\delta},~(wg_2)^{-1}(k)<(wg_2)^{-1}(j)\}.
\end{align*}
We then denote
\begin{equation*}
K_w:=\{\prod_{i=1}^{r+1}\sigma^{(i)}\big|~\sigma^{(i)}\in K_w^{(i)}\}.
\end{equation*}
For $w\in\mathscr{D}_{\delta}\cap W_{\mu}$ and $\sigma\in K_w$ we denote
\begin{equation*}
n(\sigma)=\text{the number of disjoint transpositions appeared in}~\sigma.
\end{equation*}
Finally, we set
\begin{align*}
h(w,\sigma&=|H(w,\sigma)|, \quad\mbox{where}
\\
H(w,\sigma)&=\bigcup_{i=1}^{r+1}\left\{(j,k)\in R_{3i-2}^{\delta}\times R_{3i-1}^{\delta}~\middle|~
\begin{aligned}
&(wg_2)^{-1}\sigma(j)>(wg_2)^{-1}(k),\\
&(wg_2)^{-1}(j)>(wg_2)^{-1}\sigma(k)
\end{aligned}\right\}.
\end{align*}

\begin{thm}
Let $B=\kappa(\lambda,g_1,\mu)$ be a tridiagonal mqatrix. For any $g_2\in\mathscr{D}_{\mu\nu}$ and $w\in\mathscr{D}_{\delta(B)}\cap W_{\mu}$, we have
\begin{equation*}
T_{g_1}T_{wg_2}=\sum_{\sigma\in K_w}(q^{-1}-q)^{n(\sigma)}T_{g_1\sigma wg_2}.
\end{equation*}
\end{thm}
\begin{proof}
We can write $g_1=\prod_{i=1}^{r+1}g_1^{(i)}$, where $g_1^{(i)}\in W$ is specified by the following recipe:
\begin{equation*}
g_1^{(i)}(x)=\left\{
\begin{aligned}
&x+b_{i-1,i},&&\text{if}~x\in R_{3i-2}^{\delta}\subset R_{i-1}^{\mu},\\
&x-b_{i,i-1},&&\text{if}~x\in R_{3i-1}^{\delta}\subset R_{i}^{\mu},\\
&x,&&\text{if}~x\in[1..d]\backslash(R_{3i-2}^{\delta}\cup R_{3i-1}^{\delta}).
\end{aligned}
\right.    
\end{equation*}
Note that $s_0$ and $s_d$ do not appear in $g_1$ since $0\in R_0^{\delta}$ and $d+1\in R_{3r+3}^{\delta}$. Hence, only the parameter $q$ appears in the terms of $T_{g_1}T_{wg_2}$ when expanded. We omit the following argument here since it is totally the same as those for \cite[Theorem 3.3]{FLLLW23}.  
\end{proof}

\begin{corollary}[{\cite[Corollary 3.5]{FLLLW23}}]
For any $g_1\in\mathscr{D}_{\lambda\mu}$, $g_2\in\mathscr{D}_{\mu\nu}$, $w\in\mathscr{D}_{\delta}\cap W_{\mu}$ and $\sigma\in K_w$, we have
\begin{equation*}
\ell_{\mathfrak{a}}(g_1)+\ell_{\mathfrak{a}}(w)+\ell_{\mathfrak{a}}(g_2)=\ell_{\mathfrak{a}}(g_1\sigma wg_2)+n(\sigma)+2h(w,\sigma).
\end{equation*}
\end{corollary}

For any $T=(t_{ij})\in\Theta_n$, set
\begin{equation*}
T_{\theta}=(t_{ij}+t_{-i,-j})\in\Theta_n.
\end{equation*}
For $A$ and $B=(b_{ij})\in\Xi_n$, we set
\begin{equation*}
\Theta_{B,A}=\left\{T\in\Theta_n|\ T_{\theta}\leqslant_e A,~\mathrm{row}_{\mathfrak{a}}(T)_i=b_{i-1,i}~\text{for all}~i\right\},
\end{equation*}
where $\leqslant_e$ is an entry-wise partial order on $\Theta_n$ defined by
\begin{equation*}
(a_{ij})\leqslant_e(b_{ij})\quad  \xLeftrightarrow{\hspace{0.5em}\text{def.}\hspace{0.5em}} \quad a_{ij}\leqslant b_{ij},~(\forall i,j).
\end{equation*}

For $w\in W$, we define a map
\begin{equation*}
\phi: \mathscr{D}_{\delta}\cap W_{\mu}\rightarrow\Theta_{B,A},
\quad 
\phi(w)_{ij}=|R_{3i-1}^{\delta}\cap wg_2R_j^{\nu}|.
\end{equation*}
For any $T=(t_{ij})\in\Theta_{B,A}$, let $w_{A,T}\in\phi^{-1}(T)$ be the minimal length element in $\phi^{-1}(T)$, which can be constructed as follows. For all $i$, $j$, we set
\begin{align*}
\mathcal{T}_{ij}^{-}=&\text{subset of}~(A^{\mathcal{P}})_{ij}~\text{consisting of the smallest}~t_{ij}~\text{elements},\\
\mathcal{T}_{ij}^{+}=&\text{subset of}~(A^{\mathcal{P}})_{ij}~\text{consisting of the largest}~t_{-i,-j}~\text{elements},
\end{align*}
where 
$$
(A^{\mathcal{P}})_{ij}=\left\{
\begin{aligned}
&[-a'_{00}..a'_{00}], && (i,j)=(0,0), \\ 
&[d+1-a'_{r+1,r+1}..d+1+a'_{r+1,r+1}], && (i,j)=(r+1,r+1),\\
& \big(\sum_{l=0}^{i-1}\mathrm{row}_{\mathfrak{c}}(A)_l..\sum_{l=0}^{i-1}\mathrm{row}_{\mathfrak{c}}(A)_l+\sum_{k\leqslant j}a_{ik}\big], && (i,j)\in I_{\mathfrak{a}}^+.
\end{aligned}
\right.
$$
Then $w_{A,T}=\prod_{i=0}^{r+1}w_{A,T}^{(i)}\in\mathscr{D}_{\delta}\cap W_{\mu}$, where $w_{A,T}^{(i)}$ is a permutation of $R_i^{\mu}$ determined by 
\begin{equation*}
w_{A,T}^{(i)}(x)\in\left\{
\begin{aligned}
&R_{3i-1}^{\delta},&&\text{if}~x\in\bigcup_{j}\mathcal{T}_{ij}^{-},\\
&R_{3i+1}^{\delta},&&\text{if}~x\in\bigcup_{j}\mathcal{T}_{ij}^{+},\\
&R_{3i}^{\delta},&&\text{otherwise}.\\
\end{aligned}
\right.
\end{equation*}

\begin{lemma}
For $T\in\Theta_{B,A}$, we have
\begin{align*}
\ell_{\mathfrak{c}_0}(w_{A,T})
=&\sum_{j>0}t_{0j},\quad 
\ell_{\mathfrak{c}_d}(w_{A,T})
=\sum_{j>r+1}t_{r+1,j},\\
\ell_{\mathfrak{a}}(w_{A,T})
=&\sum_{\scalebox{0.7}{$\substack{1\leqslant i\leqslant r\\ j\in\mathbb{Z}}$}}(t_{ij}\sum_{k<j}(A-T)_{ik}+t_{-i,-j}\sum_{k>j}(A-T_{\theta})_{ik})+\sum_{\scalebox{0.7}{$\substack{k<j\leqslant 0 \text{~or}\\ -k\geqslant j>0}$}}t_{0j}(A-T)_{0k}\\
&+\sum_{|k|<j}t_{0j}(A-T_{\theta}-1)_{0k}-\sum_{j>0}\binom{t_{0j}+1}{2}+\sum_{\scalebox{0.7}{$\substack{k<j\leqslant r+1 \text{~or}\\ -k\geqslant j>r+1}$}}t_{r+1,j}(A-T)_{r+1,k}\\
&+\sum_{\scalebox{0.7}{$\substack{j>r+1\\ |k|<j-r-1}$}}t_{r+1,j}(A-T_{\theta}-1)_{r+1,k}-\sum_{j>r+1}\binom{t_{r+1,j}+1}{2}.
\end{align*}
\end{lemma}

\begin{proof}
Since $w_{A,T}$ is the minimal length element in $\phi^{-1}(T)$, it shifts elements in $\bigcup_{j}\mathscr{T}_{ij}^{-}$ to the front of the elements $\bigcup_{j}(A^{\mathcal{P}})_{ij}\backslash\mathscr{T}_{ij}^{-}$, and shifts elements in $\bigcup_{j}\mathscr{T}_{ij}^{+}$ to the front of the elements $\bigcup_{j}(A^{\mathcal{P}})_{ij}\backslash\mathscr{T}_{ij}^{+}$. Note that $|\mathscr{T}_{ij}^-|=t_{ij}$ and $|\mathscr{T}_{ij}^+|=t_{-i,-j}$.

When a pair of corresponding elements in $\mathcal{T}_{0j}^-$ and $\mathcal{T}_{0,-j}^+$ are swapped, it means that a pair of numbers changes sign through the action of $w_{A,T}$. So $\ell_{\mathfrak{c}_0}(w_{A,T})
=\sum_{j>0}t_{0j}$. Similarly, we obtain $\ell_{\mathfrak{c}_d}(w_{A,T})
=\sum_{j>r+1}t_{r+1,j}$. The formula of $\ell(w_{A,T})$ is given in \cite[Lemma 4.2]{FLLLW23}. Then there comes $\ell_{\mathfrak{a}}(w_{A,T})=\ell(w_{A,T})-\ell_{\mathfrak{c}_0}(w_{A,T})-\ell_{\mathfrak{c}_d}(w_{A,T})$.
\end{proof}

\begin{lemma}
Let $T\in\Theta_{B,A}$, we have
\begin{align*}
&\sum_{w\in\phi^{-1}(T)}q_0^{-\ell_{\mathfrak{c}_0}(w)+\ell_{\mathfrak{c}_d}(w)}q_1^{-\ell_{\mathfrak{c}_0}(w)-\ell_{\mathfrak{c}_d}(w)}q^{-2\ell_{\mathfrak{a}}(w)} \\
=&q_0^{-\ell_{\mathfrak{c}_0}(w_{A,T})+\ell_{\mathfrak{c}_d}(w_{A,T})}q_1^{-\ell_{\mathfrak{c}_d}(w_{A,T})-\ell_{\mathfrak{c}_d}(w_{A,T})}q^{-2\ell_{\mathfrak{a}}(w_{A,T})}\frac{[A]_{\mathfrak{c}}^!}{[A-T_{\theta}]_\mathfrak{c}^![T]^!}.
\end{align*}
\end{lemma}

\begin{proof}
The permutation $w\in\phi^{-1}(T)$ is determined by which $t_{ij}$ and $t_{-i,-j}$ elements in $(A^{\mathcal{P}})_{ij}$ are moved to the left and right for all $(i,j)\in I^+$, respectively. So, we have
\begin{align*}
&\sum_{w\in\phi^{-1}(T)}q_0^{-\ell_{\mathfrak{c}_0}(w)+\ell_{\mathfrak{c}_d}(w)}q_1^{-\ell_{\mathfrak{c}_0}(w)-\ell_{\mathfrak{c}_d}(w)}q^{-2\ell_{\mathfrak{a}}(w)}\\
=&q_0^{-\ell_{\mathfrak{c}_0}(w_{A,T})+\ell_{\mathfrak{c}_d}(w_{A,T})}q_1^{-\ell_{\mathfrak{c}_d}(w_{A,T})-\ell_{\mathfrak{c}_d}(w_{A,T})}q^{-2\ell_{\mathfrak{a}}(w_{A,T})}\\
&\times\sum_{x+y=t_{00}}
\left[\begin{array}{c}
a_{00}^\prime\\
x
\end{array}\right]
\left[\begin{array}{c}
a_{00}^\prime-x\\
y
\end{array}\right]
q_0^{-x}q_1^{-x}(q^{-2})^{\frac{x(x-1)}{2}+x(a_{00}^\prime -t_{00})}\\
&\times
\sum_{x+y=t_{r+1,r+1}}
\left[\begin{array}{c}
a_{r+1,r+1}^\prime\\
x
\end{array}\right]
\left[\begin{array}{c}
a_{r+1,r+1}^\prime-x\\
y
\end{array}\right]
q_0^{x}q_1^{-x}(q^{-2})^{\frac{x(x-1)}{2}+x(a_{r+1,r+1}^\prime -t_{r+1,r+1})}\\
&\times\prod\limits_{(i,j)\in I_{\mathfrak{a}}^+}\frac{[a_{ij}]^!}{[t_{ij}]^![t_{-i,-j}]^![a_{ij}-(t_{\theta})_{ij}]^!}.
\end{align*}
Note that
\begin{align*}
&\sum_{x+y=t_{00}}
\left[\begin{array}{c}
a_{00}^\prime\\
x
\end{array}\right]
\left[\begin{array}{c}
a_{00}^\prime-x\\
y
\end{array}\right]
q_0^{-x}q_1^{-x}(q^{-2})^{\frac{x(x-1)}{2}+x(a_{00}^\prime -t_{00})}\\
&=
\left[\begin{array}{c}
a_{00}^\prime\\
t_{00}
\end{array}\right]
\sum_{x=0}^{t_{00}}
\left[\begin{array}{c}
t_{00}\\
x
\end{array}\right]
(q^{-1})^{x(x-1)}(q_0^{-1} q_1^{-1} q^{-2(a_{00}^\prime -t_{00})})^x\\
&\overset{(\Diamond)}{=}
\left[\begin{array}{c}
a_{00}^\prime\\
t_{00}
\end{array}\right]
\prod\limits_{i=1}^{t_{00}}(1+q_0^{-1}q_1^{-1}q^{-2(a_{00}^\prime -t_{00}+i-1)})=\frac{[a_{00}^\prime]_{\mathfrak{c}_0}^!}{[a_{00}^\prime -t_{00}]_{\mathfrak{c}_0}^![t_{00}]^!},
\end{align*}
where $(\Diamond)$ is due to the quantum binomial theorem 
\begin{equation*}
\sum_{x=0}^m
\left[\begin{array}{c}
m\\
x
\end{array}\right]
v^{x(x-1)}z^x=\prod\limits_{x=1}^{m-1}(1+v^{2i}z).
\end{equation*}
Similarly,
\begin{align*}
&\sum_{x+y=t_{r+1,r+1}}
\left[\begin{array}{c}
a_{r+1,r+1}^\prime\\
x
\end{array}\right]
\left[\begin{array}{c}
a_{r+1,r+1}^\prime-x\\
y
\end{array}\right]
q_0^{x}q_1^{-x}(q^{-2})^{\frac{x(x-1)}{2}+x(a_{r+1,r+1}^\prime -t_{r+1,r+1})}\\
=&\frac{[a_{r+1,r+1}^\prime]_{\mathfrak{c}_d}^!}{[a_{r+1,r+1}^\prime -t_{r+1,r+1}]_{\mathfrak{c}_d}^![t_{r+1,r+1}]^!}.    
\end{align*}
Therefore,
\begin{align*}
&\sum_{w\in\phi^{-1}(T)}q_0^{-\ell_{\mathfrak{c}_0}(w)+\ell_{\mathfrak{c}_d}(w)}q_1^{-\ell_{\mathfrak{c}_0}(w)-\ell_{\mathfrak{c}_d}(w)}q^{-2\ell_{\mathfrak{a}}(w)} \\
=&q_0^{-\ell_{\mathfrak{c}_0}(w_{A,T})+\ell_{\mathfrak{c}_d}(w_{A,T})}q_1^{-\ell_{\mathfrak{c}_d}(w_{A,T})-\ell_{\mathfrak{c}_d}(w_{A,T})}q^{-2\ell_{\mathfrak{a}}(w_{A,T})}\\
&\times\frac{[a_{00}^\prime]_{\mathfrak{c}_0}^!}{[a_{00}^\prime -t_{00}]_{\mathfrak{c}_0}^![t_{00}]^!}\frac{[a_{r+1,r+1}^\prime]_{\mathfrak{c}_d}^!}{[a_{r+1,r+1}^\prime -t_{r+1,r+1}]_{\mathfrak{c}_d}^![t_{r+1,r+1}]^!}\prod\limits_{(i,j)\in I_{\mathfrak{a}}^+}\frac{[a_{ij}]^!}{[t_{ij}]^![t_{-i,-j}]^![a_{ij}-(t_{\theta})_{ij}]^!}\\
=&q_0^{-\ell_{\mathfrak{c}_0}(w_{A,T})+\ell_{\mathfrak{c}_d}(w_{A,T})}q_1^{-\ell_{\mathfrak{c}_d}(w_{A,T})-\ell_{\mathfrak{c}_d}(w_{A,T})}q^{-2\ell_{\mathfrak{a}}(w_{A,T})}\frac{[A]_{\mathfrak{c}}^!}{[A-T_{\theta}]_\mathfrak{c}^![T]^!}.
\end{align*}
\end{proof}

For $T\in\Theta_{B,A}$, $w\in\phi^{-1}(T)$, $\sigma\in K_w$, we can define
\begin{equation*}
h(T,\sigma)=h(w,\sigma).
\end{equation*}
Let us fix the product $\sigma^{(i)}=\prod_{l=1}^{s_i}(j_l^{(i)},k_l^{(i)})_{\mathfrak{c}}$ by requiring
\begin{equation*}
j_1^{(i)}<j_2^{(i)}<\cdots<j_{s_i}^{(i)}.
\end{equation*}
We further define $s_{-i}=s_{i+1}$ for $0\leqslant i\leqslant r$ and
$j_l^{(-i)}=k_{s_{i+1}}^{(i+1)},~k_l^{(-i)}=j_{s_{i+1}-2l+1}^{(i+1)}$ for $0\leqslant i\leqslant r,~1\leqslant l\leqslant s_i$.

For $w\in\mathscr{D}_{\delta}\cap W_{\mu}$, we define a map
\begin{equation*}
\psi_w: K_w\rightarrow\Theta_n,\quad 
\psi_w(\sigma)_{ij}=|R_{3i-1}^{\delta}\cap\sigma(R_{3i-2}^{\delta})\cap wg_2R_j^{\nu}|.
\end{equation*}

For any $\mathbb{Z}\times\mathbb{Z}$ matrix $S=(s_{ij})$, we set 
\begin{equation*}
\widehat{S}=(\hat{s}_{ij}) ~ \mbox{with}~
\hat{s}_{ij}:=s_{i+1,j},
\quad\mbox{and}\quad S^{\dag}=(s_{ij}^{\dag}) ~\mbox{with}~
s_{ij}^{\dag}:=s_{1-i,-j}=\hat{s}_{-i,-j}.
\end{equation*}
For $T\in\Theta_{B,A}$, we set
\begin{equation*}
\Gamma_{T}=\left\{S\in\Theta_n~|~\ S\leqslant_e T,~\mathrm{row}_{\mathfrak{a}}(S)=\mathrm{row}_{\mathfrak{a}}(S^{\dag})\right\}.
\end{equation*}
For $T\in\Theta_{B,A}$ and $S\in\Gamma_{T}$, we set
\begin{equation*}
A^{(T-S)}=A-(T-S)_{\theta}+(\widehat{T-S})_{\theta}.\quad \mbox{Particularly,}\ A^{(T)}=A^{(T-0)}.
\end{equation*}

\begin{lemma}[{\cite[Lemma 4.6]{FLLLW23}}]
For $w\in\mathscr{D}_{\delta}\cap W_{\mu}$ and $\sigma\in K_w$, we have
\begin{equation*}
A^{(w,\sigma)}=A^{(T-S)},
\end{equation*}
where $T=\phi(w)$ and $S=\psi_w(\sigma)$.
\end{lemma}

We define an element
\begin{equation*}
\sigma_{w,S}=\prod_{i=1}^{r+1}\prod_{l=1}^{s_i}(j_l^{(i)},k_l^{(i)})_{\mathfrak{c}}\in\psi_w^{-1}(S)
\end{equation*}
satisfying Conditions (S1)-(S2) below:
\begin{itemize}
\item[(S1)] $k_1^{(i)}<k_2^{(i)}<\cdots<k_{s_i}^{(i)}$, $\forall i$;
\item[(S2)] $w^{-1}(\left\{k_l^{i}\right\}_l)\cap g_2R_j^{\nu}$ consists of the largest $s_{ij}$ elements in $w^{-1}R_{3i-1}^{\delta}\cap g_2R_j^{\nu}$, $\forall i$.
\end{itemize}
For $S\in\Gamma_{T}$, set
\begin{equation*}
\llbracket  S\rrbracket =\prod\limits_{i=1}^{r+1}\llbracket  S\rrbracket _i
\qquad\mbox{with}\quad
\llbracket  S\rrbracket _i:=\prod_{j\in\mathbb{Z}}
\left[\begin{array}{c}
\sum_{k\leqslant j}(S-S^{\dag})_{ik}\\
s_{i,j+1}^{\dag}
\end{array}\right]
[s_{i,j+1}^{\dag}]^!.
\end{equation*}

\begin{lemma}{\cite[Lemma 4.8, Lemma 4.10]{FLLLW23}}
Let $T\in\Theta_{B,A}$ and $S\in\Gamma_{T}$. For any $w\in \phi^{-1}(T)$, we have
\begin{equation*}
\sum_{\sigma\in\psi_w^{-1}(S)}q^{h(w,\sigma)}=q^{h(w,\sigma_{w,S})}
\left[\begin{array}{c}
T\\
S
\end{array}\right]
\llbracket  S\rrbracket,
\qquad 
n(\sigma_{w,s})=\sum_{i=1}^{r+1}\mathrm{row}_{\mathfrak{a}}(S)_i, \qquad\mbox{and}
\end{equation*}
\begin{align*}
h(T,\sigma_{w,s})
=&\sum_{i=1}^{r+1}\sum_{j=-\infty}^{\infty}s_{ij}(\sum_{k=-\infty}^{j}t_{ik}-\frac{s_{ij}+1}{2})\\
&+\sum_{i=1}^{r+1}\sum_{j=-\infty}^{\infty}(t_{1-i,-j}-s_{1-i,-j})(\sum_{k=-\infty}^{j-1}t_{ik}+\sum_{k=j}^{\infty}s_{ik}-\sum_{k=j+1}^{\infty}s_{1-i,-k}).
\end{align*}
\end{lemma}

We write $n(S)=n(\sigma_{w,s})$ and $h(T,S)=h(T,\sigma_{w,s})$ if $T\in\Theta$ and $S\in\Gamma_{T}$.
Denote $\ell(C)=\ell(g)$ if $C=\kappa(\lambda,g,\mu)$, $g\in\mathscr{D}_{\lambda\mu}$.

\begin{thm}\label{1.15}
Let $A$, $B\in\Xi_{n,d}$ with $B$ tridiagonal and $\mathrm{row}_{\mathfrak{c}}(A)=\mathrm{col}_{\mathfrak{c}}(B)$. Then
\begin{equation*}
e_{B}e_{A}
=\sum_{\scalebox{0.7}{$\substack{T\in\Theta_{B,A} \\ S\in \Gamma_{T}}$}}(q^{-2}-1)^{n(S)}q_0^{\alpha_0}q_1^{\alpha_1}q^{\alpha}\frac{[A^{(T-S)}]_\mathfrak{c}^!}{[A-T_{\theta}]_\mathfrak{c}^![S]^![T-S]^!}\llbracket  S\rrbracket e_{A^{(T-S)}},
\end{equation*}
where
\begin{align*}
\alpha_0&=-\ell_{\mathfrak{c}_0}(w_{A,T})+\ell_{\mathfrak{c}_d}(w_{A,T})-\ell_{\mathfrak{c}_0}(A)+\ell_{\mathfrak{c}_d}(A)+\ell_{\mathfrak{c}_0}(A^{(T-S)})-\ell_{\mathfrak{c}_d}(A^{(T-S)}),\\
\alpha_1&=-\ell_{\mathfrak{c}_0}(w_{A,T})-\ell_{\mathfrak{c}_d}(w_{A,T})-\ell_{\mathfrak{c}_0}(A)-\ell_{\mathfrak{c}_d}(A)+\ell_{\mathfrak{c}_0}(A^{(T-S)})+\ell_{\mathfrak{c}_d}(A^{(T-S)}),\\
\alpha&=2(-\ell_{\mathfrak{a}}(B)-\ell_{\mathfrak{a}} (w_{A,T})-\ell_{\mathfrak{a}}(A)+\ell_{\mathfrak{a}} (A^{(T-S)})+n(S)+h(S,T)).
\end{align*}
\end{thm}

\begin{proof}
It is a straightforward calculation that 
\begin{align*}
e_{B}e_{A}
=&\sum_{\scalebox{0.7}{$\substack{w \in \mathscr{D}_{\delta} \cap W_{\mu} \\ \sigma \in K_w}$}}(q^{-1}-q)^{n(\sigma)}q_0^{-\ell_{\mathfrak{c}_0}(g_1\sigma wg_2)+\ell_{\mathfrak{c}_d}(w)+\ell_{\mathfrak{c}_d}(g_2)+\ell_{\mathfrak{c}_0}(y^{(w,\sigma)})-\ell_{\mathfrak{c}_d}(y^{(w,\sigma)})}\\
&\times q_1^{-\ell_{\mathfrak{c}_d}(g_1\sigma wg_2)-\ell_{\mathfrak{c}_0}(w)-\ell_{\mathfrak{c}_0}(g_2)+\ell_{\mathfrak{c}_0}(y^{(w,\sigma)})+\ell_{\mathfrak{c}_d}(y^{(w,\sigma)})}\\
&\times q^{-\ell_{\mathfrak{a}}(g_1\sigma wg_2)-\ell_{\mathfrak{a}}(g_1)-\ell_{\mathfrak{a}}(w)-\ell_{\mathfrak{a}}(g_2)+2\ell_{\mathfrak{a}}(y^{(w,\sigma)})}\frac{[A^{(w,\sigma)}]_\mathfrak{c}^!}{[A]_\mathfrak{c}^!}e_{A^{(w,\sigma)}}\\
=&\sum_{\scalebox{0.7}{$\substack{w \in \mathscr{D}_{\delta} \cap W_{\mu} \\ \sigma \in K_w}$}}(q^{-1}-q)^{n(\sigma)}q_0^{-\ell_{\mathfrak{c}_0}(w)-\ell_{\mathfrak{c}_0}(g_2)+\ell_{\mathfrak{c}_d}(w)+\ell_{\mathfrak{c}_d}(g_2)+\ell_{\mathfrak{c}_0}(y^{(w,\sigma)})-\ell_{\mathfrak{c}_d}(y^{(w,\sigma)})}\\
&\times q_1^{-\ell_{\mathfrak{c}_d}(w)-\ell_{\mathfrak{c}_d}(g_2)-\ell_{\mathfrak{c}_0}(w)-\ell_{\mathfrak{c}_0}(g_2)+\ell_{\mathfrak{c}_0}(y^{(w,\sigma)})+\ell_{\mathfrak{c}_d}(y^{(w,\sigma)})}\\
&\times q^{-2\ell_{\mathfrak{a}}(g_1)-2\ell_{\mathfrak{a}}(w)-2\ell_{\mathfrak{a}}(g_2)+n(\sigma)+2h(w,\sigma)+2\ell_{\mathfrak{a}}(y^{(w,\sigma)})}\frac{[A^{(w,\sigma)}]_\mathfrak{c}^!}{[A]_\mathfrak{c}^!}e_{A^{(w,\sigma)}}\\
=&\sum_{\scalebox{0.7}{$\substack{T\in\Theta_{B,A} \\ S\in \Gamma_{T}}$}}(q^{-2}-1)^{n(S)}q_0^{-(\ell_{\mathfrak{c}_0}(w_{A,T})-\ell_{\mathfrak{c}_d}(w_{A,T})+\ell_{\mathfrak{c}_0}(A)-\ell_{\mathfrak{c}_d}(A)-\ell_{\mathfrak{c}_0}(A^{(T-S)})+\ell_{\mathfrak{c}_d}(A^{(T-S)}))}\\
&\times q_1^{-(\ell_{\mathfrak{c}_0}(w_{A,T})+\ell_{\mathfrak{c}_d}(w_{A,T})+\ell_{\mathfrak{c}_0}(A)+\ell_{\mathfrak{c}_d}(A)-\ell_{\mathfrak{c}_0}(A^{(T-S)})-\ell_{\mathfrak{c}_d}(A^{(T-S)}))}\\
&\times q^{-2(\ell_{\mathfrak{a}}(B)+\ell_{\mathfrak{a}} (w_{A,T})+\ell_{\mathfrak{a}}(A)-\ell_{\mathfrak{a}} (A^{(T-S)})-n(S)-h(S,T))}\\
&\times\frac{[A^{(T-S)}]_\mathfrak{c}^!}{[A-T_{\theta}]_\mathfrak{c}^![S]^![T-S]^!}\llbracket  S\rrbracket e_{A^{(T-S)}}.
\end{align*}
\end{proof}

The Following is a special case of these multiplication formulas for Chevalley generators $e_A$, where $A-RE^{i,i+1}_\theta$ is diagonal for some $i\in[-r..r]$ and $R\in\mathbb{N}$. Here $E^{i,i+1}_\theta$ is the $\mathbb{Z}\times \mathbb{Z}$ matrix with $(i+kn,i+1+kn)$-th and $(-i+kn,-i-1+kn)$-th ($k\in\mathbb{Z}$) entries are $1$ and otherwise $0$. 
\begin{thm}\label{special}
Let $A$, $B$, $C\in\Xi_{n,d}$, and $R\in \mathbb{N}$.
\item[(1)] If $\mathrm{row}_{\mathfrak{c}}(A)=\mathrm{col}_{\mathfrak{c}}(B)$ and $B-RE_{\theta}^{i,i+1}$ is diagonal and $i\in [1..r]$, then we have 
\begin{equation*}
e_{B}e_{A}
=\sum_{t}q^{-2\sum_{j>u}a_{ij}t_u} \prod\limits_{u\in\mathbb{Z}}
\left[\begin{array}{c}
a_{iu}+t_u\\
t_u
\end{array}\right]
e_{A_{ti}},
\end{equation*}    
where $t=(t_u)_{u\in\mathbb{Z}}\in\mathbb{N^Z}$ runs over the elements satisfying $\sum_{u\in\mathbb{Z}}t_u=R$ and
\begin{equation*}
\left\{
\begin{aligned}
&t_u\leqslant a_{i+1,u},&&\text{if}~i<r,\\
&t_u+t_{-u}\leqslant a_{i+1,u},&&\text{if}~i=r,\\
\end{aligned}
\right.
\end{equation*}
and 
\begin{equation*}
A_{ti}=A+\sum_{u\in\mathbb{Z}}t_u(E_{\theta}^{iu}-E_{\theta}^{i+1,u}). 
\end{equation*} 

\item[(2)] If $\mathrm{row}_{\mathfrak{c}}(A)=\mathrm{col}_{\mathfrak{c}}(B)$ and $B-RE_{\theta}^{0,1}$ is diagonal, then we have 
\begin{align*}
e_{B}e_{A}
=&\sum_{t}q_0^{-\sum_{u<0}t_{u}}q_1^{-\sum_{u<0}t_{u}}q^{-2\sum_{j>u}a_{0j}t_u-2\sum_{u<j<-u}t_ut_j-\sum_{u<0}t_u(t_u-3)}\\
&\times\frac{[a_{00}'+t_0]_{\mathfrak{c}_0}^!}{[a_{00}']_{\mathfrak{c}_0}^![t_0]^!}\prod\limits_{u>0}\frac{[a_{0u}+t_u+t_{-u}]^!}{[a_{0u}]^![t_u]^![t_{-u}]^!}
e_{A_{t0}},   
\end{align*}
where $t=(t_u)_{u\in\mathbb{Z}}\in\mathbb{N^Z}$ runs over the elements satisfying $\sum\limits_{u\in\mathbb{Z}}t_u=R$ and $t_u\leqslant a_{1u}$.

\item[(3)] If $\mathrm{row}_{\mathfrak{c}}(A)=\mathrm{col}_{\mathfrak{c}}(C)$ and $C-RE_{\theta}^{i+1,i}$ is diagonal and $i\in [0..r-1]$, then we have 
\begin{equation*}
e_{C}e_{A}
=\sum_{t}q^{-2\sum_{j<u}a_{i+1,j}t_u} \prod\limits_{u\in\mathbb{Z}}
\left[\begin{array}{c}
a_{i+1,u}+t_u\\
t_u
\end{array}\right]
e_{\widehat{A}_{ti}},
\end{equation*} 
where $t=(t_u)_{u\in\mathbb{Z}}\in\mathbb{N^Z}$ runs over the elements satisfying $\sum_{u\in\mathbb{Z}}t_u=R$ and
\begin{equation*}
\left\{
\begin{aligned}
&t_u\leqslant a_{iu},&&\text{if}~i>0,\\
&t_u+t_{-u}\leqslant a_{iu},&&\text{if}~i=0,\\
\end{aligned}
\right.
\end{equation*}
and 
\begin{equation*}
\widehat{A}_{ti}=A-\sum_{u\in\mathbb{Z}}t_u(E_{\theta}^{iu}-E_{\theta}^{i+1,u}). 
\end{equation*}  

\item[(4)]  If $\mathrm{row}_{\mathfrak{c}}(A)=\mathrm{col}_{\mathfrak{c}}(C)$ and $C-RE_{\theta}^{r+1,r}$ is diagonal, then we have 
\begin{align*}
e_{C}e_{A}
=&\sum_{t}q_0^{\sum_{u>r+1}t_{u}}q_1^{-\sum_{u>r+1}t_{u}}q^{-2\sum_{j<u}a_{r+1,j}t_u-2\sum_{n-u<j<u}t_ut_j-\sum_{u>r+1}t_u(t_u-3)}\\
&\times\frac{[a_{r+1,r+1}'+t_{r+1}]_{\mathfrak{c}_d}^!}{[a_{r+1,r+1}']_{\mathfrak{c}_d}^![t_{r+1}]^!}\prod\limits_{u>0}\frac{[a_{r+1,u}+t_u+t_{-u}]^!}{[a_{r+1,u}]^![t_u]^![t_{-u}]^!}
e_{\widehat{A}_{t,r}},
\end{align*} 
where $t=(t_u)_{u\in\mathbb{Z}}\in\mathbb{N^Z}$ runs over the elements satisfying $\sum_{u\in\mathbb{Z}}t_u=R$ and $t_u\leqslant a_{ru}$.
\end{thm}

\section{Bases of affine Schur algebras}

\subsection{Bar involution}
There is an $\mathbb{A}$-algebra involution $\bar{\quad}$: $\mathbb{H}\rightarrow\mathbb{H}$, sending $q_i\mapsto q_i^{-1}$, $T_w\mapsto T_{w^{-1}}^{-1}$, for all $w\in W$. Particularly,
\begin{align*}
\overline{T}_0=q_0q_1^{-1}T_i+q_0-q_1^{-1};\quad
\overline{T}_i=T_i+q-q^{-1}, i\in [1..d-1];\quad
\overline{T}_d=q_0q_1T_i+q_1-q_0.
\end{align*}

For $\lambda,\mu\in\Lambda$ and $g\in\mathscr{D}_{\lambda\mu}$, let $g_{\lambda\mu}^+$ be the longest element in the double coset $W_{\lambda}gW_{\mu}$. Write $w_{\circ}^{\mu}=\mathbbm{1}_{\mu\mu}^+$ as the longest element in the parabolic subgroup $W_{\mu}=W_{\mu}\mathbbm{1}W_{\mu}$. denote by $\leqslant$ the (strong) Bruhat order on $W$. 



%

It is shown in \cite{Lu03} that, for any weight function $\mathbf{L}:W\to\mathbb{N}$, there exists a bar-invariant basis $\{C_w^{\mathbf{L}}\}$ at the specialization $q=\boldsymbol{v}^{-\mathbf{L}(s_1)}$, $q_0=\boldsymbol{v}^{-\mathbf{L}(s_0)+\mathbf{L}(s_d)}$, $q_1=\boldsymbol{v}^{-\mathbf{L}(s_0)-\mathbf{L}(s_d)}$, given by
\begin{align*}
&C_w^{\bf{L}}=q_0^{\frac{\ell_{\mathfrak{c}_0}(w)-\ell_{\mathfrak{c}_d}(w)}{2}}q_1^{\frac{\ell_{\mathfrak{c}_0}(w)+\ell_{\mathfrak{c}_d}(w)}{2}}q^{\ell_{\mathfrak{a}}(w)}\sum_{y\leqslant w}P_{y,w}(\boldsymbol{v})q_y^{-1}T_y|_{*} ,    
\end{align*}
where $P_{y,w}(\boldsymbol{v})$ is an analogue of the Kazhdan-Lusztig polynomial and $|_*$ means the condition $(q,q_0,q_1)=(\boldsymbol{v}^{-\mathbf{L}(s_1)},\boldsymbol{v}^{-\mathbf{L}(s_0)+\mathbf{L}(s_d)},\boldsymbol{v}^{-\mathbf{L}(s_0)-\mathbf{L}(s_d)})$. For $\lambda,\mu\in\Lambda$, let $\mathbb{H}_{\lambda\mu}$ be the $\mathbb{Z}[q^{\pm1},q_0^{\pm1},q_1^{\pm1}]$-submodule of $\mathbb{H}$ with a basis $\{T_{\lambda\mu}^g\}_{g\in\mathscr{D}_{\lambda\mu}}$. It follows from \cite[Lemma 2.10]{CNR72} and Lemma~\ref{2.5} that $\mathbb{H}_{\lambda\mu}$ can be characterized as
\begin{equation*}
\mathbb{H}_{\lambda\mu}=\left\{h\in\mathbb{H}~\middle|~
\begin{aligned}
&T_wh=q_0^{-\ell_{\mathfrak{c}_0}(w)}q_1^{-\ell_{\mathfrak{c}_d}(w)}q^{-\ell_{\mathfrak{a}}(w)}h,~\forall w\in W_{\lambda};\\
&hT_{w'}=q_0^{-\ell_{\mathfrak{c}_0}(w')}q_1^{-\ell_{\mathfrak{c}_d}(w')}q^{-\ell_{\mathfrak{a}}(w')}h,~\forall w'\in W_{\mu}.
\end{aligned}\right\}.
\end{equation*}

In the following, we show that the bar involution is closed on $\mathbb{H}_{\lambda\mu}$ although it lacks a bar-invariant basis.

\begin{proposition}\label{bar}
Let $A=\kappa(\lambda,g,\mu)\in\Xi_{n,d}$. Then $\overline{T_{\lambda\mu}^g}\in\mathbb{H}_{\lambda\mu}$. Precisely,
\begin{equation*}
\overline{T_{\lambda\mu}^g}\in q_0^{\ell_{\mathfrak{c}_0}(g_{\lambda\mu}^+)-\ell_{\mathfrak{c}_d}(g_{\lambda\mu}^+)}q_1^{\ell_{\mathfrak{c}_0}(g_{\lambda\mu}^+)+\ell_{\mathfrak{c}_d}(g_{\lambda\mu}^+)}q^{2\ell_{\mathfrak{a}}(g_{\lambda\mu}^+)}T_{\lambda\mu}^g+\sum_{\scalebox{0.7}{$\substack{y\in\mathscr{D}_{\lambda\mu}\\ y<g}$}}\mathbb{Z}[q^{\pm1},q_0^{\pm1},q_1^{\pm1}]T_{\lambda\mu}^y.
\end{equation*}
Moreover, $q_0^{\frac{\ell_{\mathfrak{c}_0}(w_{\circ}^{\mu})-\ell_{\mathfrak{c}_d}(w_{\circ}^{\mu})}{2}}q_1^{\frac{\ell_{\mathfrak{c}_0}(w_{\circ}^{\mu})+\ell_{\mathfrak{c}_d}(w_{\circ}^{\mu})}{2}}q^{\ell_{\mathfrak{a}}(w_{\circ}^{\mu})}x_{\mu}$ is bar-invariant.
\end{proposition}

\begin{proof}
First, we show that $\overline{x_{\nu}}\in\mathbb{A}x_{\nu}$ for all $\nu\in\Lambda_{r,d}$ via bar-invariant basis $C_w^{\bf{L}}$. Let $\mathbb{H}_{\lambda\mu}^{\bf{L}}$ be the specialization of $\mathbb{H}_{\lambda\mu}$ in $q=\boldsymbol{v}^{-\mathbf{L}(s_1)},q_0=\boldsymbol{v}^{-\mathbf{L}(s_0)+\mathbf{L}(s_d)},q_1=\boldsymbol{v}^{-\mathbf{L}(s_0)-\mathbf{L}(s_d)}$. A direct computation shows that $C_{w_{\circ}^{\nu}}^{\bf{L}}\in\mathbb{H}_{\nu\nu}^{\bf{L}}$, and hence
\begin{align*}
C_{w_{\circ}^{\nu}}^{\bf{L}}=
&q_0^{\frac{\hat\ell_{\mathfrak{c}_0}(w_{\circ}^{\nu})-\hat\ell_{\mathfrak{c}_d}(w_{\circ}^{\nu})}{2}}q_1^{\frac{\hat\ell_{\mathfrak{c}_0}(w_{\circ}^{\nu})+\hat\ell_{\mathfrak{c}_d}(w_{\circ}^{\nu})}{2}}q^{\hat\ell_{\mathfrak{a}}(w_{\circ}^{\nu})}\sum_{y\leqslant w_{\circ}^{\mu}}P_{y,w_{\circ}^{\nu}}q_y^{-1}T_y\bigg|_*\\
\in&\sum_{g\in\mathscr{D}_{\nu\nu}}\mathbb{Z}(\boldsymbol{v}^{\pm\mathbf{L}(s_1)},\boldsymbol{v}^{\pm\mathbf{L}(s_0)\mp\mathbf{L}(s_d)},\boldsymbol{v}^{\pm\mathbf{L}(s_0)\pm\mathbf{L}(s_d)})T_{\nu\nu}^g\bigg|_*.
\end{align*}
Upon comparing coefficients, we obtain
\begin{equation*}
C_{w_{\circ}^{\nu}}^{\bf{L}}=q_0^{\frac{\hat\ell_{\mathfrak{c}_0}(w_{\circ}^{\nu})-\hat\ell_{\mathfrak{c}_d}(w_{\circ}^{\nu})}{2}}q_1^{\frac{\hat\ell_{\mathfrak{c}_0}(w_{\circ}^{\nu})+\hat\ell_{\mathfrak{c}_d}(w_{\circ}^{\nu})}{2}}q^{\hat\ell_{\mathfrak{a}}(w_{\circ}^{\nu})}T_{\nu\nu}^{\mathbbm{1}}\bigg|_*.   
\end{equation*}
Note that $x_{\nu}=T_{\nu\nu}^{\mathbbm{1}}$. Hence, for any weight function $\bf{L}$, we have
\begin{equation*}
(\overline{x_{\nu}}-q_0^{\ell_{\mathfrak{c}_0}(w_{\circ}^{\nu})-\ell_{\mathfrak{c}_d}(w_{\circ}^{\nu})}q_1^{\ell_{\mathfrak{c}_0}(w_{\circ}^{\nu})+\ell_{\mathfrak{c}_d}(w_{\circ}^{\nu})}q^{\ell_{\mathfrak{a}}(w_{\circ}^{\nu})}x_{\nu})|_*=0.    
\end{equation*}
Therefore, $\overline{x_{\nu}}=q_0^{\ell_{\mathfrak{c}_0}(w_{\circ}^{\nu})-\ell_{\mathfrak{c}_d}(w_{\circ}^{\nu})}q_1^{\ell_{\mathfrak{c}_0}(w_{\circ}^{\nu})+\ell_{\mathfrak{c}_d}(w_{\circ}^{\nu})}q^{\ell_{\mathfrak{a}}(w_{\circ}^{\nu})}x_{\nu}$. By Lemma \ref{1.5}, we claim that $$\overline{T_{\lambda\mu}^g}\in\sum_{z\leqslant g}\mathbb{Z}[q^{\pm1},q_0^{\pm1},q_1^{\pm1}]T_{\lambda\mu}^z\subseteq\mathbb{H}_{\lambda\mu}.$$ 

For each $A\in(a_{ij})\in\Xi_{n,d}$, we can obtain a terms of matrix $A^{X}$ (with $X$ is a tuple belongs to $\{((i_p,j_p))_{p=1}^q~|~q\in\mathbb{N}^*,i_p,j_p\in\mathbb{Z}\}$ ), which are lower than $A$ in the Bruhat order by the following algorithm. Furthermore, denote $0$-tuple by $()$ and $A^{()}=A$. We introduce an order for pairs such that $(i,j)\succ(k,l)$ if and only if $l<j$ or `$l=j, k>i$'.
\item[(1)]
Initialization:
Set $X:=()$, $s:=\min\{x~|~a_{x,r+1}>0\}$ and $t:=r+1$.

\item[(2)] 
Process all pairs $(i,j)$ such that $i\in[s..r]$ and $j\in[t..i+1]$ in lexicographic order as follows:

If $a_{ij}=0$, output $A^{(X,(i,j))}=0$;

If $a_{ij}>0$, we choose $(k,l)=\max\{(k,l)~|~(i,j-1)\succ(k,l)\neq(-i,0),a_{kl}>0\}$ and output $A^{(X,(i,j))}=A^{X}-E^{ij}_{\theta}+E^{il}_{\theta}-E^{kl}_{\theta}+E^{kj}_{\theta}$.

\item[(3)]
If all $A^{(X,(i,j))}=0$, the algorithm terminates. Otherwise, set $X:=(X,(i,j))$, $Y:=\emptyset$, $s:=i$ and $t:=j$ respectively; then go to Step (2).

We can see that the coefficient change from $A^{X}$ to $A^{(X,(i,j))}$ in Step (2) belongs to $\mathbb{Z}[q^{\pm1},q_0^{\pm1},q_1^{\pm1}]$. 
It is sufficient to consider a $2\times2$-matrix and we obmit the subscript $\mathfrak{c}$ in the following. 
\begin{equation*}
A=\begin{pmatrix}
a & b \\
c & d
\end{pmatrix}
\to
A'=\begin{pmatrix}
a+x & b-x \\
c-x & d+x
\end{pmatrix},
~a,b,c,d,x\in\mathbb{N},~x\leqslant\mathrm{min}\{b,c\}.
\end{equation*}

It is obvious to see $g(A)>g(A')$. We denote the coefficient of $T_{g(A')}$ by $r(A')$. When we shift $x$ elements from the position $b$ to the position $a$ and from the position $c$ to the position $d$, which means
\begin{equation*}
\left[\begin{array}{c}
b\\
x
\end{array}\right]
\left[\begin{array}{c}
c\\
x
\end{array}\right]
[x]^!\bigg|r(A'),
\end{equation*}
we have
\begin{align*}
&\frac{[A']^!}{[A]^!}\cdot\frac{[b]^!}{[b-x]![x]^!}\cdot\frac{[c]^!}{[c-x]![x]^!}\cdot[x]^!\\
=&\frac{[a+x]^![b-x]^![c-x]^![d+x]^!}{[a]^![b]^![c]^![d]^!}\cdot\frac{[b]^!}{[b-x]![x]^!}\cdot\frac{[c]^!}{[c-x]![x]^!}\cdot[x]^!\\
=&\frac{[a+x]^![d+x]^!}{[a]^![d]^![x]^!}\in\mathbb{Z}[q^{\pm1},q_0^{\pm1},q_1^{\pm1}].
\end{align*}
Therefore, $r(A')\in\mathbb{Z}[q^{\pm1},q_0^{\pm1},q_1^{\pm1}]$.

The leading coefficient is obtained by a direct computation.
\end{proof}

Thanks to the above proposition, the bar involution $\bar{\quad}$ on $\mathbb{S}_{n,d}^{\mathfrak{c}}$ can be defined as follows. For each $f\in\mathrm{Hom}_{\mathbb{H}}(x_{\mu}\mathbb{H},x_{\lambda}\mathbb{H})$, let $\overline{f}\in\mathrm{Hom}_{\mathbb{H}}(x_{\mu}\mathbb{H},x_{\lambda}\mathbb{H})$ be the $\mathbb{H}$-linear map that sends $x_{\mu}$ to $\overline{f(\overline{x_{\mu}})}$.

\subsection{Standard basis}

The following proposition is obtained directly from Lemma~\ref{length}.
\begin{proposition}
For $A\in\Xi_n$, we have  
\begin{align*}
\ell{(A)}&=\frac{1}{2}\bigg(\sum_{(i,j)\in I^+}\bigg(\sum_{\scalebox{0.7}{$\substack{x<i\\ y>j}$}}+\sum_{\scalebox{0.7}{$\substack{x>i\\ y<j}$}}\bigg)a_{ij}'a_{xy}\bigg),\qquad
\ell_{\mathfrak{c}_0}{(A)}=\frac{1}{2}\bigg(\sum_{\scalebox{0.7}{$\substack{x<0\\ y>0}$}}+\sum_{\scalebox{0.7}{$\substack{x>0\\ y<0}$}}\bigg)a_{xy},\\
\ell_{\mathfrak{c}_d}{(A)}&=\frac{1}{2}\bigg(\sum_{\scalebox{0.7}{$\substack{x<r+1\\ y>r+1}$}}+\sum_{\scalebox{0.7}{$\substack{x>r+1\\ y<r+1}$}}\bigg)a_{xy},\qquad
\ell_{\mathfrak{a}}{(A)}=\frac{1}{2}\bigg(\sum_{(i,j)\in I^+}\bigg(\sum_{\scalebox{0.7}{$\substack{x<i\\ y>j}$}}+\sum_{\scalebox{0.7}{$\substack{x>i\\ y<j}$}}\bigg)a_{ij}''a_{xy}\bigg),
\end{align*}
where $a_{ii}''=a_{ii}'-1=\frac{1}{2}(a_{ii}-3)$ if $i=0,r+1$ and $a_{ij}''=a_{ij}$ if $i\in I_{\mathfrak{a}}^+$.
\end{proposition}

For $A\in\Xi_n$, we define
\begin{align*}
\hat\ell{(A)}&=\frac{1}{2}\bigg(\sum_{(i,j)\in I^+}\bigg(\sum_{\scalebox{0.7}{$\substack{x\leqslant i\\ y>j}$}}+\sum_{\scalebox{0.7}{$\substack{x\geqslant i\\ y<j}$}}\bigg)a_{ij}'a_{xy}\bigg),\qquad
\hat\ell_{\mathfrak{c}_0}{(A)}=\frac{1}{2}\bigg(\sum_{\scalebox{0.7}{$\substack{x\leqslant 0\\ y>0}$}}+\sum_{\scalebox{0.7}{$\substack{x\geqslant 0\\ y<0}$}}\bigg)a_{xy},\\
\hat\ell_{\mathfrak{c}_d}{(A)}&=\frac{1}{2}\bigg(\sum_{\scalebox{0.7}{$\substack{x\leqslant r+1\\ y>r+1}$}}+\sum_{\scalebox{0.7}{$\substack{x\geqslant r+1\\ y<r+1}$}}\bigg)a_{xy},\qquad
\hat\ell_{\mathfrak{a}}{(A)}=\frac{1}{2}\bigg(\sum_{(i,j)\in I^+}\bigg(\sum_{\scalebox{0.7}{$\substack{x\leqslant i\\ y>j}$}}+\sum_{\scalebox{0.7}{$\substack{x\geqslant i\\ y<j}$}}\bigg)a_{ij}''a_{xy}\bigg).
\end{align*}

Set
\begin{equation*}
[A]=q_0^{\frac{\hat\ell_{\mathfrak{c}_0}(A)-\hat\ell_{\mathfrak{c}_d}(A)}{2}}q_1^{\frac{\hat\ell_{\mathfrak{c}_0}(A)+\hat\ell_{\mathfrak{c}_d}(A)}{2}}q^{\hat\ell_{\mathfrak{a}}(A)}e_A.   
\end{equation*}
The set $\{[A]~|~A\in\Xi_{n,d}\}$ forms an $\mathbb{A}$-basis of $\mathbb{S}^{\mathfrak{c}}_{n,d}$, which we call the standard basis.

For $A\in\Xi_{n}$, we denote 
\begin{equation*}
\sigma_{ij}(A)=\sum_{x\leqslant i,y\geqslant j}a_{xy}.    
\end{equation*}
Now we define a partial order $\leqslant_{\mathrm{alg}}$ on $\Xi_n$ by letting, for $A,B\in\Xi_n$,
\begin{equation*}
A\leqslant_{\mathrm{alg}}B\Leftrightarrow\mathrm{row}_{\mathfrak{c}}(A)=\mathrm{row}_{\mathfrak{c}}(B),~\mathrm{col}_{\mathfrak{c}}(A)=\mathrm{col}_{\mathfrak{c}}(B),~\text{and}~\sigma_{ij}(A)\leqslant\sigma_{ij}(B),~\forall i<j. 
\end{equation*}
We denote $A<_{\mathrm{alg}}B$ if $A\leqslant_{\mathrm{alg}}B$ and $A\neq B$.

\begin{proposition}\label{invariant}
Let $A=\kappa(\lambda,g,\mu)\in\Xi_{n,d}$. Then we have $\overline{[A]}\in[A]+\sum_{B<_{\mathrm{alg}}A}\mathbb{A}[B]$.    
\end{proposition}
\begin{proof}
It has been shown in \cite[Proposition 5.3]{FLLLW23} that
\begin{align*}
\hat\ell_{\mathfrak{c}_0}{(A)}=\ell_{\mathfrak{c}_0}{(g_{\lambda\mu}^+)}-\ell_{\mathfrak{c}_0}{(w_{\circ}^{\mu})},\quad
\hat\ell_{\mathfrak{c}_d}{(A)}=\ell_{\mathfrak{c}_d}{(g_{\lambda\mu}^+)}-\ell_{\mathfrak{c}_d}{(w_{\circ}^{\mu})},\quad
\hat\ell_{\mathfrak{a}}{(A)}=\ell_{\mathfrak{a}}{(g_{\lambda\mu}^+)}-\ell_{\mathfrak{a}}{(w_{\circ}^{\mu})}.        
\end{align*}
Hence, 
\begin{align*}
&[A](q_0^{\frac{\hat\ell_{\mathfrak{c}_0}(w_{\circ}^{\mu})-\hat\ell_{\mathfrak{c}_d}(w_{\circ}^{\mu})}{2}}q_1^{\frac{\hat\ell_{\mathfrak{c}_0}(w_{\circ}^{\mu})+\hat\ell_{\mathfrak{c}_d}(w_{\circ}^{\mu})}{2}}q^{\hat\ell_{\mathfrak{a}}(w_{\circ}^{\mu})}x_{\mu})\\
=&q_0^{\frac{\hat\ell_{\mathfrak{c}_0}(g_{\lambda\mu}^+)-\hat\ell_{\mathfrak{c}_d}(g_{\lambda\mu}^+)}{2}}q_1^{\frac{\hat\ell_{\mathfrak{c}_0}(g_{\lambda\mu}^+)+\hat\ell_{\mathfrak{c}_d}(g_{\lambda\mu}^+)}{2}}q^{\hat\ell_{\mathfrak{a}}(g_{\lambda\mu}^+)}T_{\lambda\mu}^g.    
\end{align*}
Thus, by Proposition~\ref{bar}, the map $\overline{[A]}$ is determined by 
\begin{align*}
&\overline{[A]}(q_0^{\frac{\hat\ell_{\mathfrak{c}_0}(w_{\circ}^{\mu})-\hat\ell_{\mathfrak{c}_d}(w_{\circ}^{\mu})}{2}}q_1^{\frac{\hat\ell_{\mathfrak{c}_0}(w_{\circ}^{\mu})+\hat\ell_{\mathfrak{c}_d}(w_{\circ}^{\mu})}{2}}q^{\hat\ell_{\mathfrak{a}}(w_{\circ}^{\mu})}x_{\mu})\\
=&q_0^{-\frac{\hat\ell_{\mathfrak{c}_0}(g_{\lambda\mu}^+)-\hat\ell_{\mathfrak{c}_d}(g_{\lambda\mu}^+)}{2}}q_1^{-\frac{\hat\ell_{\mathfrak{c}_0}(g_{\lambda\mu}^+)+\hat\ell_{\mathfrak{c}_d}(g_{\lambda\mu}^+)}{2}}q^{-\hat\ell_{\mathfrak{a}}(g_{\lambda\mu}^+)}\overline{T_{\lambda\mu}^g}\\
\in&\ q_0^{\frac{\hat\ell_{\mathfrak{c}_0}(g_{\lambda\mu}^+)-\hat\ell_{\mathfrak{c}_d}(g_{\lambda\mu}^+)}{2}}q_1^{\frac{\hat\ell_{\mathfrak{c}_0}(g_{\lambda\mu}^+)+\hat\ell_{\mathfrak{c}_d}(g_{\lambda\mu}^+)}{2}}q^{\hat\ell_{\mathfrak{a}}(g_{\lambda\mu}^+)}T_{\lambda\mu}^g+\sum_{y<g}\mathbb{A}T_{\lambda\mu}^y.
\end{align*}
We note that $[\kappa(\lambda,y,\mu)](x_{\mu})\in\mathbb{A}T_{\lambda\mu}^y$. An induction on $\ell(g)$ shows that 
\begin{equation*}
\overline{[A]}\in[A]+\sum_{\scalebox{0.7}{$\substack{y\in\mathscr{D}_{\lambda\mu} \\ y<g}$}}\mathbb{A}[\kappa(\lambda,y,\mu)].    
\end{equation*}
It was shown in \cite[Lemma 5.4]{FLLLW23} that $\kappa(\lambda,y,\mu)<_{\mathrm{alg}}A$ if $y<g$. We conclude the statement.
\end{proof}

For $A$, $B\in\Xi_{n,d}$ with $B$ tridiagonal, $T\in\Theta_{B,A}$ and $S\in\Gamma_T$, let $\gamma_0(A,S,T)$, $\gamma_1(A,S,T)$ and $\gamma(A,S,T)$ be the integers such that
\begin{equation*}
\overline{\llbracket A;S;T\rrbracket}=q_0^{\gamma_0(A,S,T)}q_1^{\gamma_1(A,S,T)}q^{\gamma(A,S,T)}\llbracket A;S;T\rrbracket.      
\end{equation*}
We provide their explicit expressions:
\begin{align*}
\gamma_0(A,S,T)=&\big(s_{00}+\widehat{(T-S)}_{00}\big)-\big(s_{r+1,r+1}+\widehat{(T-S)}_{r+1,r+1}\big),\\
\gamma_1(A,S,T)=&\big(s_{00}+\widehat{(T-S)}_{00}\big)+\big(s_{r+1,r+1}+\widehat{(T-S)}_{r+1,r+1}\big),\\
\gamma(A,S,T)=&\sum_{(i,j)\in I_{\mathfrak{a}}^+}\big(s_{\theta,ij}+\widehat{(T-S)}_{\theta,ij}\big)\big(s_{\theta,ij}+\widehat{(T-S)}_{\theta,ij}+2a_{ij}-2t_{\theta,ij}-1\big)\\
&+\sum_{k\in\{0,r+1\}}\big(a_{kk}-2-t_{kk}-(T-S)_{kk}+\widehat{(T-S)}_{kk}\big)\\
&-2\sum_{i=1}^n\sum_{j\in\mathbb{Z}}\bigg(\binom{(T-S)}{2}+\binom{s_{ij}}{2}\bigg)\\
&-2\sum_{i=1}^{r+1}\sum_{j\in\mathbb{Z}}s_{i,j+1}^{\dag}\big(s_{i,j+1}^{\dag}-\sum_{k\leqslant j}(S-S^{\dag})_{ik}\big)-2\binom{s_{i,j+1}^{\dag}}{2},
\end{align*}
which are derived by a direct computation that
\begin{align*}
\overline{[A]^!}&=q^{2\sum_{i=1}^n\sum_{j\in\mathbb{Z}}\binom{a_{ij}}{2}}[A]^!,\\
\overline{[A]_{\mathfrak{c}}^!}&=q_0^{a_{00}'-a_{r+1,r+1}'}q_1^{a_{00}'+a_{r+1,r+1}'}q^{4\binom{a_{00}'}{2}+4\binom{a_{r+1,r+1}'}{2}+\sum_{(i,j)\in I_{\mathfrak{a}}^+}2\binom{a_{ij}'}{2}}[A]_{\mathfrak{c}}^!,\\
\overline{\llbracket S\rrbracket}&=q^{2\sum_{i=1}^{r+1}\sum_{j\in\mathbb{Z}}s_{i,j+1}^{\dag}(s_{i,j+1}^{\dag}-\sum_{k\leqslant j}(S-S^{\dag})_{ik})-2\binom{s_{i,j+1}^{\dag}}{2}}\llbracket S\rrbracket.
\end{align*}

The multiplication formula for $\mathbb{S}_{n,d}^{\mathfrak{c}}$ in Theorem \ref{1.15} can be reformulated in terms of the standard basis.

\begin{thm}\label{standard}
Let $A$, $B\in\Xi_{n,d}$ with $B$ tridiagonal and $\mathrm{row}_{\mathfrak{c}}(A)=\mathrm{col}_{\mathfrak{c}}(B)$. Then
\begin{equation*}
[B][A]
=\sum_{\scalebox{0.7}{$\substack{T\in\Theta_{B,A} \\ S\in \Gamma_{T}}$}}(q^{-2}-1)^{n(S)}q_0^{\beta_0(A,S,T)}q_1^{\beta_1(A,S,T)}q^{\beta(A,S,T)}\overline{\llbracket  A;S;T\rrbracket} [A^{(T-S)}],\quad \mbox{where}
\end{equation*}
\begin{align*}
\beta_0(A,S,T)=&\alpha_0(A,S,T)-\gamma_0(A,S,T)+\frac{1}{2}\big(\hat\ell_{\mathfrak{c}_0}(A)+\hat\ell_{\mathfrak{c}_0}(B)-\hat\ell_{\mathfrak{c}_0}(A^{(T-S)})\big)\\
&-\frac{1}{2}\big(\hat\ell_{\mathfrak{c}_d}(A)+\hat\ell_{\mathfrak{c}_d}(B)-\hat\ell_{\mathfrak{c}_d}(A^{(T-S)})\big),\\
\beta_1(A,S,T)=&\alpha_1(A,S,T)-\gamma_1(A,S,T)+\frac{1}{2}\big(\hat\ell_{\mathfrak{c}_0}(A)+\hat\ell_{\mathfrak{c}_0}(B)-\hat\ell_{\mathfrak{c}_0}(A^{(T-S)})\big)\\&+\frac{1}{2}\big(\hat\ell_{\mathfrak{c}_d}(A)+\hat\ell_{\mathfrak{c}_d}(B)-\hat\ell_{\mathfrak{c}_d}(A^{(T-S)})\big),\\
\beta(A,S,T)=&\alpha(A,S,T)-\gamma(A,S,T)+\hat\ell_{\mathfrak{a}}(A)+\hat\ell_{\mathfrak{a}}(B)-\hat\ell_{\mathfrak{a}}(A^{(T-S)}).
\end{align*}
\end{thm}

We rewrite the multiplication formulas for Chevalley generators in Theorem \ref{special} in terms of the standard basis in the following.
\begin{thm}\label{standard formula}
Let $A$, $B$, $C\in\Xi_{n,d}$, and $R\in \mathbb{N}$.
\item[(1)] If $\mathrm{row}_{\mathfrak{c}}(A)=\mathrm{col}_{\mathfrak{c}}(B)$ and $B-RE_{\theta}^{i,i+1}$ is diagonal and $i\in [1..r]$, then we have 
\begin{align*}
[B][A]
=&\sum_{t}(q_0q_1)^{\frac{1}{2}\delta_{ir}\sum_{u<r+1}t_u}q^{\Delta_1(t)}\prod\limits_{u\in\mathbb{Z}}
\overline{\left[\begin{array}{c}
a_{iu}+t_u\\
t_u
\end{array}\right]}
[A_{t,i}],
\end{align*}    
where 
\begin{equation*}
\Delta_1(t)=-\sum_{j\geqslant u}t_ua_{ij}+\sum_{j>u}t_u(a_{i+1,j}-t_j)-\delta_{ir}(\sum_{u<j<n-u}t_jt_u+\sum_{u<r+1}\frac{1}{2}t_u(t_u+3)).    
\end{equation*}

\item[(2)] If $\mathrm{row}_{\mathfrak{c}}(A)=\mathrm{col}_{\mathfrak{c}}(B)$ and $B-RE_{\theta}^{0,1}$ is diagonal, then we have 
\begin{align*}
[B][A]
=\sum_{t}(q_0q_1)^{-\frac{1}{2}\sum_{u\leqslant0}t_{u}}q^{\Delta_2(t)}\overline{\bigg(\frac{[a_{00}'+t_0]_{\mathfrak{c}_0}^!}{[a_{00}']_{\mathfrak{c}_0}^![t_0]^!}\prod_{u>0}\frac{[a_{0u}+t_u+t_{-u}]^!}{[a_{0u}]^![t_u]^![t_{-u}]^!}\bigg)}
[A_{t,0}], 
\end{align*}
where 
\begin{equation*}
\Delta_2(t)=-\sum_{j\geqslant u}t_ua_{0j}+\sum_{j>u}t_u(a_{1j}-t_j)-\sum_{u<j\leqslant-u}t_ut_j-\sum_{u\leqslant0}\frac{1}{2}t_u(t_u-3).    
\end{equation*}

\item[(3)] If $\mathrm{row}_{\mathfrak{c}}(A)=\mathrm{col}_{\mathfrak{c}}(C)$ and $C-RE_{\theta}^{i+1,i}$ is diagonal and $i\in [0..r-1]$, then we have 
\begin{align*}
[C][A]
=&\sum_{t}(q_0^{-1}q_1)^{\frac{1}{2}\delta_{i0}\sum_{u>0}t_u}q^{\Delta_3(t)}\prod_{u\in\mathbb{Z}}
\overline{\left[\begin{array}{c}
a_{i+1,u}+t_u\\
t_u
\end{array}\right]}
[\widehat{A}_{t,i}],
\end{align*}
where 
\begin{equation*}
\Delta_3(t)=-\sum_{j\leqslant u}t_ua_{i+1,j}+\sum_{j<u}t_u(a_{ij}-t_j)-\delta_{ir}(\sum_{-u<j<u}t_jt_u+\sum_{u>0}\frac{1}{2}t_u(t_u+3)).    
\end{equation*}

\item[(4)]  If $\mathrm{row}_{\mathfrak{c}}(A)=\mathrm{col}_{\mathfrak{c}}(C)$ and $B-RE_{\theta}^{r+1,r}$ is diagonal, then we have 
\begin{align*}
[C][A]
=&\sum_{t}(q_0q_1^{-1})^{\frac{1}{2}\sum_{u\geqslant r+1}t_{u}}q^{\Delta_4(t)}\overline{\bigg(\frac{[a_{r+1,r+1}'+t_{r+1}]_{\mathfrak{c}_d}^!}{[a_{r+1,r+1}']_{\mathfrak{c}_d}^![t_{r+1}]^!}\prod_{u>0}\frac{[a_{r+1,u}+t_u+t_{-u}]^!}{[a_{r+1,u}]^![t_u]^![t_{-u}]^!}\bigg)}
[\widehat{A}_{t,r}],
\end{align*} 
where 
\begin{equation*}
\Delta_4(t)=-\sum_{j\leqslant u}t_ua_{r+1,j}+\sum_{j<u}t_u(a_{rj}-t_j)-\sum_{n-u\leqslant j<u}t_ut_j-\sum_{u\geqslant r+1}\frac{1}{2}t_u(t_u-3).    
\end{equation*}
\end{thm}

\subsection{Canonical basis at the specialization}\label{sec:4.4}

For any weight function $\mathbf{L}$, let $\boldsymbol{c}=\mathrm{gcd}(|\boldsymbol{L(s_0)}-\boldsymbol{L(s_d)}|,\boldsymbol{L(s_0)}+\boldsymbol{L(s_d)},\boldsymbol{L(s_1)})$. Denote by $\mathbb{H}^{\mathbf{L}}$ (resp. $\mathbb{S}_{n,d}^{\mathfrak{c},\mathbf{L}}$) the specialization of $\mathbb{H}$ (resp. $\mathbb{S}_{n,d}^\mathfrak{c}$) at $q=\boldsymbol{v}^{-\mathbf{L}(s_1)}$, $q_0=\boldsymbol{v}^{-\mathbf{L}(s_0)+\mathbf{L}(s_d)}$, $q_1=\boldsymbol{v}^{-\mathbf{L}(s_0)-\mathbf{L}(s_d)}$. We show that $\mathbb{S}_{n,d}^{\mathfrak{c},\mathbf{L}}$ admits canonical basis with respect to $\boldsymbol{v}^{-c}$. For $A\in\Xi_{n,d}$, let $[A]^{\mathbf{L}}$ be the standard basis of $\mathbb{S}_{n,d}^{\mathfrak{c},\mathbf{L}}$. 

For any $A\in\Xi_{n,d}$, following \cite{Du92} we define 
\begin{equation*}
\{A\}^{\mathbf{L}}\in\mathrm{Hom}_{\mathbb{H}^{\mathbf{L}}}(x_{\mu}\mathbb{H}^{\mathbf{L}},x_{\lambda}\mathbb{H}^{\mathbf{L}})~\text{by requiring}~\{A\}^{\mathbf{L}}(C_{w_0^{\mu}}^{\mathbf{L}})=C_{g_{\lambda\mu}^+}^{\mathbf{L}}, 
\end{equation*}
and hence $\{A\}^{\mathbf{L}}\in\mathbb{S}_{n,d}^{\mathfrak{c},\mathbf{L}}$. Note that $\overline{\{A\}^{\mathbf{L}}}(C_{w_0^{\mu}}^{\mathbf{L}})=C_{g_{\lambda\mu}^+}^{\mathbf{L}}$, and hence $\overline{\{A\}^{\mathbf{L}}}=\{A\}^{\mathbf{L}}$.

Following \cite[(2.c),Lemma 3.8]{Du92}, we have
\begin{equation}\label{can}
\{A\}^{\mathbf{L}}\in[A]^{\mathbf{L}}+\sum_{B<_{\mathrm{alg}}A}\boldsymbol{v}^c\mathbb{Z}[\boldsymbol{v}^c][B]^{\mathbf{L}}    
\end{equation}

\begin{thm}\label{thm:canonicalatspe}
There exists a canonical basis $\mathfrak{B}_{n,d}^{\mathfrak{c}}=\{\{A\}^{\mathbf{L}}~|~A\in\Xi_{n,d}\}$ for $\mathbb{S}_{n,d}^{\mathfrak{c},\mathbf{L}}$ which is characterized by the property (\ref{can}).    
\end{thm}

\subsection{Monomial basis}
For $A\in\Xi_{n,d}$, we can use \cite[Algorithm 5.15]{FLLLW23} with the fixed order therein to produce a unique family of tridiagonal matrices 
$$A^{(1)},\ldots,A^{(x)} \in \Xi_{n,d}$$ for some $x=x(A)\in\mathbb{N}$.
Denote
\begin{equation*}
m'_A=[A^{(1)}]\cdots[A^{(x)}]\in\mathbb{S}_{n,d}^\mathfrak{c}.   
\end{equation*}
Since \cite[Algorithm 5.15]{FLLLW23} produces matrices $A^{(1)},\ldots,A^{(x)}$ according to mainly the off-diagonal matrices of $A$ and then determines the diagonal entries of these $A^{(i)}$ by the row and column sums, we have $x(A)=x(A+pI)$ and $(A+pI)^{(i)}=A^{(i)}+pI$ for all $p\in2\mathbb{N}$. That is,
\begin{equation*}
m'_{A+pI}=[A^{(1)}+pI]\cdots[A^{(x)}+pI],    
\end{equation*} where $I$ is the identity matrix.

\begin{proposition}\label{monomial basis}
For $A\in\Xi_{n,d}$ the element $m'_A\in\mathbb{S}_{n,d}^\mathfrak{c}$ satisfies
\begin{equation*}
m'_A=[A]+\sum_{B<_{\mathrm{alg}}A}\mathbb{A}[B].    
\end{equation*}
Therefore, the set $\{m'_A\}_{A\in\Xi_{n,d}}$ also forms a basis of $\mathbb{S}_{n,d}^\mathfrak{c}$.
\end{proposition}
\begin{proof}
The proof can be pursued using Proposition~\ref{invariant}, similar to the proof of \cite[Theorem 5.16]{FLLLW23}.    
\end{proof}    

Denote $$m_A^\mathbf{L}=\{A^{(1)}\}^\mathbf{L}\{A^{(2)}\}^{\mathbf{L}}\cdots\{A^{(x)}\}^{\mathbf{L}}\in  \mathbb{S}_{n,d}^{\mathfrak{c},\mathbf{L}}.$$ By definition, we have the following statement.

\begin{proposition}
\begin{itemize}
\item[(1)] The element $m_A^\mathbf{L}$ is bar invariant. i.e. $\overline{m_A^\mathbf{L}}=m_A^\mathbf{L}$.
\item[(2)] The set $\{m_A^{\mathbf{L}}\}$ forms a basis of $\mathbb{S}_{n,d}^{\mathfrak{c},\mathbf{L}}$.
\end{itemize}
\end{proposition}

We call $\{m_A^{\mathbf{L}}\}$ the {\em monomial basis} of $\mathbb{S}_{n,d}^{\mathfrak{c},\mathbf{L}}$. Traditionally, the monomial basis was introduced as an intermediate step toward construction of the canonical basis. We reverse the order in which these two bases were introduced for $\mathbb{S}_{n,d}^{\mathfrak{c},\mathbf{L}}$. This monomial basis of $\mathbb{S}_{n,d}^{\mathfrak{c},\mathbf{L}}$ will be employed to construct the monomial basis (and hence the canonical basis) of the stabilization algebra (at the specialization) introduced in the next section. 

\section{Stabilization Algebra}

We denote the field $\mathbb{F}=\mathbb{Q}(q_0^{\pm\frac{1}{2}},q_1^{\pm\frac{1}{2}},q^{\pm\frac{1}{2}})$.

\subsection{Stabilization of multiplication and bar involution}\label{sec:5.1}

Let
\begin{align*}
\widetilde{\Xi}_n=\bigg\{(a_{ij})_{-n\leqslant i,j\leqslant n}\in\mathrm{Mat}_{N\times N}(\mathbb{Z})~\bigg|~&a_{-i,-j}=a_{ij}~(\forall i,j);\\
&a_{xy}\in\mathbb{N}~(\forall x\neq y);~a_{ii}\in2\mathbb{Z}+1~(\forall i\in\mathbb{Z}(r+1))\bigg\}.
\end{align*}
Noting that $\sigma_{ij}(A)$ for $i<j$ does not involve the diagonal elements of $A$, we can extend the partial ordering $\leqslant_{\mathrm{alg}}$ for $\Xi_n$ to a partial ordering $\leqslant_{\mathrm{alg}}$ on $\widetilde{\Xi}_n$ using the same notation. For each $A\in\widetilde{\Xi}_n$ and $p\in2\mathbb{N}$, we write
${}_{p}A=A+pI\in\widetilde{\Xi}_n$. Then ${}_{p}A=A+pI$ for even $p\gg 0$. Let $\pi$ and $\pi'$ be indeterminates (independent of $q_0$, $q_1$, $q$), and $\mathcal{R}_1$ be the subring of $\mathbb{F}[\pi,\pi^{-1}]$ generated by, for $a\in\frac{1}{2}\mathbb{Z}$, $k\in\mathbb{Z}_{>0}$,
\begin{equation*}
r_{a,k}^{(1)},\quad r_{a,k}^{(2)},\quad r_{a,k}^{(3)},\quad q_0^a,\quad q_1^a,\quad q^a,  
\end{equation*}
where
\begin{gather*}
r_{a,k}^{(1)}(q_0,q_1,q,\pi)=\prod\limits_{i=1}^k\frac{q^{-2(a+i)}\pi^{-2}-1}{q^{-2i}-1},\\
r_{a,k}^{(2)}(q_0,q_1,q,\pi)=\prod\limits_{i=1}^k\frac{(q_0^{-1}q_1^{-1}q^{-2(a+i-1)}\pi^{-1}+1)(q^{-2(a+i)}\pi^{-1}-1)}{q^{-2i}-1},\\
r_{a,k}^{(3)}(q_0,q_1,q,\pi)=\prod\limits_{i=1}^k\frac{(q_0q_1^{-1}q^{-2(a+i-1)}\pi^{-1}+1)(q^{-2(a-i)}\pi^{-1}-1)}{q^{2i}-1}.
\end{gather*}
Let $\mathcal{R}_2$ be the subring of $\mathbb{F}[\pi,\pi^{-1}]$ generated by, for $a\in\frac{1}{2}\mathbb{Z}$, $k\in\mathbb{Z}_{>0}$,
\begin{equation}\label{coe}
r_{a,k}^{(1)},\quad \overline{r}_{a,k}^{(1)},\quad r_{a,k}^{(2)},\quad \overline{r}_{a,k}^{(2)},\quad r_{a,k}^{(3)},\quad \overline{r}_{a,k}^{(3)},\quad q_0^a,\quad q_1^a,\quad q^a,  
\end{equation}
and $\mathcal{R}_3=\mathcal{R}_2[\pi',\pi'^{-1}]$ be the subring of $\mathbb{F}[\pi,\pi^{-1},\pi',\pi'^{-1}]$, where we extend the bar-involution to act on $\pi$ (resp. $\pi'$) by $\overline{\pi}=\pi^{-1}$ (resp. $\overline{\pi'}=\pi'^{-1}$). 

\begin{proposition}\label{5.1}
Let $A_1,\ldots,A_f\in\widetilde{\Xi}_n$ be such that $\mathrm{col}_{\mathfrak{c}}(A_i)=\mathrm{row}_{\mathfrak{c}}(A_{i+1})$ for all $i$. There exist matrices $Z_1,\ldots,Z_m\in\widetilde{\Xi}_n$ and $\zeta_i(q_0,q_1,q,\pi)\in\mathcal{R}_2$ such that for even integer $p\gg0$,
\begin{equation*}
[{}_pA_1][{}_pA_2]\cdots[{}_pA_f]=\sum_{i=1}^m\zeta_i(q_0,q_1,q,q^{-p})[{}_pZ_i].    
\end{equation*}
\end{proposition}

\begin{proof}
We first assume that $f=2$ and $B=A_1$ is tridiagonal. Let $A=A_2$. 

For even $p\gg0$ such that all entries in $A_i$ are in $\mathbb{N}$, we apply Theorem \ref{standard} to obtain
\begin{equation*}
[_pB][_pA]
=\sum_{\scalebox{0.7}{$\substack{T\in\Theta_{_pB,{}_pA} \\ S\in \Gamma_{T}}$}}(q^{-2}-1)^{n(S)}q_0^{\beta_0(_pA,S,T)}q_1^{\beta_1(_pA,S,T)}q^{\beta(_pA,S,T)}\overline{\llbracket  {}_pA;S;T\rrbracket} [_pA^{(T-S)}],~\mbox{where}
\end{equation*}
\begin{align*}
\beta_0({}_pA,S,T)=&\alpha_0({}_pA,S,T)-\gamma_0({}_pA,S,T)+\frac{1}{2}\big(\hat\ell_{\mathfrak{c}_0}({}_pA)+\hat\ell_{\mathfrak{c}_0}({}_pB)-\hat\ell_{\mathfrak{c}_0}({}_pA^{(T-S)})\big)\\
&-\frac{1}{2}\big(\hat\ell_{\mathfrak{c}_d}({}_pA)+\hat\ell_{\mathfrak{c}_d}({}_pB)-\hat\ell_{\mathfrak{c}_d}({}_pA^{(T-S)})\big),\\
\beta_1({}_pA,S,T)=&\alpha_1({}_pA,S,T)-\gamma_1({}_pA,S,T)+\frac{1}{2}\big(\hat\ell_{\mathfrak{c}_0}({}_pA)+\hat\ell_{\mathfrak{c}_0}({}_pB)-\hat\ell_{\mathfrak{c}_0}({}_pA^{(T-S)})\big)\\&+\frac{1}{2}\big(\hat\ell_{\mathfrak{c}_d}({}_pA)+\hat\ell_{\mathfrak{c}_d}({}_pB)-\hat\ell_{\mathfrak{c}_d}({}_pA^{(T-S)})\big),\\
\beta({}_pA,S,T)=&\alpha({}_pA,S,T)-\gamma({}_pA,S,T)+\hat\ell_{\mathfrak{a}}({}_pA)+\hat\ell_{\mathfrak{a}}({}_pB)-\hat\ell_{\mathfrak{a}}({}_pA^{(T-S)}).
\end{align*}
It is clear that $\beta_i({}_pA,S,T)=\beta_i(A,S,T)$, $i=0,1$. Hence, we compute the difference $\beta({}_pA,S,T)-\beta(A,S,T)$ term by term as follows:
\begin{align*}
\ell_{\mathfrak{a}}(w_{{}_pA,T})-\ell_{\mathfrak{a}}(w_{A,T})=&p\sum_{i=1}^{n}\sum_{j>i}t_{ij},\\
\gamma({}_pA,S,T)-\gamma(A,S,T)=&p\sum_{l\in\{0,r+1\}}(S+\widehat{(T-S)})_{\theta,ij}+2p\sum_{(i,j)\in  I^+}(S+\widehat{(T-S)})_{\theta,ll},\\
2\ell_{\mathfrak{a}}({}_pA)-2\ell_{\mathfrak{a}}(A)=&p\sum_{i=0}^{r+1}(\sum_{y<i<x}+\sum_{x<i<y})a_{xy}\\
&+p\sum_{x=1}^r\sum_{y<x-1}(x-y-1)a_{xy}+p\sum_{x=1}^r\sum_{y>x+1}(y-x-1)a_{xy}\\
&+p\sum_{y>1}(y-1)a_{0y}+p\sum_{y<r}(r-y)a_{r+1,y},\\
2\hat\ell_{\mathfrak{a}}({}_pA)-2\hat\ell_{\mathfrak{a}}(A)=&p\sum_{i=0}^{r+1}(\sum_{y<i\leqslant x}+\sum_{x\leqslant i<y})a_{xy}+p\sum_{x=1}^r|y-x|a_{xy}\\
&+p\sum_{y>0}ya_{0y}+p\sum_{y<r+1}(r+1-y)a_{r+1,y},\\
\hat\ell_{\mathfrak{a}}({}_pB)-\hat\ell_{\mathfrak{a}}(B)=&p\sum_{i=1}^nb_{i,i+1}.
\end{align*}
Combining all the above gives us
\begin{equation*}
\beta({}_pA,S,T)=\beta(A,S,T)+p\zeta_1(B,A,S,T),    
\end{equation*}
where $\zeta_1(B,A,S,T)$ depends only on the entries of $B,A,S,T$ (independent of $p$).

On the other hand, set
\begin{align*}
a_{ij}^{(1)}=&p\delta_{ij}+(A-T_{\theta})_{ij}+s_{-i,-j}+\widehat{(T-S)}_{\theta,ij},\\
a_{ij}^{(2)}=&p\delta_{ij}+(A-T_{\theta})_{ij}+\widehat{(T-S)}_{\theta,ij},\\
a_{ij}^{(3)}=&p\delta_{ij}+(A-T_{\theta})_{ij}+\widehat{(T-S)}_{-i,-j},\\
a_{ij}^{(4)}=&p\delta_{ij}+(A-T_{\theta})_{ij},\\
a_{ij}^{(5)}=&p+a_{kk}-2t_{kk}-1+2s_{kk}\in2\mathbb{Z},\\
a_{ij}^{(6)}=&p+a_{kk}-2t_{kk}-1\in2\mathbb{Z}.
\end{align*}
We have 
\begin{align*}
&\llbracket  {}_pA;S;T\rrbracket=\\
&\prod\limits_{(i,j)\in I_{\mathfrak{a}}^+}\big(\prod_{l=1}^{s_{ij}}\frac{[a_{ij}^{(1)}+l]}{[l]}\prod_{l=1}^{s_{-i,-j}}\frac{[a_{ij}^{(2)}+l]}{[l]}\prod\limits_{l=1}^{\widehat{(T-S)}_{ij}}\frac{[a_{ij}^{(3)}+l]}{[l]}\prod\limits_{l=1}^{\widehat{(T-S)}_{-i,-j}}\frac{[a_{ij}^{(4)}+l]}{[l]}\big)\\
&\cdot\prod\limits_{l=1}^{\widehat{(T-S)}_{00}}\frac{[a_{00}^{(5)}+2l]_{\mathfrak{c}_0}}{[l]}\prod\limits_{l=1}^{s_{00}}\frac{[a_{00}^{(6)}+2l]_{\mathfrak{c}_0}}{[l]}\prod_{l=1}^{\widehat{(T-S)}_{r+1,r+1}}\frac{[a_{r+1,r+1}^{(5)}+2l]_{\mathfrak{c}_1}}{[l]}\prod\limits_{l=1}^{s_{r+1,r+1}}\frac{[a_{r+1,r+1}^{(6)}+2l]_{\mathfrak{c}_1}}{[l]}\llbracket S\rrbracket.
\end{align*}
Hence $q_0^{\beta_0({}_pA,S,T)}q_1^{\beta_1({}_pA,S,T)}q^{\beta({}_pA,S,T)}\overline{\llbracket  {}_pA;S;T\rrbracket}$ is of the form $G(q_0,q_1,q,q^{-p})$ for some $G(q_0,q_1,q,\pi)\in\mathcal{R}_2$.

Using induction on $f$, we know that the proposition holds for general $f$ in the case where $A_1,\ldots,A_f$ are tridiagonal matrices. It follows from Proposition \ref{monomial basis} that for any $A\in\Xi_{n,d}$, there exists tridiagonal matrices $B_1,B_2,\ldots B_M$ such that
\begin{equation*}
[B_1][B_2]\cdots[B_M]=[A]+\text{lower terms}.    
\end{equation*}
Finally, we can prove the proposition by using induction on $\psi(A)=\sum_{(i,j)\in I^+}\sigma_{ij}(A)$ similar to the proof of \cite[Proposition 4.2]{BLM90}.
\end{proof}

Using an argument similar to \cite[Proposition 6.2]{FLLLW23}, we obtain below the stabilization of bar involution at specialization $q=\boldsymbol{v}^{-\mathbf{L}(s_1)}$, $q_0=\boldsymbol{v}^{-\mathbf{L}(s_0)+\mathbf{L}(s_d)}$, $q_1=\boldsymbol{v}^{-\mathbf{L}(s_0)-\mathbf{L}(s_d)}$, by allowing extra coefficients as seen in (\ref{coe}).
\begin{proposition}\label{prop:sta-bar}
For any $A\in\widetilde{\Xi}_n$, there exist matrices $T_1,\ldots,T_s\in\widetilde{\Xi}_n$ and \\
$\tau_i(\boldsymbol{v}^{-\mathbf{L}(s_0)+\mathbf{L}(s_d)},\boldsymbol{v}^{-\mathbf{L}(s_0)-\mathbf{L}(s_d)},\boldsymbol{v}^{-\mathbf{L}(s_1)},\pi,\pi')\in\mathcal{R}_2\big|_*$ such that for even integer $p\gg0$,
\begin{equation*}
\overline{[{}_pA]^{\mathbf{L}}}=\sum_{i=1}^s\tau_i(\boldsymbol{v}^{-\mathbf{L}(s_0)+\mathbf{L}(s_d)},\boldsymbol{v}^{-\mathbf{L}(s_0)-\mathbf{L}(s_d)},\boldsymbol{v}^{-\mathbf{L}(s_1)},\boldsymbol{v}^{p},\boldsymbol{v}^{-p^2})\overline{[{}_pT_i]^{\mathbf{L}}}.   
\end{equation*}
\end{proposition}

\subsection{Stabilization algebra \texorpdfstring{$\dot{\mathbb{K}}_n^\mathfrak{c}$}{Kc}}
Let $\dot{\mathbb{K}}_n^{\mathfrak{c}}$ be the free $\mathbb{A}$-module with basis given by the symbols $[A]$ for $A\in\widetilde{\Xi}_n$ (which will be called a standard basis of $\dot{\mathbb{K}}_n^{\mathfrak{c}}$). By Proposition \ref{5.1} and applying a specialization at $\pi=1$, we have the following corollary.

\begin{corollary}\label{5.3}
There is a unique associative $\mathbb{A}$-algebra structure on $\dot{\mathbb{K}}_n^{\mathfrak{c}}$ with multiplication given by
\begin{equation*}
[A_1][A_2]\cdots[A_f]=\sum_{i=1}^m\zeta_i(q_0,q_1,q,1)[Z_i]
\end{equation*}
if $\mathrm{col}_{\mathfrak{c}}(A_i)=\mathrm{row}_{\mathfrak{c}}(A_{i+1})$ for all $i$, and $[A_1][A_2]\cdots[A_f]=0$ otherwise.
\end{corollary}

The following multiplication formula in $\dot{\mathbb{K}}_n^{\mathfrak{c}}$ follows directly from Theorem \ref{standard formula} by the stabilization construction.

\begin{thm}
Let $A$, $B\in\widetilde{\Xi}_n$ with $B$ tridiagonal and $\mathrm{row}_{\mathfrak{c}}(A)=\mathrm{col}_{\mathfrak{c}}(B)$. Then
\begin{equation*}
[B][A]
=\sum_{\scalebox{0.7}{$\substack{T\in\widetilde{\Theta}_{B,A} \\ S\in \Gamma_{T}}$}}(q^{-2}-1)^{n(S)}q_0^{\beta_0(A,S,T)}q_1^{\beta_1(A,S,T)}q^{\beta(A,S,T)}\overline{\llbracket  A;S;T\rrbracket} [A^{(T-S)}],
\end{equation*}
where
$\widetilde{\Theta}_{B,A}=\left\{T\in\Theta_n|\ T_{\theta}\leqslant_e A~\text{unless}~i=j,~\mathrm{row}_{\mathfrak{a}}(T)_i=b_{i-1,i}~\text{for all}~i\right\}$.
\end{thm}

The following proposition follows from Proposition \ref{monomial basis} by the stabilization construction.

\begin{proposition}
For any $A\in\widetilde{\Xi}_{n,d}$, there exist tridiagonal matrices $A^{(1)},\ldots,A^{(x)}$ in $\widetilde{\Xi}_{n,d}$ satisfying $\mathrm{row}_{\mathfrak{c}}(A^{(1)})=\mathrm{row}_{\mathfrak{c}}(A)$, $\mathrm{col}_{\mathfrak{c}}(A)=\mathrm{col}_{\mathfrak{c}}(A^{(x)})$, $\mathrm{col}_{\mathfrak{c}}(A^{(i)})=\mathrm{row}_{\mathfrak{c}}(A^{(i+1)})$ for $1\leqslant i\leqslant x-1$
\begin{equation*}
[A^{(1)}][A^{(2)}]\cdots[A^{(x)}]\in[A]+\sum_{B<_{\mathrm{alg}}A}\mathbb{A}[B]\in\dot{\mathbb{K}}_n^{\mathfrak{c}}.    
\end{equation*}    
\end{proposition}

By abuse of notation, we denote the product in $\dot{\mathbb{K}}_n^{\mathfrak{c}}$ by
\begin{equation*}
m'_A=[A^{(1)}][A^{(2)}]\cdots[A^{(x)}]\in\dot{\mathbb{K}}_n^{\mathfrak{c}}.    
\end{equation*}
Hence, $\{m'_A~|~A\in\widetilde{\Xi}_n\}$ forms a basis for $\dot{\mathbb{K}}_n^{\mathfrak{c}}$, which is called the semi-monomial basis. 

Denote $\dot{\mathbb{K}}_n^{\mathfrak{c},\mathbf{L}}$ the specialization of $\dot{\mathbb{K}}_n^{\mathfrak{c}}$ at $q=\boldsymbol{v}^{-\mathbf{L}(s_1)}$, $q_0=\boldsymbol{v}^{-\mathbf{L}(s_0)+\mathbf{L}(s_d)}$, $q_1=\boldsymbol{v}^{-\mathbf{L}(s_0)-\mathbf{L}(s_d)}$. By abuse of notation, we define elements $[A]^{\mathbf{L}}$ and ${m'}_A^{\mathbf{L}}$ as the corresponding basis elements of $\dot{\mathbb{K}}_n^{\mathfrak{c},\mathbf{L}}$. 

Thanks to Proposition~\ref{prop:sta-bar}, we can introduce a bar involution $\overline{~\cdot~}: \dot{\mathbb{K}}_n^{\mathfrak{c},\mathbf{L}}\to\dot{\mathbb{K}}_n^{\mathfrak{c},\mathbf{L}}$ given by 
$$\overline{v^k[A]^\mathbf{L}}=v^{-k}\sum_{i=1}^s\tau_i(\boldsymbol{v}^{-\mathbf{L}(s_0)+\mathbf{L}(s_d)},\boldsymbol{v}^{-\mathbf{L}(s_0)-\mathbf{L}(s_d)},\boldsymbol{v}^{-\mathbf{L}(s_1)},1,1)\overline{[T_i]^{\mathbf{L}}}.$$

Similarly to Section \ref{sec:4.4}, we have the following conclusion.

\begin{thm}\label{thm:canL}
There exists a stably canonical basis $\dot{\mathfrak{B}}^{\mathfrak{c}}=\{\{A\}^{\mathbf{L}}~|~A\in\widetilde{\Xi}_n\}$ for $\dot{\mathbb{K}}_n^{\mathfrak{c},\mathbf{L}}$ such that $\overline{\{A\}^\mathbf{L}}=\{A\}^\mathbf{L}$ and $\{A\}^\mathbf{L}\in [A]^\mathbf{L}+\sum_{B\in\widetilde{\Xi}_n, B<_{\mathrm{alg}}A}v^\mathbf{c}\mathbb{Z}[v^\mathbf{c}][B]^\mathbf{L}$. 
\end{thm}

\begin{corollary}
there is a monomial basis $\{m_A^\mathbf{L}\}$ for $\dot{\mathbb{K}}_n^{\mathfrak{c},\mathbf{L}}$ with  $\overline{m_A^\mathbf{L}}=m_A^\mathbf{L}$, where $$m_A^\mathbf{L}:=\{A^{(1)}\}\{A^{(2)}\}\cdots\{A^{(x)}\}.$$
\end{corollary}

Let $\dot{\mathcal{K}}_n^{\mathfrak{c}}$ be the subalgebra of $\dot{\mathbb{K}}_n^{\mathfrak{c}}$ generated by Chevalley generators $\{[A]~|~A\in\widetilde{\Xi}_n \mbox{~such that~} A-R E^{i,i+1}_\theta \mbox{is diagonal for some $i\in [-r..r]$ and $R\in \mathbb{N}$}\}$. Furthermore, the following multiplication formula in $\dot{\mathcal{K}}_n^{\mathfrak{c}}$ follows directly from Theorem \ref{standard formula} by the stabilization construction. 

\begin{thm}\label{stabilization formula}
Let $A$, $B$, $C\in\widetilde{\Xi}_n$, and $R\in \mathbb{N}$.
\item[(1)] If $\mathrm{row}_{\mathfrak{c}}(A)=\mathrm{col}_{\mathfrak{c}}(B)$ and $B-RE_{\theta}^{i,i+1}$ is diagonal and $i\in [1..r]$, then we have 
\begin{align*}
[B][A]
=&\sum_{t}(q_0q_1)^{\frac{1}{2}\delta_{ir}\sum_{u<r+1}t_u}q^{\Delta_1(t)}\prod\limits_{u\in\mathbb{Z}}
\overline{\left[\begin{array}{c}
a_{iu}+t_u\\
t_u
\end{array}\right]}
[A_{t,i}],
\end{align*}    
where $t=(t_u|u\in\mathbb{Z})$ with $t_u\in\mathbb{N}$ and $\sum_{u\in\mathbb{Z}}t_u=R$ such that 
\begin{equation}\label{sf1}
\left\{
\begin{aligned}
&t_u\leqslant a_{i+1,u},&&\text{if}~i\neq u-1<r,\\
&t_u+t_{-u}\leqslant a_{i+1,u},&&\text{if}~i=r,u\neq r+1.\\
\end{aligned}
\right.
\end{equation}

\item[(2)] If $\mathrm{row}_{\mathfrak{c}}(A)=\mathrm{col}_{\mathfrak{c}}(B)$ and $B-RE_{\theta}^{0,1}$ is diagonal, then we have 
\begin{align*}
[B][A]
=&\sum_{t}(q_0q_1)^{-\frac{1}{2}\sum_{u\leqslant0}t_{u}}q^{\Delta_2(t)}\overline{\bigg(\frac{[a_{00}'+t_0]_{\mathfrak{c}_0}^!}{[a_{00}']_{\mathfrak{c}_0}^![t_0]^!}\prod\limits_{u>0}\frac{[a_{0u}+t_u+t_{-u}]^!}{[a_{0u}]^![t_u]^![t_{-u}]^!}\bigg)}
[A_{t,0}], 
\end{align*}
where $t=(t_u|u\in\mathbb{Z})$ with $t_u\in\mathbb{N}$ and $\sum_{u\in\mathbb{Z}}t_u=R$ such that $t_u\leqslant a_{1u}$ if $u\neq 1$.

\item[(3)] If $\mathrm{row}_{\mathfrak{c}}(A)=\mathrm{col}_{\mathfrak{c}}(C)$ and $C-RE_{\theta}^{i+1,i}$ is diagonal and $i\in [0..r-1]$, then we have 
\begin{align*}
[C][A]
=&\sum_{t}(q_0^{-1}q_1)^{\frac{1}{2}\delta_{i0}\sum_{u>0}t_u}q^{\Delta_3(t)}\prod\limits_{u\in\mathbb{Z}}
\overline{\left[\begin{array}{c}
a_{i+1,u}+t_u\\
t_u
\end{array}\right]}
[\widehat{A}_{t,i}],
\end{align*}
where $t=(t_u|u\in\mathbb{Z})$ with $t_u\in\mathbb{N}$ and $\sum_{u\in\mathbb{Z}}t_u=R$ such that 
\begin{equation}\label{sf3}
\left\{
\begin{aligned}
&t_u\leqslant a_{iu},&&\text{if}~i\neq u>0,\\
&t_u+t_{-u}\leqslant a_{iu},&&\text{if}~i=0,u\neq 0.\\
\end{aligned}
\right.
\end{equation}

\item[(4)]  If $\mathrm{row}_{\mathfrak{c}}(A)=\mathrm{col}_{\mathfrak{c}}(C)$ and $B-RE_{\theta}^{r+1,r}$ is diagonal, then we have 
\begin{align*}
[C][A]
=&\sum_{t}(q_0q_1^{-1})^{\frac{1}{2}\sum_{u\geqslant r+1}t_{u}}q^{\Delta_4(t)}\overline{\bigg(\frac{[a_{r+1,r+1}'+t_{r+1}]_{\mathfrak{c}_d}^!}{[a_{r+1,r+1}']_{\mathfrak{c}_d}^![t_{r+1}]^!}\prod_{u>0}\frac{[a_{r+1,u}+t_u+t_{-u}]^!}{[a_{r+1,u}]^![t_u]^![t_{-u}]^!}\bigg)}
[\widehat{A}_{t,r}].
\end{align*} 
where $t=(t_u|u\in\mathbb{Z})$ with $t_u\in\mathbb{N}$ and $\sum_{u\in\mathbb{Z}}t_u=R$ such that $t_u\leqslant a_{ru}$ if $u\neq r$.
\end{thm}

\subsection{Stabilization algebra \texorpdfstring{$\dot{\mathbb{K}}_{\mathfrak{n}}^{\jmath\imath}$}{Knji}}\label{sec:5.2}

We set $\mathfrak{n}=n-1=2r+1$.

Let
\begin{gather*}
\Xi_{\mathfrak{n},d}^{\jmath\imath}=\{A\in\Xi_{n,d}~|~\mathrm{row}_{\mathfrak{c}}(A)_{r+1}=0=\mathrm{col}_{\mathfrak{c}}(A)_{r+1}\},\\
\Lambda^{\jmath\imath}=\{\lambda\in\Lambda_{r,d}~|~\lambda_{r+1}=0\}.
\end{gather*}
By a similar argument, the $\jmath\imath$-analog bijection is following below
\begin{equation*}
\kappa^{\jmath\imath}: \bigsqcup\limits_{\lambda,\mu\in\Lambda^{\jmath\imath}}\{\lambda\}\times\mathscr{D}_{\lambda\mu}\times\{\mu\} \rightarrow\Xi_{\mathfrak{n},d}^{\jmath\imath},\qquad (\lambda,g,\mu)\mapsto A=\kappa^{\jmath\imath}(\lambda,g,\mu).    
\end{equation*}

Now we define the $q$-Schur algebra of type $\jmath\imath$ as 
\begin{equation*}
\mathbb{S}_{\mathfrak{n},d}^{\jmath\imath}=\mathrm{End}_{\mathbb{H}}(\mathop{\oplus}\limits_{\lambda\in\Lambda^{\jmath\imath}}x_{\lambda}\mathbb{H})    
\end{equation*}
By definition the algebra $\mathbb{S}_{\mathfrak{n},d}^{\jmath\imath}$ is naturally a subalgebra of $\mathbb{S}_{\mathfrak{n},d}^{\mathfrak{c}}$. Moreover, both $\{e_A~|~A\in\Xi^{\jmath\imath}\}$ and $\{[A]~|~A\in\Xi^{\jmath\imath}\}$ are bases of $\mathbb{S}_{\mathfrak{n},d}^{\jmath\imath}$ as a free $\mathbb{A}$-module.

\begin{proposition}
For each $A\in\Xi_{\mathfrak{n},d}^{\jmath\imath}$, we have $m_A\in\mathbb{S}_{\mathfrak{n},d}^{\jmath\imath}$. Hence, the set $\{m_A~|~A\in\Xi^{\jmath\imath}\}$ forms a $\mathbb{A}$-basis of $\mathbb{S}_{\mathfrak{n},d}^{\jmath\imath}$. Furthermore, we have $m_A\in[A]+\sum\limits_{\Xi^{\jmath\imath}\ni B<_{\mathrm{alg}}A}\mathbb{A}[B]$.   
\end{proposition}
\begin{proof}
It is similar to the proof of \cite[Theorem 8.6]{FLLLW20}.     
\end{proof}

\begin{thm}
There exists a canonical basis $\dot{\mathfrak{B}}_{n,d}^{\jmath\imath}=\{\{A\}^{\mathbf{L}}~|~A\in\Xi_{\mathfrak{n}}^{\jmath\imath}\}$ for $\mathbb{S}_{\mathfrak{n}}^{\jmath\imath,\mathbf{L}}$ such that $\overline{\{A\}^{\mathbf{L}}}=\{A\}^{\mathbf{L}}$ and $\overline{\{A\}^{\mathbf{L}}}\in[A]^{\mathbf{L}}+\sum\limits_{\Xi^{\jmath\imath}\ni B<_{\mathrm{alg}}A}\boldsymbol{v}^c\mathbb{Z}[\boldsymbol{v}^c][B]^{\mathbf{L}}$. Moreover, we have $\mathfrak{B}_{n,d}^{\jmath\imath}=\mathfrak{B}_{n,d}^{\mathfrak{c}}\cap\mathbb{S}_{\mathfrak{n}}^{\jmath\imath,\mathbf{L}}$.    
\end{thm}
\begin{proof}
The proof of the first half statement is similar to that in \S\ref{sec:5.1}. The second half statement follows from the uniqueness characterization of the canonical basis $\mathfrak{B}_{\mathfrak{n},d}^{\jmath\imath}$.    
\end{proof}

We define two subsets of $\widetilde{\Xi}_{n}$ as follows:
\begin{equation*}
\widetilde{\Xi}_{n}^{<_{\jmath\imath}}=\{A=(a_{ij})\in\widetilde{\Xi}_n~|~a_{r+1,r+1}<0\},\qquad \widetilde{\Xi}_{n}^{>_{\jmath\imath}}=\{A=(a_{ij})\in\widetilde{\Xi}_n~|~a_{r+1,r+1}>0\}.    
\end{equation*}
For any matrix $A\in\widetilde{\Xi}_n$ and $p\in2\mathbb{N}$, we define
\begin{equation*}
{}_{\check{p}}A=A+p(I-E^{r+1,r+1}).    
\end{equation*}

\begin{lemma}
For $A_1,\ldots,A_f\in\widetilde{\Xi}_n^{>_{\jmath\imath}}$, there exist $Z_i\in\widetilde{\Xi}_n^{>_{\jmath\imath}}$ and $\zeta_i^{\jmath\imath}(q_0,q_1,q,\pi)\in\mathcal{R}_1$ such that for all even integers $p\gg0$, we have an identity in $\mathbb{S}_{n,d}^{\jmath\imath}$ of the form
\begin{equation*}
[{}_{\check{p}}A_1][{}_{\check{p}}A_2]\cdots[{}_{\check{p}}A_f]=\sum_{i=1}^m\zeta_i(q_0,q_1,q,q^{-p})^{\jmath\imath}[{}_{\check{p}}Z_i].    
\end{equation*}    
\end{lemma}
\begin{proof}
The proof is similar to the proof of Proposition \ref{5.1} where ${}_pA=A+pI$ is used instead of ${}_{\check{p}}A$.    
\end{proof}

Consequently, the vector space $\dot{\mathbb{K}}_{n}^{>_{\jmath\imath}}$ over $\mathbb{A}$ spanned by the symbols $[A]$, for $A\in\widetilde{\Xi}_{\mathfrak{n}}^{>_{\jmath\imath}}$, is a stabilization algebra whose multiplicative structure is given by
\begin{equation*}
[A_1][A_2]\cdots[A_f]=\sum_{i=1}^m\zeta_i^{\jmath\imath}(q_0,q_1,q,1)[Z_i],    
\end{equation*}
where $\mathrm{col}_{\mathfrak{c}}(A_i)=\mathrm{row}_{\mathfrak{c}}(A_{i+1})$ for all $i$, and 0 otherwise.
By arguments entirely analogous to those for Corollary \ref{5.3}, $\dot{\mathbb{K}}_{n}^{>_{\jmath\imath}}$ admits a semi-monomial basis $\{m_A~|~A\in\widetilde{\Xi}_{\mathfrak{n}}^{>_{\jmath\imath}}\}$. Similarly, $\dot{\mathbb{K}}_{n}^{>_{\jmath\imath},\mathbf{L}}$ admits a canonical basis $\dot{\mathfrak{B}}^{\mathfrak{c},>_{\jmath\imath}}$. Let $\dot{\mathbb{K}}_{\mathfrak{n}}^{\jmath\imath}$ be the $\mathbb{A}$-submodule of $\dot{\mathbb{K}}_{n}^{>_{\jmath\imath}}$ generated by $\{[A]~|~A\in\widetilde{\Xi}_{\mathfrak{n}}^{\imath\jmath}\}$, where 
\begin{equation*}
\widetilde{\Xi}_{\mathfrak{n}}^{\jmath\imath}=\{A\in\widetilde{\Xi}_n^{>_{\jmath\imath}}~|~\mathrm{row}_{\mathfrak{c}}(A)_{r+1}=0=\mathrm{col}_{\mathfrak{c}}(A)_{r+1}\}.    
\end{equation*}
So $\dot{\mathbb{K}}_{\mathfrak{n}}^{\jmath\imath}$ is a subalgebra of $\dot{\mathbb{K}}_{n}^{>_{\jmath\imath}}$. Since the bar-involution on $\dot{\mathbb{K}}_{n}^{>_{\jmath\imath}}$ restricts to an involution on $\dot{\mathbb{K}}_{\mathfrak{n}}^{\jmath\imath}$, we reach the following conclusion.

\begin{proposition}
The set $\dot{\mathbb{K}}_{\mathfrak{n}}^{\jmath\imath,\mathbf{L}}\cap\dot{\mathfrak{B}}^{\mathfrak{c},>_{\jmath\imath}}$ forms a stably canonical basis of $\dot{\mathbb{K}}_{\mathfrak{n}}^{\jmath\imath,\mathbf{L}}$.   
\end{proposition}

Let $\dot{\mathcal{K}}_{n}^{>_{\jmath\imath}}$ (resp. $\dot{\mathcal{K}}_{n}^{\jmath\imath}$) be the subalgebra of $\dot{\mathbb{K}}_{n}^{>_{\jmath\imath}}$ (resp. $\dot{\mathbb{K}}_{\mathfrak{n}}^{\jmath\imath}$) generated by Chevalley generators. Furthermore, the following multiplication formula in $\dot{\mathcal{K}}_{n}^{>_{\jmath\imath}}$ follows directly from Theorem \ref{standard formula} by the stabilization construction.
\begin{thm}
Let $A$, $B$, $C\in\widetilde{\Xi}_n^{>_{\jmath\imath}}$, and $R\in \mathbb{N}$. Theorem \ref{stabilization formula} holds in $\dot{\mathcal{K}}_{n}^{>_{\jmath\imath}}$ as well, except that condition $(\ref{sf3})$ is replaced by:
\begin{equation*}
\left\{
\begin{aligned}
&t_u\leqslant a_{iu},&&\text{if}~i\neq u>0,\\
&t_u+t_{-u}\leqslant a_{iu},&&\text{if}~i=0.\\
\end{aligned}
\right.
\end{equation*}
\end{thm}

\subsection{Stabilization algebra \texorpdfstring{$\dot{\mathbb{K}}_{\mathfrak{n}}^{\imath\jmath}$}{Knij}}

Let
\begin{align*}
\Xi_{\mathfrak{n},d}^{\imath\jmath}&=\{A\in\Xi_{n,d}~|~\mathrm{row}_{\mathfrak{c}}(A)_0=0=\mathrm{col}_{\mathfrak{c}}(A)_0\},\\
\Lambda^{\imath\jmath}&=\{\lambda\in\Lambda_{r,d}~|~\lambda_0=0\}.
\end{align*}
By a similar argument, the $\imath\jmath$-analog bijection is following below
\begin{equation*}
\kappa^{\imath\jmath}: \bigsqcup\limits_{\lambda,\mu\in\Lambda^{\imath\jmath}}\{\lambda\}\times\mathscr{D}_{\lambda\mu}\times\{\mu\} \rightarrow\Xi_{\mathfrak{n},d}^{\imath\jmath},\qquad (\lambda,g,\mu)\mapsto A=\kappa^{\imath\jmath}(\lambda,g,\mu).    
\end{equation*}

Now we define the $q$-Schur algebra of type $\imath\jmath$ as 
\begin{equation*}
\mathbb{S}_{\mathfrak{n},d}^{\imath\jmath}=\mathrm{End}_{\mathbb{H}}(\mathop{\oplus}\limits_{\lambda\in\Lambda^{\imath\jmath}}x_{\lambda}\mathbb{H}).    
\end{equation*}

By definition the algebra $\mathbb{S}_{\mathfrak{n},d}^{\imath\jmath}$ is naturally a subalgebra of $\mathbb{S}_{\mathfrak{n},d}^{\mathfrak{c}}$. Moreover, both $\{e_A~|~A\in\Xi^{\imath\jmath}\}$ and $\{[A]~|~A\in\Xi_{\mathfrak{n},d}^{\imath\jmath}\}$ are bases of $\mathbb{S}_{\mathfrak{n},d}^{\imath\jmath}$ as a free $\mathbb{A}$-module. Similarly to Section \ref{sec:5.2}, we have the following two conclusions.

\begin{proposition}
For each $A\in\Xi_{\mathfrak{n},d}^{\imath\jmath}$, we have $m_A\in\mathbb{S}_{\mathfrak{n},d}^{\imath\jmath}$. Hence, the set $\{m_A~|~A\in\Xi^{\imath\jmath}\}$ forms a $\mathbb{A}$-basis of $\mathbb{S}_{\mathfrak{n},d}^{\imath\jmath}$. Furthermore, we have $m_A\in[A]+\sum\limits_{\Xi^{\imath\jmath}\ni B<_{\mathrm{alg}}A}\mathbb{A}[B]$.   
\end{proposition}

\begin{thm}
There exists a canonical basis $\dot{\mathfrak{B}}_{n,d}^{\imath\jmath}=\{\{A\}^{\mathbf{L}}~|~A\in\Xi_{\mathfrak{n}}^{\imath\jmath}\}$ for $\mathbb{S}_{\mathfrak{n}}^{\imath\jmath,\mathbf{L}}$ such that $\overline{\{A\}^{\mathbf{L}}}=\{A\}^{\mathbf{L}}$ and $\overline{\{A\}^{\mathbf{L}}}\in[A]^{\mathbf{L}}+\sum\limits_{\Xi^{\imath\jmath}\ni B<_{\mathrm{alg}}A}\boldsymbol{v}^c\mathbb{Z}[\boldsymbol{v}^c][B]^{\mathbf{L}}$. Moreover, we have $\mathfrak{B}_{n,d}^{\imath\jmath}=\mathfrak{B}_{n,d}^{\mathfrak{c}}\cap\mathbb{S}_{\mathfrak{n}}^{\imath\jmath,\mathbf{L}}$.    
\end{thm}

We define two subsets of $\widetilde{\Xi}_n$ as follows:
\begin{equation*}
\widetilde{\Xi}_n^{<_{\imath\jmath}}=\{A=(a_{ij})\in\widetilde{\Xi}_n~|~a_{00}<0\},\qquad \widetilde{\Xi}_n^{>_{\imath\jmath}}=\{A=(a_{ij})\in\widetilde{\Xi}_n~|~a_{00}>0\}.    
\end{equation*}
By arguments entirely analogous to those for Section \ref{sec:5.2}, $\dot{\mathbb{K}}_{n}^{>_{\imath\jmath}}$ admits a semi-monomial basis $\{m_A~|~A\in\widetilde{\Xi}_{\mathfrak{n}}^{>_{\imath\jmath}}\}$. Similarly, $\dot{\mathbb{K}}_{n}^{>_{\imath\jmath},\mathbf{L}}$ admits a canonical basis $\dot{\mathfrak{B}}^{\mathfrak{c},>_{\imath\jmath}}$. Let $\dot{\mathbb{K}}_{\mathfrak{n}}^{\imath\jmath}$ be the $\mathbb{A}$-submodule of $\dot{\mathbb{K}}_n^{>_{\imath\jmath}}$ generated by $\{[A]~|~A\in\widetilde{\Xi}_{\mathfrak{n}}^{\imath\jmath}\}$, where 
\begin{equation*}
\widetilde{\Xi}_{\mathfrak{n}}^{\jmath\imath}=\{A\in\widetilde{\Xi}_n^{>_{\jmath\imath}}~|~\mathrm{row}_{\mathfrak{c}}(A)_0=0=\mathrm{col}_{\mathfrak{c}}(A)_0\}.    
\end{equation*}
So $\dot{\mathbb{K}}_{\mathfrak{n}}^{\imath\jmath}$ is a subalgebra of $\dot{\mathbb{K}}_{n}^{>_{\imath\jmath}}$. Furthermore, we have the following conclusion.
\begin{proposition}
The set $\dot{\mathbb{K}}_{\mathfrak{n}}^{\jmath\imath,\mathbf{L}}\cap\dot{\mathfrak{B}}^{\mathfrak{c},>_{\jmath\imath}}$ forms a stably canonical basis of $\dot{\mathbb{K}}_{\mathfrak{n}}^{\jmath\imath,\mathbf{L}}$.   
\end{proposition}

Define $\dot{\mathcal{K}}_n^{>_{\imath\jmath}}$ (resp. $\dot{\mathcal{K}}_{\mathfrak{n}}^{\imath\jmath}$) as a subalgebra of $\dot{\mathbb{K}}_n^{>_{\imath\jmath}}$ (resp. $\dot{\mathbb{K}}_{\mathfrak{n}}^{\imath\jmath}$) generated by Chevalley generators. Furthermore, the following multiplication formula in $\dot{\mathcal{K}}_n^{>_{\imath\jmath}}$ follows directly from Theorem \ref{standard formula} by the stabilization construction.
\begin{thm}
Let $A$, $B$, $C\in\widetilde{\Xi}_n^{>_{\imath\jmath}}$, and $R\in \mathbb{N}$. Theorem \ref{stabilization formula} holds in $\dot{\mathcal{K}}_{n}^{>_{\imath\jmath}}$ as well, except that condition $(\ref{sf1})$ is replaced by:
\begin{equation*}
\left\{
\begin{aligned}
&t_u\leqslant a_{i+1,u},&&\text{if}~i\neq u-1<r,\\
&t_u+t_{-u}\leqslant a_{i+1,u},&&\text{if}~i=r.\\
\end{aligned}
\right.
\end{equation*}
\end{thm}

\subsection{Stabilization algebra \texorpdfstring{$\dot{\mathbb{K}}^{\imath\imath}_{\eta}$}{Knii}}

We set $\eta=n-2=2r$.
Let
\begin{gather*}
\Xi_{\eta,d}^{\imath\imath}=\Xi_{\mathfrak{n}}^{\imath\jmath}\cap\Xi_{\mathfrak{n}}^{\imath\imath}\quad\mbox{and}\quad
\Lambda^{\imath\imath}=\Lambda^{\jmath\imath}\cap\Lambda^{\jmath\imath}.
\end{gather*}
By a similar argument, the $\imath\imath$-analog bijection is following below
\begin{equation*}
\kappa^{\imath\imath}: \bigsqcup\limits_{\lambda,\mu\in\Lambda^{\imath\imath}}\{\lambda\}\times\mathscr{D}_{\lambda\mu}\times\{\mu\} \rightarrow\Xi_{\eta,d}^{\imath\imath},\qquad (\lambda,g,\mu)\mapsto A=\kappa^{\imath\imath}(\lambda,g,\mu).    
\end{equation*}

Now we define the $q$-Schur algebra of type $\imath\imath$ as 
\begin{equation*}
\mathbb{S}_{\eta,d}^{\imath\imath}=\mathrm{End}_{\mathbb{H}}(\mathop{\oplus}\limits_{\lambda\in\Lambda^{\imath\imath}}x_{\lambda}\mathbb{H}).    
\end{equation*}
By definition, the algebra $\mathbb{S}_{\eta,d}^{\imath\imath}$ is naturally a subalgebra of $\mathbb{S}_{\mathfrak{n},d}^{\jmath\imath}$, $\mathbb{S}_{\mathfrak{n},d}^{\imath\jmath}$ and $\mathbb{S}_{n,d}^{\mathfrak{c}}$. Moreover, both $\{e_A~|~A\in\Xi^{\imath\imath}\}$ and $\{[A]~|~A\in\Xi_{\eta,d}^{\imath\imath}\}$ are bases of $\mathbb{S}_{\eta,d}^{\imath\imath}$ as a free $\mathbb{A}$-module. Similarly to Section \ref{sec:5.2}, we have the following two conclusions.

\begin{proposition}
For each $A\in\Xi_{\eta,d}^{\imath\imath}$, we have $m_A\in\mathbb{S}_{\eta,d}^{\imath\imath}$. Hence, the set $\{m_A~|~A\in\Xi^{\imath\imath}\}$ forms a $\mathbb{A}$-basis of $\mathbb{S}_{\eta,d}^{\imath\imath}$. Furthermore, we have $m_A\in[A]+\sum\limits_{\Xi^{\imath\imath}\ni B<_{\mathrm{alg}}A}\mathbb{A}[B]$.   
\end{proposition}

\begin{thm}
There exists a canonical basis $\dot{\mathfrak{B}}_{n,d}^{\imath\imath}=\{\{A\}^{\mathbf{L}}~|~A\in\Xi_{\eta}^{\imath\imath}\}$ for $\mathbb{S}_{\eta}^{\imath\imath,\mathbf{L}}$ such that $\overline{\{A\}^{\mathbf{L}}}=\{A\}^{\mathbf{L}}$ and $\overline{\{A\}^{\mathbf{L}}}\in[A]^{\mathbf{L}}+\sum\limits_{\Xi^{\imath\imath}\ni B<_{\mathrm{alg}}A}\boldsymbol{v}^c\mathbb{Z}[\boldsymbol{v}^c][B]^{\mathbf{L}}$. Moreover, we have $\mathfrak{B}_{n,d}^{\imath\imath}=\mathfrak{B}_{n,d}^{\mathfrak{c}}\cap\mathbb{S}_{\eta}^{\imath\imath,\mathbf{L}}$.    
\end{thm}

We define two subsets of $\widetilde{\Xi}_n$ as follows:
\begin{equation*}
\widetilde{\Xi}_n^{>_{\imath\imath}}=\{A=(a_{ij})\in\widetilde{\Xi}_n~|~a_{00}>0,a_{r+1,r+1}>0\}.    
\end{equation*}
By arguments entirely analogous to those for Section \ref{sec:5.2}, $\dot{\mathbb{K}}_{n}^{>_{\imath\imath}}$ admits a semi-monomial basis $\{m_A~|~A\in\widetilde{\Xi}_{\eta}^{>_{\imath\imath}}\}$. Similarly, $\dot{\mathbb{K}}_{n}^{>_{\imath\imath},\mathbf{L}}$ admits a canonical basis $\dot{\mathfrak{B}}^{\mathfrak{c},>_{\imath\imath}}$. Let $\dot{\mathbb{K}}_{\eta}^{\imath\imath}$ generated by $\{[A]~|~A\in\widetilde{\Xi}_{\eta}^{\imath\imath}\}$, where 
\begin{equation*}
\widetilde{\Xi}_{\eta}^{\imath\imath}=\widetilde{\Xi}_{\mathfrak{n}}^{\imath\jmath}\cap\widetilde{\Xi}_{\mathfrak{n}}^{\jmath\imath}.   
\end{equation*}
So $\dot{\mathbb{K}}_{\eta}^{\imath\imath}$ is a subalgebra of $\dot{\mathbb{K}}_{n}^{>_{\imath\imath}}$. Furthermore, we have the following conclusion.
\begin{proposition}
The set $\dot{\mathbb{K}}_{\eta}^{\imath\imath,\mathbf{L}}\cap\dot{\mathfrak{B}}^{\mathfrak{c},>_{\imath\imath}}$ forms a stably canonical basis of $\dot{\mathbb{K}}_{\eta}^{\imath\imath,\mathbf{L}}$.   
\end{proposition}

Define $\dot{\mathcal{K}}_n^{>_{\imath\imath}}$ (resp. $\dot{\mathcal{K}}_n^{\imath\imath}$) as the subalgebra of $\dot{\mathbb{K}}_n^{>_{\imath\imath}}$ (resp. $\dot{\mathbb{K}}_n^{\imath\imath}$) generated by Chevalley generators. Furthermore, the following multiplication formula in $\dot{\mathcal{K}}_n^{>_{\imath\imath}}$ follows directly from Theorem \ref{standard formula} by the stabilization construction.
\begin{thm}
Let $A$, $B$, $C\in\widetilde{\Xi}_n^{>_{\imath\imath}}$, and $R\in \mathbb{N}$. Theorem \ref{stabilization formula} holds in $\dot{\mathcal{K}}_{n}^{>_{\imath\imath}}$ as well, except that condition $(\ref{sf1})$ and $(\ref{sf3})$ are replaced respectively by: 
\begin{equation*}
\left\{
\begin{aligned}
&t_u\leqslant a_{i+1,u},&&\text{if}~i\neq u-1<r,\\
&t_u+t_{-u}\leqslant a_{i+1,u},&&\text{if}~i=r,\\
\end{aligned}
\right.
\quad\mathrm{and}\quad 
\left\{
\begin{aligned}
&t_u\leqslant a_{iu},&&\text{if}~i\neq u>0,\\
&t_u+t_{-u}\leqslant a_{iu},&&\text{if}~i=0.\\
\end{aligned}
\right.
\end{equation*}
\end{thm}

\section{Quantum symmetric pairs}

\subsection{The quantum symmetric pair \texorpdfstring{$(\mathbb{U},\mathbb{U}^{\mathfrak{c}})$}{(U,Uc)}}

\begin{center}
\begin{tikzpicture}[scale=.4]
\node at (0,0.75) {$0$};
\node at (4,0.75) {$1$};
\node at (12,0.75) {$r-1$};
\node at (16,0.75) {$r$};     
\node at (0,-3.75) {$2r+1$};
\node at (4,-3.75) {$2r$};
\node at (12,-3.75) {$r+2$};
\node at (16,-3.75) {$r+1$}; 
       
\node at (8,0) {$\dots$};
\node at (8,-3) {$\dots$};
    
\foreach \x in {0,2,6,8} 
{\draw[thick,xshift=\x cm] (\x, 0) circle (0.3); 
\draw[thick,xshift=\x cm] (\x, -3) circle (0.3);}

\foreach \x in {0,6}
{\draw[thick,xshift=\x cm] (\x,0) ++(0.5,0) -- +(3,0);
\draw[thick,xshift=\x cm] (\x,-3) ++(0.5,0) -- +(3,0);}
   
\foreach \x in {2,4.5}
{\draw[thick,xshift=\x cm] (\x,0) ++(0.5,0) -- +(2,0);
\draw[thick,xshift=\x cm] (\x,-3) ++(0.5,0) -- +(2,0);}
    
\foreach \x in {0,8}
\draw[thick,xshift=\x cm] (\x,-0.5) -- +(0,-2);
    
\foreach \x in {0,4} 
\draw[thick,<->, blue, bend right=50] (\x+0.3,-2.5) to (\x+0.3,-0.5);
\foreach \x in {12,16} 
\draw[thick,<->, blue, bend left=50] (\x-0.3,-2.5) to (\x-0.3,-0.5);
 
\end{tikzpicture}
\end{center}

\begin{center}
Figure 1: Dynkin diagram of type $A_{2r+1}^{(1)}$ with involution of type $\jmath\jmath\equiv\mathfrak{c}$.
\end{center}

For $i,j\in[0..r]$, we denote the Cartan integers by
\begin{equation*}
c_{i,j}=2\delta_{ij}-\delta_{ij+1}-\delta_{ij-1}.    
\end{equation*}
The quantum affine $\mathfrak{sl}_n$ is the associative algebra $\mathbb{U}=\mathbb{U}(\widehat{\mathfrak{sl}}_n)$ over $\mathbb{F}$ generated by 
\begin{equation*}
E_i,\quad F_i,\quad K_{i}^{\pm1}\quad (0\leqslant i\leqslant n-1),    
\end{equation*}
subject to the following relations: for $0\leqslant i,j\leqslant n-1$,
\begin{gather*}
K_iK_i^{-1}=K_i^{-1}K_i=1,\quad K_iK_j=K_iK_j,\\
K_iE_jK_i^{-1}=q^{c_{i,j}}E_j,\quad K_iF_jK_i^{-1}=q^{c_{i,j}}F_j,\\
E_iF_j-F_jE_i=\delta_{ij}\frac{K_i-K_i^{-1}}{q-q^{-1}},\\
E_i^2E_j+E_jE_i^2=(q+q^{-1})E_iE_jE_i,\quad F_i^2F_j+F_jF_i^2=(q+q^{-1})F_iF_jF_i\quad (|i-j|\equiv1),\\
E_iE_j=E_jE_i,\quad F_iF_j=F_jF_i\quad (i\not\equiv j\pm1),
\end{gather*}
where $i\equiv j$ means $i\equiv j$ (mod $n$).

Let $\mathbb{U}^{\mathfrak{c}}=\mathbb{U}^{\mathfrak{c}}(\widehat{\mathfrak{sl}}_n)$ be the associative algebra over $\mathbb{F}$ generated by
\begin{equation*}
e_i,\quad f_i,\quad k_{i}^{\pm1}\quad (0\leqslant i\leqslant r),    
\end{equation*}
subject to the following relations: for $0\leqslant i,j\leqslant r$,
\begin{gather*}
k_ik_i^{-1}=k_i^{-1}k_i=1,\quad k_ik_j=k_jk_i,\\
k_ie_jk_i^{-1}=q^{c_{i,j}+\delta_{i0}\delta_{j,0}+\delta_{ir}\delta_{j,r}}e_j,\quad k_if_jk_i^{-1}=q^{c_{i,j}+\delta_{i0}\delta_{j,0}+\delta_{ir}\delta_{j,r}}f_j,\\
e_if_j-f_je_i=\delta_{ij}\frac{k_i-k_i^{-1}}{q-q^{-1}}\quad (i,j)\neq(0,0),(r,r),\\
e_i^2e_j+e_je_i^2=(q+q^{-1})e_ie_je_i,\quad f_i^2f_j+f_jf_i^2=(q+q^{-1})f_if_jf_i\quad (|i-j|\equiv1),\\
e_ie_j=e_je_i,\quad f_if_j=f_jf_i\quad (i\neq j\pm1),\\
e_r^2f_r+f_re_r^2=(q+q^{-1})(e_rf_re_r-q_0^{\frac{1}{2}}q_1^{-\frac{1}{2}}q^3e_rk_r-q_0^{-\frac{1}{2}}q_1^{\frac{1}{2}}q^{-3}e_rk_r^{-1}),\\
f_r^2e_r+e_rf_r^2=(q+q^{-1})(f_re_rf_r-q_0^{\frac{1}{2}}q_1^{-\frac{1}{2}}q^3k_rf_r-q_0^{-\frac{1}{2}}q_1^{\frac{1}{2}}q^{-3}k_r^{-1}f_r),\\
e_0^2f_0+f_0e_0^2=(q+q^{-1})(e_0f_0e_0-q_0^{\frac{1}{2}}q_1^{\frac{1}{2}}e_0k_0-q_0^{-\frac{1}{2}}q_1^{-\frac{1}{2}}e_0k_0^{-1}),\\
f_0^2e_0+e_0f_0^2=(q+q^{-1})(f_0e_0f_0-q_0^{\frac{1}{2}}q_1^{\frac{1}{2}}k_0f_0-q_0^{-\frac{1}{2}}q_1^{-\frac{1}{2}}k_0^{-1}f_0).
\end{gather*}

We adopt the following identification for all $i\in\mathbb{Z}$:
\begin{equation*}
E_i=E_{i+n},\quad F_i=F_{i+n},\quad K_i=K_{i+n}.    
\end{equation*}

\begin{proposition}{\cite[Proposition 2.2]{FLLLWW20}}
There are injective $\mathbb{F}$-algebra homomorphisms $\jmath\jmath$: $\mathbb{U}^{\mathfrak{c}}(\widehat{\mathfrak{sl}}_n)\to\mathbb{U}(\widehat{\mathfrak{sl}}_n)$ defined by

\item[(1)] $k_i\mapsto q^{\delta_{i0}-\delta_{ir}}K_iK_{-i-1}^{-1},\quad (0\leqslant i\leqslant r)$,

\item[(2)] $e_i\mapsto E_i+F_{-i-1}K_{i}^{-1},\quad f_i\mapsto E_{-i-1}+F_{i}K_{-i-1}^{-1},\quad (1\leqslant i\leqslant r-1)$,

\item[(3)] $e_0\mapsto E_0+q_0^{-\frac{1}{2}}q_1^{-\frac{1}{2}}F_{-1}K_{0}^{-1},\quad f_0\mapsto E_{-1}+q_0^{\frac{1}{2}}q_1^{\frac{1}{2}}q^{-1}F_{0}K_{-1}^{-1}$,

\item[(4)] $e_r\mapsto E_r+q_0^{-\frac{1}{2}}q_1^{\frac{1}{2}}q^{-1}F_{-r-1}K_{r}^{-1},\quad f_r\mapsto E_{-r-1}+q_0^{\frac{1}{2}}q_1^{-\frac{1}{2}}F_{r}K_{-r-1}^{-1}$.
\end{proposition}

It turns out $(\mathbb{U},\mathbb{U}^{\mathfrak{c}})$ forms a quantum symmetric pair à la Letzter and Kolb.

\subsection{Isomorphism \texorpdfstring{$\dot{\mathbb{U}}^\mathfrak{c}\simeq\dot{\mathcal{K}}_n^\mathfrak{c}$}{Uc=Knc}}

Following \cite[§6.5]{FLLLW20},we define the modified quantum algebra $\dot{\mathbb{U}}^{\mathfrak{c}}$ from $\mathbb{U}^{\mathfrak{c}}$. Let $\widetilde{\Xi}_n^{\mathrm{diag}}$ be the set of all diagonal matrices in $\widetilde{\Xi}_n$. Denote by $\lambda=\sum_{i}\lambda_iE_{ii}$ a diagonal matrix in $\widetilde{\Xi}_n^{\mathrm{diag}}$. For $\lambda,\lambda'\in\widetilde{\Xi}_n^{\mathrm{diag}}$, we set 
\begin{equation*}
{}_{\lambda}\mathbb{U}^{\mathfrak{c}}_{\lambda'}=\mathbb{U}^{\mathfrak{c}}\bigg/\bigg(\sum_{i=0}^{r}(k_i-q^{\lambda_i-\lambda_{i-1}})\mathbb{U}^{\mathfrak{c}}+\sum_{i=0}^{r}\mathbb{U}^{\mathfrak{c}}(k_i-q^{\lambda'_i-\lambda'_{i-1}})\bigg).    
\end{equation*}
The modified quantum algebra $\dot{\mathbb{U}}^\mathfrak{c}$ is defined by
\begin{equation*}
\dot{\mathbb{U}}^\mathfrak{c}=\bigoplus\limits_{\lambda,\lambda'\in\widetilde{\Xi}_n^{\mathrm{diag}}}{}_{\lambda}\mathbb{U}^{\mathfrak{c}}_{\lambda'}.    
\end{equation*}
Let $1_{\lambda}=p_{\lambda,\lambda}(1)$, where $p_{\lambda,\lambda}:~\mathbb{U}^{\mathfrak{c}}\to{}_{\lambda}\mathbb{U}^{\mathfrak{c}}_{\lambda}$ is the canonical projection. Thus, the unit of $\mathbb{U}^{\mathfrak{c}}$ is replaced by a collection of orthogonal idempotents $1_{\lambda}$ in $\dot{\mathbb{U}}^\mathfrak{c}$. It is clear that
\begin{equation*}
\dot{\mathbb{U}}^\mathfrak{c}=\sum_{\lambda\in\widetilde{\Xi}_n^{\mathrm{diag}}}\mathbb{U}^{\mathfrak{c}}1_{\lambda}=\sum_{\lambda\in\widetilde{\Xi}_n^{\mathrm{diag}}}1_{\lambda} \mathbb{U}^{\mathfrak{c}}.   
\end{equation*}
For $\lambda\in\widetilde{\Xi}_n^{\mathrm{diag}}$ and $i\in[0..r]$, we use the following short-hand notations:
\begin{equation*}
\lambda+\alpha_i=\lambda+E_{ii}^{\theta}-E_{i+1,i+1}^{\theta},\quad\lambda-\alpha_i=\lambda-E_{ii}^{\theta}+E_{i+1,i+1}^{\theta}.    
\end{equation*}
We also define, for $r\in\mathbb{N}$,
\begin{equation*}
\llbracket r\rrbracket=\frac{q^r-q^{-r}}{q-q^{-1}}.    
\end{equation*}
We have a presentation of $\dot{\mathbb{U}}^\mathfrak{c}$ as a $\mathbb{F}$-algebra generated by the symbols, for $i\in[0..r]$, $\lambda\in\widetilde{\Xi}_n^{\mathrm{diag}}$,
\begin{equation*}
1_{\lambda},\quad e_i1_{\lambda},\quad 1_{\lambda}e_i,\quad f_i1_{\lambda},\quad 1_{\lambda}f_i,     
\end{equation*}
subject to the following relations, for $i,j\in[0..r]$, $\lambda,\mu\in\widetilde{\Xi}_n^{\mathrm{diag}}$, $x,y\in\{1,e_i,e_j,f_i,f_j\}$:
\begin{gather*}
x1_{\lambda}1_{\mu}y=\delta_{\lambda,\mu}x1_{\lambda}y,\\
e_i1_{\lambda}=1_{\lambda+\alpha_i}e_i,\quad f_i1_{\lambda}=1_{\lambda-\alpha_i}f_i,\\
e_i1_{\lambda}f_j=f_j1_{\lambda+\alpha_i+\alpha_j}e_i\quad(i\neq j),\\
(e_if_i-f_ie_i)1_{\lambda}=\llbracket \lambda_i-\lambda_{i+1}\rrbracket1_{\lambda}\quad(i\neq0,r),\\
(e_i^2e_j+e_je_i^2)1_{\lambda}=\llbracket 2\rrbracket e_ie_je_i1_{\lambda},\quad (f_i^2f_j+f_jf_i^2)1_{\lambda}=\llbracket 2\rrbracket f_if_jf_i1_{\lambda}\quad (|i-j|=1),\\
e_ie_j1_{\lambda}=e_je_i1_{\lambda},\quad f_if_j1_{\lambda}=f_jf_i1_{\lambda}\quad(i\neq j\pm1),\\
(\llbracket 2\rrbracket e_rf_re_r-e_r^2f_r-f_re_r^2)1_{\lambda}=\llbracket 2\rrbracket(q_0^{-\frac{1}{2}}q_1^{\frac{1}{2}}q^{\lambda_{r+1}-\lambda_r-3}+q_0^{\frac{1}{2}}q_1^{-\frac{1}{2}}q^{-\lambda_{r+1}+\lambda_r+3})e_r1_{\lambda},\\
(\llbracket 2\rrbracket f_re_rf_r-f_r^2e_r-e_rf_r^2)1_{\lambda}=\llbracket 2\rrbracket(q_0^{-\frac{1}{2}}q_1^{\frac{1}{2}}q^{\lambda_{r+1}-\lambda_r}+q_0^{\frac{1}{2}}q_1^{-\frac{1}{2}}q^{-\lambda_{r+1}+\lambda_r})f_r1_{\lambda},\\
(\llbracket 2\rrbracket e_0f_0e_0-e_0^2f_0-f_0e_0^2)1_{\lambda}=\llbracket 2\rrbracket(q_0^{\frac{1}{2}}q_1^{\frac{1}{2}}q^{\lambda_{0}-\lambda_1}+q_0^{-\frac{1}{2}}q_1^{-\frac{1}{2}}q^{-\lambda_{0}+\lambda_1})e_01_{\lambda},\\
(\llbracket 2\rrbracket f_0e_0f_0-f_0^2e_0-e_0f_0^2)1_{\lambda}=\llbracket 2\rrbracket(q_0^{\frac{1}{2}}q_1^{\frac{1}{2}}q^{\lambda_{0}-\lambda_1-3}+q_0^{-\frac{1}{2}}q_1^{-\frac{1}{2}+3}q^{-\lambda_{0}+\lambda_1+3})f_01_{\lambda}.
\end{gather*}
Here and below, we shall always write $x_11_{\lambda^1}x_21_{\lambda^2}\cdots x_k1_{\lambda^k}=x_1x_2\cdots x_k1_{\lambda^k}$, if the product is not zero; in this case, such $\lambda^1,\lambda^2,\ldots,\lambda^{k-1}$ are all uniquely determined by $\lambda^k$.

For all $i\in[0..r]$, $\lambda\in\widetilde{\Xi}_n^{\mathrm{diag}}$, write
\begin{equation*}
\mathbf{e}_i1_{\lambda}=[\lambda-E_{i+1,i+1}^{\theta}+E_{i,i+1}^{\theta}]\in\dot{\mathcal{K}}_n^{\mathfrak{c}}\quad\text{and}\quad \mathbf{f}_i1_{\lambda}=[\lambda-E_{ii}^{\theta}+E_{i+1,i}^{\theta}]\in\dot{\mathcal{K}}_n^{\mathfrak{c}}.
\end{equation*}
Set ${}_\mathbb{F}\dot{\mathcal{K}}_n^{\mathfrak{c}}=\mathbb{F}\otimes_{\mathbb{A}}\dot{\mathcal{K}}_n^{\mathfrak{c}}$.

\begin{thm} \label{thm:isoUK}
There is an isomorphism of $\mathbb{F}$-algebras $\aleph:~\dot{\mathbb{U}^{\mathfrak{c}}}\to{}_\mathbb{F}\dot{\mathcal{K}}_n^{\mathfrak{c}}$ such that, $\forall i\in[0..r]$, $\lambda\in\widetilde{\Xi}_n^{\mathrm{diag}}$,
\begin{equation*}
e_i1_{\lambda}\mapsto\mathbf{e}_i1_{\lambda},\quad f_i1_{\lambda}\mapsto\mathbf{f}_i1_{\lambda},\quad 1_{\lambda}\mapsto[\lambda].    
\end{equation*}
\end{thm}

\begin{proof}
Here we only show the details for
\begin{equation*}
(\llbracket 2\rrbracket f_re_rf_r-f_r^2e_r-e_rf_r^2)1_{\lambda}=\llbracket 2\rrbracket(q_0^{-\frac{1}{2}}q_1^{\frac{1}{2}}q^{\lambda_{r+1}-\lambda_r}+q_0^{\frac{1}{2}}q_1^{-\frac{1}{2}}q^{-\lambda_{r+1}+\lambda_r})f_r1_{\lambda}    
\end{equation*}
regarding $\mathbf{e}_r1_{\lambda}$ and $\mathbf{f}_r1_{\lambda}$ as follows:
\begin{align*}
\mathbf{e}_r\mathbf{f}_r^2=&q_0^{-1}q_1q^{\lambda_r+2\lambda_{r+1}-2}[2]\big(\mathbf{e}_{\lambda-2E_{rr}^{\theta}-E_{r+1,r+1}^{\theta}+E_{r,r+1}^{\theta}+2E_{r+1,r}^{\theta}}+e_{\lambda-2E_{rr}^{\theta}+E_{r,r+2}^{\theta}+E_{r+1,r}^{\theta}}\\
&+[\lambda_r-1]e_{\lambda-E_{rr}^{\theta}+E_{r+1,r}^{\theta}}\big),\\
\mathbf{f}_r^2\mathbf{e}_r=&q_0^{-1}q_1q^{\lambda_r+2\lambda_{r+1}-4}[2]\big(e_{\lambda-2E_{rr}^{\theta}-E_{r+1,r+1}^{\theta}+E_{r,r+1}^{\theta}+2E_{r+1,r}^{\theta}}\\
&+[\lambda_{r+1}-1]_{\mathfrak{c}}e_{\lambda-E_{rr}^{\theta}+E_{r+1,r}^{\theta}}\big),\\
\mathbf{f}_r\mathbf{e}_r\mathbf{f}_r=&q_0^{-1}q_1q^{\lambda_r+3\lambda_{r+1}-4}[2]\big(e_{\lambda-2E_{rr}^{\theta}+E_{r,r+2}^{\theta}+E_{r+1,r}^{\theta}}+[2]e_{\lambda-2E_{rr}^{\theta}-E_{r+1,r+1}^{\theta}+E_{r,r+1}^{\theta}+2E_{r+1,r}^{\theta}}\\
&+(q_0q_1^{-1}q^{-2\lambda_{r+1}+2}+[\lambda_r]+[\lambda_{r+1}-1]_{\mathfrak{c}}q^{-2})e_{\lambda-E_{rr}^{\theta}+E_{r+1,r}^{\theta}}\big).
\end{align*}
To sum up, we obtion
\begin{equation*}
(\llbracket 2\rrbracket \mathbf{f}_r\mathbf{e}_r\mathbf{f}_r-\mathbf{f}_r^2\mathbf{e}_r-\mathbf{e}_r\mathbf{f}_r^2)1_{\lambda}=\llbracket 2\rrbracket(q_0^{-\frac{1}{2}}q_1^{\frac{1}{2}}q^{\lambda_{r+1}-\lambda_r}+q_0^{\frac{1}{2}}q_1^{-\frac{1}{2}}q^{-\lambda_{r+1}+\lambda_r})\mathbf{f}_r1_{\lambda}.    
\end{equation*}
Therefore, $\aleph$ is an algebra homomorphism. 

It remains to show that $\aleph$ is a linear isomorphism. The argument is the same as the case of specialization at $(q,q_0,q_1)=(v^{-1},1,v^{-2})$, which will be proved in the following. Set $\epsilon_{0,r+1}=(1,0,\ldots,0,1)\in\mathbb{N}^{r+2}$, where $1$ is in the $0$st and $(r+1)$st position. Note that $\epsilon_{0,r+1}$ can be regarded as the diagonal matrix $E_{00}^{\theta}+E_{r+1,r+1}^{\theta}$ in $\widetilde{\Xi}^{\mathrm{diag}}$. Also set
\begin{align*}
\widetilde{\Theta}
&=\{A=(a_{ij})\in\mathrm{Mat}_{\mathbb{Z}\times\mathbb{Z}}(\mathbb{Z})~|~a_{ij}=a_{i+n,j+n},~\forall i,j\in\mathbb{Z};~ a_{ij}\geqslant0(i\neq j)\},\\    
\widetilde{\Theta}^-
&=\{A\in\widetilde{\Theta}~|~a_{ij}=0(i<j),~\mathrm{col}_{\mathfrak{c}}(A)=0\},\\
\widetilde{\Xi}^-
&=\{A\in\widetilde{\Xi}~|~\mathrm{col}_{\mathfrak{c}}(A)=\epsilon_0+\epsilon_{r+1}\}.
\end{align*}
The diagonal entries of a matrix $A'\in\widetilde{\Theta}^-$ (resp. $A\in\widetilde{\Xi}^-$) are completely determined by its strictly lower triangular entries. Hence, there is a natural bijection $\widetilde{\Theta}^-\leftrightarrow\widetilde{\Xi}^-$, which
sends $A'\in\widetilde{\Theta}^-$ to $A\in\widetilde{\Xi}^-$ such that the strictly lower triangular parts of $A$ and $A'$ are identical.

The quantum group $\mathbb{U}$ has a triangular decomposition $\mathbb{U}=\mathbb{U}^-\mathbb{U}^0\mathbb{U}^+$. Set $\mathcal{A}=\mathbb{Z}[v^{\pm1}]$. Denote by $\dot{\mathbb{U}}$ the modified quantum group with idempotents $D_{\lambda}$ and denote its $\mathcal{A}$-form by $_{\mathcal{A}}\dot{\mathbb{U}}$. Let $\dot{\mathbb{K}}_n$ be the $\mathcal{A}$-algebra in the type A stabilization of \cite{DF15} and ${}_\mathbb{Q}\dot{\mathbb{K}}_n=\mathbb{Q}(v)\otimes_{\mathcal{A}}\dot{\mathbb{K}}_n$. It was shown in \cite{DF15} that there exists a $\mathbb{Q}(v)$-algebra isomorphism $\aleph^{\mathfrak{a}} : \dot{\mathbb{U}}\to\dot{\mathbb{K}}_n$. Note that $\dot{\mathbb{K}}$ has a semi-monomial $\mathcal{A}$-basis $\{m^{\mathfrak{a}}_{A'}~|~A'\in\widetilde{\Theta}\}$ which is given in \cite{LL17}. Via the isomorphism $\aleph^{\mathfrak{a}}$, the semi-monomial $\mathcal{A}$-basis $\{m^{\mathfrak{a}}_{A'}~|~A'\in\widetilde{\Theta}^-\}$ for the $\mathcal{A}$-submodule $\aleph^{\mathfrak{a}}(_{\mathcal{A}}\mathbb{U}^-D_{\epsilon_{0,r+1}})$ of $\dot{\mathbb{K}}_n$ corresponds to a semi-monomial $\mathcal{A}$-basis $\{m^{\mathfrak{u}}_{A'}~|~A'\in\widetilde{\Theta}^-\}$ for the $\mathcal{A}$-submodule $_{\mathcal{A}}\mathbb{U}^-D_{\epsilon_{0,r+1}}$ of $_{\mathcal{A}}\dot{\mathbb{U}}$, where $\aleph^{\mathfrak{a}}(m^{\mathfrak{u}}_{A'})=m^{\mathfrak{a}}_{A'}$.

We also use the notation $\dot{\mathcal{K}}^{\mathfrak{c}}_n$ in this case. Similarly to \cite[Theorem 4.7]{BKLW18}, we obtain a semi-monomial $\mathbb{Q}(v)$-basis $\{\widetilde{m}_{A'}~|~A'\in\widetilde{\Xi}^-\}$ for $\dot{\mathbb{U}}^{\mathfrak{c}}1_{\epsilon_{0,r+1}}$ from the semi-monomial basis $\{m^{\mathfrak{u}}_{A'}~|~A'\in\widetilde{\Theta}^-\}$ for $\mathbb{U}^-D_{\epsilon_{0,r+1}}$. The homomorphism $\aleph:~\dot{\mathbb{U}^{\mathfrak{c}}}\to{}_\mathbb{Q}\dot{\mathcal{K}}_n^{\mathfrak{c}}$ restricts to a $\mathbb{Q}(v)$-linear map $\aleph|_{\epsilon_{0,r+1}}:~\dot{\mathbb{U}}^{\mathfrak{c}}1_{\epsilon_{0,r+1}}\to{}_\mathbb{Q}\dot{\mathcal{K}}^{\mathfrak{c}}_n[\epsilon_{0,r+1}]$, which sends $\widetilde{m}_A$ to $m_A$ for $A\in\widetilde{\Xi}^-$. Hence $\aleph|_{\epsilon_{0,r+1}}$ is a $\mathbb{Q}(v)$-linear isomorphism. This leads to a $\mathbb{Q}(v)$-linear isomorphism $\aleph|_{\lambda}:~\dot{\mathbb{U}}^{\mathfrak{c}}1_{\lambda}\to{}_\mathbb{Q}\dot{\mathcal{K}}_n^{\mathfrak{c}}[\lambda]$, for any $\lambda\in\widetilde{\Xi}^{\mathrm{diag}}$, via the following commutative diagram:
\begin{displaymath}
\xymatrix{
\dot{\mathbb{U}}^{\mathfrak{c}}1_{\epsilon_{0,r+1}} \ar[r]^{\aleph|_{\epsilon_{0,r+1}}} \ar[d]_{\sharp_{\lambda}}
&{}_\mathbb{Q}\dot{\mathcal{K}}^{\mathfrak{c}}_n[\epsilon_{0,r+1}] \ar[d]^{\sharp_{\lambda}} \\
\dot{\mathbb{U}}^{\mathfrak{c}}1_{\lambda}
\ar[r]_{\aleph|_{\lambda}}
& {}_\mathbb{Q}\dot{\mathcal{K}}^{\mathfrak{c}}_n[\lambda] 
}   
\end{displaymath}
Here $\sharp_{\lambda}:\dot{\mathcal{K}}^{\mathfrak{c}}_n[\epsilon_{0,r+1}]\to\dot{\mathcal{K}}^{\mathfrak{c}}_n[\lambda]$ is a $\mathbb{Q}(v)$-linear isomorphism which sends a semi-monomial basis element $m_{A+\epsilon_{0,r+1}}$ to $m_{A+\lambda}$, and $\sharp_{\lambda}:\dot{\mathbb{U}}^{\mathfrak{c}}1_{\epsilon_{0,r+1}}\to\dot{\mathbb{U}}^{\mathfrak{c}}1_{\lambda}$ is defined accordingly. Putting $\aleph|_{\lambda}$ together, we have shown that $\aleph:~\dot{\mathbb{U}^{\mathfrak{c}}}\to{}_\mathbb{Q}\dot{\mathcal{K}}_n^{\mathfrak{c}}$ is a $\mathbb{Q}(v)$-linear isomorphism. 
Thus, $\aleph$ is indeed an isomorphism of $\mathbb{F}$-algebras.
\end{proof}

\subsection{Quantum symmetric pairs \texorpdfstring{$(\mathbb{U},\mathbb{U}^\mathfrak{\jmath\imath})$}{(U,Uji)}}

\begin{center}
\begin{tikzpicture}[scale=.4]
\node at (0,0.75) {$0$};
\node at (4,0.75) {$1$};
\node at (12,0.75) {$r-1$};
    
\node at (0,-3.75) {$2r$};
\node at (4,-3.75) {$2r-1$};
\node at (12,-3.75) {$r+1$};
    
\node at (16,-0.75) {$r$};
       
\node at (8,0) {$\dots$};
\node at (8,-3) {$\dots$};
    
\foreach \x in {0,2,6} 
{\draw[thick,xshift=\x cm] (\x, 0) circle (0.3); 
\draw[thick,xshift=\x cm] (\x, -3) circle (0.3);}
    
\foreach \x in {8} 
\draw[thick,xshift=\x cm] (\x, -1.5) circle (0.3); 
   
\foreach \x in {0}
{\draw[thick,xshift=\x cm] (\x,0) ++(0.5,0) -- +(3,0);
\draw[thick,xshift=\x cm] (\x,-3) ++(0.5,0) -- +(3,0);}
   
\foreach \x in {6}
{\draw[thick,xshift=\x cm] (\x,0) ++(0.5,0) -- +(3,-1.2);
\draw[thick,xshift=\x cm] (\x,-3) ++(0.5,0) -- +(3,1.2);}
   
\foreach \x in {2,4.5}
{\draw[thick,xshift=\x cm] (\x,0) ++(0.5,0) -- +(2,0);
\draw[thick,xshift=\x cm] (\x,-3) ++(0.5,0) -- +(2,0);}
    
\foreach \x in {0}
\draw[thick,xshift=\x cm] (\x,-0.5) -- +(0,-2);
    
\foreach \x in {0,4} 
\draw[thick,<->, blue, bend right=50] (\x+0.3,-2.5) to (\x+0.3,-0.5);
\foreach \x in {12} 
\draw[thick,<->, blue, bend left=50] (\x-0.3,-2.5) to (\x-0.3,-0.5);
    
\foreach \x in {16} 
\draw[thick,<->, blue, bend right=100,looseness=10] (\x+0.4,-1.7) to (\x+0.4,-1.3);

\end{tikzpicture}
\end{center}

\begin{center}
Figure 2: Dynkin diagram of type $A_{2r}^{(1)}$ with involution of type $\jmath\imath$.
\end{center}

We consider an isomorphic copy of $\mathbb{U}(\widehat{\mathfrak{sl}}_\mathfrak{n})$, denoted by $\mathbb{U}('\widehat{\mathfrak{sl}}_\mathfrak{n})$, which is generated by $E_i$, $F_i$, $K_i^{\pm1}~(i\in[0..\mathfrak{n}]\backslash\{r+1\})$. 
Define $\mathbb{U}^{\jmath\imath}(\widehat{\mathfrak{sl}}_\mathfrak{n})$ as an $\mathbb{F}$-algebra generated by
\begin{equation*}
t_0,\quad e_i,\quad f_i,\quad k_i^{\pm1}\quad (0\leqslant i\leqslant r-1),    
\end{equation*}
subject to the following relations: for $0\leqslant i,j\leqslant r-1$,
\begin{gather*}
k_ik_i^{-1}=k_i^{-1}k_i=1,\quad k_ik_j=k_jk_i,\quad k_it_r=t_rk_i,\\
k_ie_jk_i^{-1}=q^{c_{i,j}+\delta_{i0}\delta_{j,0}}e_{j},\quad k_if_jk_i^{-1}=q^{-c_{i,j}-\delta_{i0}\delta_{j,0}}f_{j},\\
e_ie_j=e_je_i,\quad f_if_j=f_jf_i\quad(|i-j|>1), \\
e_if_j-f_je_i=\delta_{ij}\frac{k_i-k_i^{-1}}{q-q^{-1}}\quad (i,j)\neq(0,0),\\
e_it_r=t_re_i,\quad f_rt_r=t_rf_r\quad(i\leqslant r-2),\\
e_{r-1}^2t_r+t_re_{r-1}^2=(q+q^{-1})e_{r-1}t_re_{r-1},\\
f_{r-1}^2t_r+t_rf_{r-1}^2=(q+q^{-1})f_{r-1}t_rf_{r-1},\\
t_r^2e_{r-1}+e_{r-1}t_r^2=(q+q^{-1})t_re_{r-1}t_r+e_{r-1},\\t_r^2f_{r-1}+f_{r-1}t_r^2=(q+q^{-1})t_rf_{r-1}t_r+f_{r-1},\\
e_i^2e_j+e_je_i^2=(q+q^{-1})e_ie_je_i,\quad f_i^2f_j+f_jf_i^2=(q+q^{-1})f_if_jf_i\quad (|i-j|=1),\\
e_0^2f_0+f_0e_0^2=(q+q^{-1})(e_0f_0e_0-q_0^{\frac{1}{2}}q_1^{\frac{1}{2}}e_0k_0-q_0^{-\frac{1}{2}}q_1^{-\frac{1}{2}}e_0k_0^{-1}),\\
f_0^2e_0+e_0f_0^2=(q+q^{-1})(f_0e_0f_0-q_0^{\frac{1}{2}}q_1^{\frac{1}{2}}k_0f_0-q_0^{-\frac{1}{2}}q_1^{-\frac{1}{2}}k_0^{-1}f_0).
\end{gather*}

\begin{proposition}{\cite[Proposition 4.1]{FLLLWW20}}
There is an injective $\mathbb{F}$-algebra homomorphism $\jmath\imath$: $\mathbb{U}^{\jmath\imath}(\widehat{\mathfrak{sl}}_{\mathfrak{n}})\to\mathbb{U}('\widehat{\mathfrak{sl}}_{\mathfrak{n}})$ such that

\item[(1)] $k_a\mapsto q^{\delta_{i0}}K_iK_{-i-1}^{-1},\quad (0\leqslant i\leqslant r-1)$,

\item[(2)] $e_i\mapsto E_i+F_{-i-1}K_{i}^{-1},\quad f_i\mapsto E_{-i-1}+F_{i}K_{-i-1}^{-1},\quad (1\leqslant i\leqslant r-1)$,

\item[(3)] $e_0\mapsto E_0+q_0^{-\frac{1}{2}}q_1^{-\frac{1}{2}}F_{-1}K_{0}^{-1},\quad f_0\mapsto E_{-1}+q_0^{\frac{1}{2}}q_1^{\frac{1}{2}}q^{-1}F_{0}K_{-1}^{-1}$,

\item[(4)] $t_r\mapsto E_r+q_0^{\frac{1}{2}}q_1^{-\frac{1}{2}}qF_0K_0^{-1}+\frac{q_0^{-\frac{1}{2}}q_1^{\frac{1}{2}}-q_0^{\frac{1}{2}}q_1^{-\frac{1}{2}}}{q-q^{-1}}K_0^{-1}$.
\end{proposition}

\subsection{Isomorphism \texorpdfstring{$\dot{\mathbb{U}}^\mathfrak{\jmath\imath}\simeq\dot{\mathcal{K}}^\mathfrak{\jmath\imath}_{\mathfrak{n}}$}{Uji=Knji}}

Let ${}^{\jmath\imath}\widetilde{\Xi}_{\mathfrak{n}}^{\mathrm{diag}}$ be the set of all diagonal matrices in $\widetilde{\Xi}_{\mathfrak{n}}^{\jmath\imath}$. Denote by $\lambda=\sum_{i}\lambda_iE_{ii}$ a diagonal matrix in ${}^{\jmath\imath}\widetilde{\Xi}_{\mathfrak{n}}^{\mathrm{diag}}$. For $\lambda,\lambda'\in\widetilde{\Xi}_{\mathfrak{n}}^{\mathrm{diag}}$ we define the modified quantum algebra $\dot{\mathbb{U}}^{\jmath\imath}$ similarly to the construction of $\dot{\mathbb{U}}^{\mathfrak{c}}$ as follows:
\begin{equation*}
\dot{\mathbb{U}}^{\jmath\imath}=\bigoplus\limits_{\lambda,\lambda'\in^{}{}^{\jmath\imath}\widetilde{\Xi}_{\mathfrak{n}}^{\mathrm{diag}}}\widetilde{\Xi}_{\mathfrak{n}}^{\mathrm{diag}}{}_{\lambda}\mathbb{U}^{\jmath\imath}_{\lambda'}=\sum_{\lambda\in{}^{\jmath\imath}\widetilde{\Xi}_{\mathfrak{n}}^{\mathrm{diag}}}\mathbb{U}^{\jmath\imath}1_{\lambda}=\sum_{\lambda\in{}^{\jmath\imath}\widetilde{\Xi}_{\mathfrak{n}}^{\mathrm{diag}}}1_{\lambda} \mathbb{U}^{\jmath\imath},    
\end{equation*}
where ${}_{\lambda}\mathbb{U}^{\jmath\imath}_{\lambda'}=\mathbb{U}^{\jmath\imath}/(\sum_{i=0}^{r-1}(k_i-q^{\lambda_i-\lambda_{i+1}})\mathbb{U}^{\jmath\imath}+\sum_{i=0}^{r-1}\mathbb{U}^{\jmath\imath}(k_i-q^{\lambda'_i-\lambda'_{i+1}}))$ and $1_{\lambda}\in{}_{\lambda}\mathbb{U}^{\jmath\imath}_{\lambda}$ is the canonical projection image of the unit of $\mathbb{U}^{\jmath\imath}$.

For $\lambda\in{}^{\jmath\imath}\widetilde{\Xi}_{\mathfrak{n}}^{\mathrm{diag}}$ and $i\in[0..r-1]$, we use the following short-hand notations:
\begin{equation*}
\lambda+\alpha_i=\lambda+E_{ii}^{\theta}-E_{i+1,i+1}^{\theta},\quad\lambda-\alpha_i=\lambda-E_{ii}^{\theta}+E_{i+1,i+1}^{\theta}.    
\end{equation*}
We have a presentation of $\dot{\mathbb{U}}^{\jmath\imath}$ as a $\mathbb{F}$-algebra generated by the symbols, for $i\in[0..r-1]$, $\lambda\in{}^{\jmath\imath}\widetilde{\Xi}_{\mathfrak{n}}^{\mathrm{diag}}$,
\begin{equation*}
1_{\lambda},\quad t_r1_{\lambda},\quad 1_{\lambda}t_r,\quad e_i1_{\lambda},\quad 1_{\lambda}e_i,\quad f_i1_{\lambda},\quad 1_{\lambda}f_i,     
\end{equation*}
subject to the following relations, for $i,j\in[1..r]$, $\lambda,\mu\in{}^{\jmath\imath}\widetilde{\Xi}_{\mathfrak{n}}^{\mathrm{diag}}$, $x,y\in\{1,e_i,e_j,f_i,f_j,t_r\}$:
\begin{gather*}
x1_{\lambda}1_{\mu}y=\delta_{\lambda,\mu}x1_{\lambda}y,\\
e_i1_{\lambda}=1_{\lambda+\alpha_i}e_i,\quad f_i1_{\lambda}=1_{\lambda-\alpha_i}f_i,\quad t_r1_{\lambda}=1_{\lambda}t_r,\\
e_i1_{\lambda}f_j=f_j1_{\lambda+\alpha_i+\alpha_j}e_i\quad(i\neq j),\\
(e_if_i-f_ie_i)1_{\lambda}=\llbracket \lambda_i-\lambda_{i+1}\rrbracket1_{\lambda}\quad(i\neq0),\\
(e_i^2e_j+e_je_i^2)1_{\lambda}=\llbracket 2\rrbracket e_ie_je_i1_{\lambda},\quad (f_i^2f_j+f_jf_i^2)1_{\lambda}=\llbracket 2\rrbracket f_if_jf_i1_{\lambda}\quad (|i-j|=1),\\
e_it_r1_{\lambda}=t_re_i1_{\lambda},\quad f_it_r1_{\lambda}=t_rf_i1_{\lambda}\quad(i\neq r-1),\\
(e_{r-1}^2t_r+t_re_{r-1}^2)1_{\lambda}=\llbracket 2\rrbracket e_{r-1}t_re_{r-1}1_{\lambda},\\
(f_{r-1}^2t_r+t_rf_{r-1}^2)1_{\lambda}=\llbracket 2\rrbracket f_{r-1}t_rf_{r-1}1_{\lambda},\\
(t_r^2e_{r-1}+e_{r-1}t_r^2)1_{\lambda}=\llbracket 2\rrbracket(t_re_{r-1}t_r+e_{r-1})1_{\lambda},\\
(t_r^2f_{r-1}+f_{r-1}t_r^2)1_{\lambda}=\llbracket 2\rrbracket (t_rf_{r-1}t_r+f_{r-1})1_{\lambda},\\
(\llbracket 2\rrbracket e_0f_0e_0-e_0^2f_0-f_0e_0^2)1_{\lambda}=\llbracket 2\rrbracket(q_0^{\frac{1}{2}}q_1^{\frac{1}{2}}q^{\lambda_{0}-\lambda_1}+q_0^{-\frac{1}{2}}q_1^{-\frac{1}{2}}q^{-\lambda_{0}+\lambda_1})e_01_{\lambda},\\
(\llbracket 2\rrbracket f_0e_0f_0-f_0^2e_0-e_0f_0^2)1_{\lambda}=\llbracket 2\rrbracket(q_0^{\frac{1}{2}}q_1^{\frac{1}{2}}q^{\lambda_{0}-\lambda_1-3}+q_0^{-\frac{1}{2}}q_1^{-\frac{1}{2}}q^{-\lambda_{0}+\lambda_1+3})f_01_{\lambda}.
\end{gather*}

For all $i\in[0..r-1]$, $\lambda\in{}^{\jmath\imath}\widetilde{\Xi}_{\mathfrak{n}}^{\mathrm{diag}}$, write
\begin{gather*}
\mathbf{e}_i1_{\lambda}=[\lambda-E_{i+1,i+1}^{\theta}+E_{i,i+1}^{\theta}],\quad \mathbf{f}_i1_{\lambda}=[\lambda-E_{ii}^{\theta}+E_{i+1,i}^{\theta}],\\
t_r1_{\lambda}=[\lambda-E_{rr}^{\theta}+E_{r,r+2}^{\theta}]+q^{\lambda_r}\frac{q_0^{-\frac{1}{2}}q_1^{\frac{1}{2}}-q_0^{\frac{1}{2}}q_1^{-\frac{1}{2}}}{q-q^{-1}}[\lambda].
\end{gather*}
Set ${}_\mathbb{F}\dot{\mathcal{K}}_{\mathfrak{n}}^{\jmath\imath}=\mathbb{F}\otimes_{\mathbb{A}}\dot{\mathcal{K}}_{\mathfrak{n}}^{\jmath\imath}$.

\begin{thm}\label{ji}
There is an isomorphism of $\mathbb{F}$-algebras $\aleph:\dot{\mathbb{U}^{\jmath\imath}}\to{}_\mathbb{F}\dot{\mathcal{K}}_{\mathfrak{n}}^{\jmath\imath}$ determined by 
\begin{equation*}
e_i1_{\lambda}\mapsto\mathbf{e}_i1_{\lambda},\quad f_i1_{\lambda}\mapsto\mathbf{f}_i1_{\lambda},\quad
t_r1_{\lambda}\mapsto\mathbf{t}_r1_{\lambda},\quad
1_{\lambda}\mapsto[\lambda]
\end{equation*} for $i\in[0..r-1]$ and $\lambda\in{}^{\jmath\imath}\widetilde{\Xi}_{\mathfrak{n}}^{\mathrm{diag}}$.
\end{thm}

\begin{proof}
Here we only show the details for
\begin{equation*}
(t_r^2f_{r-1}+f_{r-1}t_r^2)1_{\lambda}=\llbracket 2\rrbracket (t_rf_{r-1}t_r+f_{r-1})1_{\lambda}   
\end{equation*}
regarding $\mathbf{f}_{r-1}1_{\lambda}$ and $\mathbf{t}_r1_{\lambda}$ as follows.
Noting
\begin{align*}
\mathbf{t}_r1_{\lambda}=&\mathbf{e}_r\mathbf{f}_r1_{\lambda} +\frac{q_0^{-\frac{1}{2}}q_1^{\frac{1}{2}}q^{-\lambda_r}-q_0^{\frac{1}{2}}q_1^{-\frac{1}{2}}q^{\lambda_r}}{q-q^{-1}}1_{\lambda},
\end{align*} we have
\begin{align*}
\mathbf{t}_r^2\mathbf{f}_{r-1}1_{\lambda}=&\llbracket 2\rrbracket q^{3\lambda_r}\frac{q_0^{-1}q_1-1}{q-q^{-1}}\big(e_{\lambda-E_{r-1,r-1}^{\theta}+E_{r,r+3}^{\theta}}+e_{\lambda-E_{rr}^{\theta}-E_{r-1,r-1}^{\theta}+E_{r,r-1}^{\theta}+E_{r,r+2}^{\theta}}\big)\\
&+q_0^{-1}q_1q^{3\lambda_r}[2]\big(e_{\lambda-E_{rr}^{\theta}-E_{r-1,r-1}^{\theta}+E_{r,r+2}^{\theta}+E_{r,r+3}^{\theta}}\\
&+e_{\lambda-2E_{rr}^{\theta}-E_{r-1,r-1}^{\theta}+E_{r,r-1}^{\theta}+2E_{r,r+2}^{\theta}}\big)\\
&+\bigg(q^{3\lambda_r}[\lambda_r+1]+q^{3\lambda_r+2}\frac{q_0^{-1}q_1+q_0q_1^{-1}-2}{(q-q^{-1})^2}\bigg)e_{\lambda-E_{r-1,r-1}+E_{r,r-1}},\\
\mathbf{f}_{r-1}\mathbf{t}_r^21_{\lambda}=&\llbracket 2\rrbracket q^{3\lambda_r-2}\frac{q_0^{-1}q_1-1}{q-q^{-1}}e_{\lambda-E_{rr}^{\theta}-E_{r-1,r-1}^{\theta}+E_{r,r-1}^{\theta}+E_{r,r+2}^{\theta}}\\
&+q_0^{-1}q_1q^{3\lambda_r-2}[2]e_{\lambda-2E_{rr}^{\theta}-E_{r-1,r-1}^{\theta}+E_{r,r-1}^{\theta}+2E_{r,r+2}^{\theta}}\\
&+\bigg(q^{3\lambda_r-2}[\lambda_r]+q^{3\lambda_r}\frac{q_0^{-1}q_1+q_0q_1^{-1}-2}{(q-q^{-1})^2}\bigg)e_{\lambda-E_{r-1,r-1}+E_{r,r-1}},\\
\mathbf{t}_r\mathbf{f}_{r-1}\mathbf{t}_r1_{\lambda}=& q^{3\lambda_r}\frac{q_0^{-1}q_1-1}{q-q^{-1}}\big([2]e_{\lambda-E_{rr}^{\theta}-E_{r-1,r-1}^{\theta}+E_{r,r-1}^{\theta}+E_{r,r+2}^{\theta}}+e_{\lambda-E_{r-1,r-1}^{\theta}+E_{r,r+3}^{\theta}}\big)\\
&+q_0^{-1}q_1q^{3\lambda_r-1}\big([2]e_{\lambda-2E_{rr}^{\theta}-E_{r-1,r-1}^{\theta}+E_{r,r-1}^{\theta}+2E_{r,r+2}^{\theta}}\\
&+e_{\lambda-E_{rr}^{\theta}-E_{r-1,r-1}^{\theta}+E_{r,r+2}^{\theta}+E_{r,r+3}^{\theta}}\big)\\
&+\bigg(q^{2\lambda_r}[\lambda_r]+q^{3\lambda_r+1}\frac{q_0^{-1}q_1+q_0q_1^{-1}-2}{(q-q^{-1})^2}\bigg)e_{\lambda-E_{r-1,r-1}+E_{r,r-1}}.
\end{align*}
To sum up, we obtion
\begin{equation*}
(\mathbf{t}_r^2\mathbf{f}_{r-1}+\mathbf{f}_{r-1}\mathbf{t}_r^2)1_{\lambda}=\llbracket 2\rrbracket (\mathbf{t}_r\mathbf{f}_{r-1}\mathbf{t}_r+\mathbf{f}_{r-1})1_{\lambda}.    
\end{equation*}
Therefore, $\aleph$ is an algebra homomorphism. Similarly to \cite[Theorem A.15]{BKLW18}, we know that $\aleph$ is a linear isomorphism. Thus, $\aleph$ is indeed an isomorphism of $\mathbb{F}$-algebras.
\end{proof}

\subsection{Quantum symmetric pairs \texorpdfstring{$(\mathbb{U},\mathbb{U}^\mathfrak{\imath\jmath})$}{(U,Uji)}}
\begin{center}
\begin{tikzpicture}[scale=.4]
\node at (4,0.75) {$1$};
\node at (12,0.75) {$r-1$};
\node at (16,0.75) {$r$};
    
\node at (4,-3.75) {$2r$};
\node at (12,-3.75) {$r+2$};
\node at (16,-3.75) {$r+1$};
    
\node at (0,-0.75) {$0$};
       
\node at (8,0) {$\dots$};
\node at (8,-3) {$\dots$};
    
\foreach \x in {2,6,8} 
{\draw[thick,xshift=\x cm] (\x, 0) circle (0.3); 
\draw[thick,xshift=\x cm] (\x, -3) circle (0.3);}
    
\foreach \x in {0} 
\draw[thick,xshift=\x cm] (\x, -1.5) circle (0.3);

\foreach \x in {0}
{\draw[thick,xshift=\x cm] (\x,-1.2) ++(0.5,0) -- +(3,1.2);
\draw[thick,xshift=\x cm] (\x,-1.8) ++(0.5,0) -- +(3,-1.2);}
   
\foreach \x in {2,4.5}
{\draw[thick,xshift=\x cm] (\x,0) ++(0.5,0) -- +(2,0);
\draw[thick,xshift=\x cm] (\x,-3) ++(0.5,0) -- +(2,0);}

\foreach \x in {6}
{\draw[thick,xshift=\x cm] (\x,0) ++(0.5,0) -- +(3,0);
\draw[thick,xshift=\x cm] (\x,-3) ++(0.5,0) -- +(3,0);}
    
\foreach \x in {8}
\draw[thick,xshift=\x cm] (\x,-0.5) -- +(0,-2);
    
\foreach \x in {4} 
\draw[thick,<->, blue, bend right=50] (\x+0.3,-2.5) to (\x+0.3,-0.5);
\foreach \x in {12,16} 
\draw[thick,<->, blue, bend left=50] (\x-0.3,-2.5) to (\x-0.3,-0.5);
    
\foreach \x in {0} 
\draw[thick,<->, blue, bend left=100,looseness=10] (\x-0.4,-1.7) to (\x-0.4,-1.3);

\end{tikzpicture}
\end{center}

\begin{center}
Figure 3: Dynkin diagram of type $A_{2r}^{(1)}$ with involution of type $\imath\jmath$.
\end{center}
\vspace{0.5cm}

Denote $\mathbb{U}(\widehat{\mathfrak{sl}}_\mathfrak{n})$ as generated by $E_i$, $F_i$, $K_i^{\pm1}~(i\in[0..\mathfrak{n}-1])$. 

Define $\mathbb{U}^{\jmath\imath}(\widehat{\mathfrak{sl}}_\mathfrak{n})$ to be an $\mathbb{F}$-algebra generated by
\begin{equation*}
t_0,\quad e_i,\quad f_i,\quad k_i^{\pm1}\quad (1\leqslant i\leqslant r),    
\end{equation*}
subject to the following relations for $1\leqslant i,j\leqslant r$:
\begin{gather*}
k_ik_i^{-1}=k_i^{-1}k_i=1,\quad k_ik_j=k_jk_i,\quad t_0k_i=k_it_0,\\
k_ie_jk_i^{-1}=q^{c_{i,j}+\delta_{ir}\delta_{j,r}}e_{j},\quad k_if_jk_i^{-1}=q^{-c_{i,j}-\delta_{ir}\delta_{j,r}}f_{j}\\
e_ie_j=e_je_i,\quad f_if_j=f_jf_i\quad(|i-j|>1), \\
e_if_j-f_je_i=\delta_{ij}\frac{k_i-k_i^{-1}}{q-q^{-1}}\quad (i,j)\neq(r,r),\\
t_0e_i=e_it_0,\quad t_0f_i=f_it_0\quad(i\geqslant2),\\
e_1^2t_0+t_0e_1^2=(q+q^{-1})e_1t_0e_1,\quad f_1^2t_0+t_0f_1^2=(q+q^{-1})f_1t_0f_1,\\
t_0^2e_1+e_1t_0^2=(q+q^{-1})t_0e_1t_0+e_1,\quad t_0^2f_1+f_1t_0^2=(q+q^{-1})t_0f_1t_0+f_1,\\
e_i^2e_j+e_je_i^2=(q+q^{-1})e_ie_je_i,\quad f_i^2f_j+f_jf_i^2=(q+q^{-1})f_if_jf_i\quad (|i-j|=1),\\
e_r^2f_r+f_re_r^2=(q+q^{-1})(e_rf_re_r-q_0^{\frac{1}{2}}q_1^{-\frac{1}{2}}q^3e_rk_r-q_0^{-\frac{1}{2}}q_1^{\frac{1}{2}}q^{-3}e_rk_r^{-1}),\\
f_r^2e_r+e_rf_r^2=(q+q^{-1})(f_re_rf_r-q_0^{\frac{1}{2}}q_1^{-\frac{1}{2}}q^3k_rf_r-q_0^{-\frac{1}{2}}q_1^{\frac{1}{2}}q^{-3}k_r^{-1}f_r).
\end{gather*}

\begin{proposition}{\cite[Proposition 4.4]{FLLLWW20}}
There are injective $\mathbb{F}$-algebra homomorphisms $\imath\jmath$: $\jmath\jmath$: $\mathbb{U}^{\mathfrak{c}}(\widehat{\mathfrak{sl}}_{\mathfrak{n}})\to\mathbb{U}(\widehat{\mathfrak{sl}}_{\mathfrak{n}})$ defined by

\item[(1)] $k_i\mapsto q^{-\delta_{ir}}K_iK_{-i-1}^{-1},\quad (1\leqslant i\leqslant r)$,

\item[(2)] $e_i\mapsto E_i+F_{-i-1}K_{i}^{-1},\quad f_i\mapsto E_{-i-1}+F_{i}K_{-i-1}^{-1},\quad (1\leqslant i\leqslant r-1)$,

\item[(3)] $e_r\mapsto E_r+q_0^{-\frac{1}{2}}q_1^{\frac{1}{2}}q^{-1}F_{-r-1}K_{r}^{-1},\quad f_r\mapsto E_{-r-1}+q_0^{\frac{1}{2}}q_1^{-\frac{1}{2}}F_{r}K_{-r-1}^{-1}$,

\item[(4)] $t_0\mapsto E_0+q_0^{-\frac{1}{2}}q_1^{\frac{1}{2}}qF_0K_0^{-1}+\frac{q_0^{\frac{1}{2}}q_1^{\frac{1}{2}}-q_0^{-\frac{1}{2}}q_1^{-\frac{1}{2}}}{q-q^{-1}}K_0^{-1}$.
\end{proposition}

\subsection{Isomorphism \texorpdfstring{$\dot{\mathbb{U}}^\mathfrak{\imath\jmath}\simeq\dot{\mathcal{K}}^\mathfrak{\imath\jmath}_{\mathfrak{n}}$}{Uij=Knij}}

Let ${}^{\imath\jmath}\widetilde{\Xi}_{\mathfrak{n}}^{\mathrm{diag}}$ be the set of all diagonal matrices in $\widetilde{\Xi}_{\mathfrak{n}}^{\imath\jmath}$. Denote by $\lambda=\sum_{i}\lambda_iE_{ii}$ a diagonal matrix in ${}^{\imath\jmath}\widetilde{\Xi}_{\mathfrak{n}}^{\mathrm{diag}}$. For $\lambda,\lambda'\in\widetilde{\Xi}_{\mathfrak{n}}^{\mathrm{diag}}$ we define the modified quantum algebra $\dot{\mathbb{U}}^{\imath\jmath}$ similarly to the construction of $\dot{\mathbb{U}}^{\mathfrak{c}}$ as follows:
\begin{equation*}
\dot{\mathbb{U}}^{\imath\jmath}=\bigoplus\limits_{\lambda,\lambda'\in^{}{}^{\imath\jmath}\widetilde{\Xi}_{\mathfrak{n}}^{\mathrm{diag}}}\widetilde{\Xi}_{\mathfrak{n}}^{\mathrm{diag}}{}_{\lambda}\mathbb{U}^{\imath\jmath}_{\lambda'}=\sum_{\lambda\in{}^{\imath\jmath}\widetilde{\Xi}_{\mathfrak{n}}^{\mathrm{diag}}}\mathbb{U}^{\imath\jmath}1_{\lambda}=\sum_{\lambda\in{}^{\imath\jmath}\widetilde{\Xi}_{\mathfrak{n}}^{\mathrm{diag}}}1_{\lambda} \mathbb{U}^{\imath\jmath},    
\end{equation*}
where ${}_{\lambda}\mathbb{U}^{\imath\jmath}_{\lambda'}=\mathbb{U}^{\imath\jmath}/(\sum_{i=1}^{r}(k_i-q^{\lambda_i-\lambda_{i+1}})\mathbb{U}^{\imath\jmath}+\sum_{i=1}^{r}\mathbb{U}^{\imath\jmath}(k_i-q^{\lambda'_i-\lambda'_{i+1}}))$ and $1_{\lambda}\in{}_{\lambda}\mathbb{U}^{\imath\jmath}_{\lambda}$ is the canonical projection image of the unit of $\mathbb{U}^{\imath\jmath}$.

For $\lambda\in{}^{\imath\jmath}\widetilde{\Xi}_{\mathfrak{n}}^{\mathrm{diag}}$ and $i\in[1..r]$, we use the following short-hand notations:
\begin{equation*}
\lambda+\alpha_i=\lambda+E_{ii}^{\theta}-E_{i+1,i+1}^{\theta},\quad\lambda-\alpha_i=\lambda-E_{ii}^{\theta}+E_{i+1,i+1}^{\theta}.    
\end{equation*}
We have a presentation of $\dot{\mathbb{U}}^{\imath\jmath}$ as an $\mathbb{F}$-algebra generated by
\begin{equation*}
1_{\lambda},\quad t_01_{\lambda},\quad 1_{\lambda}t_0,\quad e_i1_{\lambda},\quad 1_{\lambda}e_i,\quad f_i1_{\lambda},\quad 1_{\lambda}f_i,    \quad(i\in[1..r-1], \lambda\in{}^{\imath\imath}\widetilde{\Xi}_{\eta}^{\mathrm{diag}}), 
\end{equation*}
subject to the following relations, for $i,j\in[1..r-1]$, $\lambda,\mu\in{}^{\imath\jmath}\widetilde{\Xi}_{\mathfrak{n}}^{\mathrm{diag}}$,\\ $x,~y\in\{1,e_i,e_j,f_i,f_j,t_0\}$:
\begin{gather*}
x1_{\lambda}1_{\mu}y=\delta_{\lambda,\mu}x1_{\lambda}y,\\
e_i1_{\lambda}=1_{\lambda+\alpha_i}e_i,\quad f_i1_{\lambda}=1_{\lambda-\alpha_i}f_i,\quad t_01_{\lambda}=1_{\lambda}t_0,\\
e_i1_{\lambda}f_j=f_j1_{\lambda+\alpha_i+\alpha_j}e_i\quad(i\neq j),\\
(e_if_i-f_ie_i)1_{\lambda}=\llbracket \lambda_i-\lambda_{i+1}\rrbracket1_{\lambda}\quad(i\neq r),\\
(e_i^2e_j+e_je_i^2)1_{\lambda}=\llbracket 2\rrbracket e_ie_je_i1_{\lambda},\quad (f_i^2f_j+f_jf_i^2)1_{\lambda}=\llbracket 2\rrbracket f_if_jf_i1_{\lambda}\quad (|i-j|=1),\\
e_it_01_{\lambda}=t_0e_i1_{\lambda},\quad f_it_01_{\lambda}=t_0f_i1_{\lambda}\quad(i\neq1),\\
(e_1^2t_0+t_0e_1^2)1_{\lambda}=\llbracket 2\rrbracket e_1t_0e_11_{\lambda},\quad (f_1^2t_0+t_0f_1^2)1_{\lambda}=\llbracket 2\rrbracket f_1t_0f_11_{\lambda},\\
(t_0^2e_1+e_1t_0^2)1_{\lambda}=\llbracket 2\rrbracket(t_0e_1t_0+e_1)1_{\lambda},\quad (t_0^2f_1+f_1t_0^2)1_{\lambda}=\llbracket 2\rrbracket (t_0f_1t_0+f_1)1_{\lambda},\\
(\llbracket 2\rrbracket e_rf_re_r-e_r^2f_r-f_re_r^2)1_{\lambda}=\llbracket 2\rrbracket(q_0^{-\frac{1}{2}}q_1^{\frac{1}{2}}q^{\lambda_{r+1}-\lambda_r-3}+q_0^{\frac{1}{2}}q_1^{-\frac{1}{2}}q^{-\lambda_{r+1}+\lambda_r+3})e_r1_{\lambda},\\
(\llbracket 2\rrbracket f_re_rf_r-f_r^2e_r-e_rf_r^2)1_{\lambda}=\llbracket 2\rrbracket(q_0^{-\frac{1}{2}}q_1^{\frac{1}{2}}q^{\lambda_{r+1}-\lambda_r}+q_0^{\frac{1}{2}}q_1^{-\frac{1}{2}}q^{-\lambda_{r+1}+\lambda_r})f_r1_{\lambda}.
\end{gather*}

For all $i\in[1..r]$, $\lambda\in{}^{\imath\jmath}\widetilde{\Xi}_{\mathfrak{n}}^{\mathrm{diag}}$, write
\begin{gather*}
\mathbf{e}_i1_{\lambda}=[\lambda-E_{i+1,i+1}^{\theta}+E_{i,i+1}^{\theta}],\quad \mathbf{f}_i1_{\lambda}=[\lambda-E_{ii}^{\theta}+E_{i+1,i}^{\theta}],\\
t_01_{\lambda}=[\lambda-E_{11}^{\theta}+E_{1,-1}^{\theta}]+q^{\lambda_1}\frac{q_0^{\frac{1}{2}}q_1^{\frac{1}{2}}-q_0^{-\frac{1}{2}}q_1^{-\frac{1}{2}}}{q-q^{-1}}[\lambda].
\end{gather*}
Set ${}_\mathbb{F}\dot{\mathcal{K}}_{\mathfrak{n}}^{\imath\jmath}=\mathbb{F}\otimes_{\mathbb{A}}\dot{\mathcal{K}}_{\mathfrak{n}}^{\imath\jmath}$.

\begin{thm}
There is an isomorphism of $\mathbb{F}$-algebras $\aleph:~\dot{\mathbb{U}^{\imath\jmath}}\to{}_\mathbb{F}\dot{\mathcal{K}}_{\mathfrak{n}}^{\imath\jmath}$ such that, $\forall i\in[1..r]$, $\lambda\in{}^{\imath\jmath}\widetilde{\Xi}_{\mathfrak{n}}^{\mathrm{diag}}$,
\begin{equation*}
e_i1_{\lambda}\mapsto\mathbf{e}_i1_{\lambda},\quad f_i1_{\lambda}\mapsto\mathbf{f}_i1_{\lambda},\quad
t_01_{\lambda}\mapsto\mathbf{t}_01_{\lambda},\quad
1_{\lambda}\mapsto[\lambda].    
\end{equation*}
\end{thm}

\begin{proof}
The proof is similar to the  Theorem \ref{ji}.   
\end{proof}

\subsection{Quantum symmetric pairs \texorpdfstring{$(\mathbb{U},\mathbb{U}^\mathfrak{\imath\imath})$}{(U,Uii)}}

\begin{center}
\begin{tikzpicture}[scale=.4]
\node at (4,0.75) {$1$};
\node at (12,0.75) {$r-1$};

\node at (4,-3.75) {$2r-1$};
\node at (12,-3.75) {$r+1$};
    
\node at (0,-0.75) {$0$};
\node at (16,-0.75) {$r$};
       
\node at (8,0) {$\dots$};
\node at (8,-3) {$\dots$};
    
\foreach \x in {2,6} 
{\draw[thick,xshift=\x cm] (\x, 0) circle (0.3); 
\draw[thick,xshift=\x cm] (\x, -3) circle (0.3);}
    
\foreach \x in {0,8} 
\draw[thick,xshift=\x cm] (\x, -1.5) circle (0.3);

 \foreach \x in {0}
 {\draw[thick,xshift=\x cm] (\x,-1.2) ++(0.5,0) -- +(3,1.2);
 \draw[thick,xshift=\x cm] (\x,-1.8) ++(0.5,0) -- +(3,-1.2);}
   
\foreach \x in {2,4.5}
{\draw[thick,xshift=\x cm] (\x,0) ++(0.5,0) -- +(2,0);
\draw[thick,xshift=\x cm] (\x,-3) ++(0.5,0) -- +(2,0);}

\foreach \x in {6}
{\draw[thick,xshift=\x cm] (\x,0) ++(0.5,0) -- +(3,-1.2);
\draw[thick,xshift=\x cm] (\x,-3) ++(0.5,0) -- +(3,1.2);}
    
\foreach \x in {4} 
\draw[thick,<->, blue, bend right=50] (\x+0.3,-2.5) to (\x+0.3,-0.5);
\foreach \x in {12} 
\draw[thick,<->, blue, bend left=50] (\x-0.3,-2.5) to (\x-0.3,-0.5);
    
\foreach \x in {0} 
\draw[thick,<->, blue, bend left=100,looseness=10] (\x-0.4,-1.7) to (\x-0.4,-1.3);
\foreach \x in {16} 
\draw[thick,<->, blue, bend right=100,looseness=10] (\x+0.4,-1.7) to (\x+0.4,-1.3);

\end{tikzpicture}
\end{center}

\begin{center}
Figure 4: Dynkin diagram of type $A_{2r-1}^{(1)}$ with involution of type $\imath\imath$.
\end{center}

Denote $\mathbb{U}(\widehat{\mathfrak{sl}}_{\eta})$ as generated by $E_i$, $F_i$, $K_i^{\pm1}~(i\in[0..\eta]/\{r+1\})$. 

Define $\mathbb{U}^{\imath\imath}(\widehat{\mathfrak{sl}}_\mathfrak{n})$ to be an $\mathbb{F}$-algebra generated by
\begin{equation*}
t_0,\quad t_r,\quad e_i,\quad f_i,\quad k_i^{\pm1}\quad (1\leqslant i\leqslant r-1),    
\end{equation*}
subject to the following relations for $1\leqslant i,j\leqslant r-1$, $a\in\{0,r\}$:
\begin{gather*}
k_1\cdots k_{r-1}=1,\quad k_ik_i^{-1}=k_i^{-1}k_i=1,\quad k_ik_j=k_jk_i,\\
t_ak_i=k_it_a,\quad t_0t_r=t_rt_0,\\
k_ie_jk_i^{-1}=q^{c_{i,j}}e_{j},\quad k_if_jk_i^{-1}=q^{-c_{i,j}}f_{j},\\
e_ie_j=e_je_i,\quad f_if_j=f_jf_i\quad(|i-j|>1), \\
e_it_a=t_ae_i,\quad f_it_a=t_af_i\quad(|i-a|>1), \\
e_if_j-f_je_i=\delta_{ij}\frac{k_i-k_i^{-1}}{q-q^{-1}}\quad (i,j)\neq(r,r),\\
e_i^2t_a+t_ae_i^2=(q+q^{-1})e_it_ae_i,\quad f_i^2t_a+t_af_i^2=(q+q^{-1})f_it_af_i\quad(|i-a|=1),\\
t_a^2e_i+e_it_a^2=(q+q^{-1})t_ae_it_a+e_i,\quad t_a^2f_i+f_it_a^2=(q+q^{-1})t_af_it_a+f_i\quad(|i-a|=1),\\
e_i^2e_j+e_je_i^2=(q+q^{-1})e_ie_je_i,\quad f_i^2f_j+f_jf_i^2=(q+q^{-1})f_if_jf_i\quad (|i-j|=1).
\end{gather*}

\begin{proposition}{\cite[Proposition 4.4]{FLLLWW20}}
There are injective $\mathbb{F}$-algebra homomorphisms $\imath\jmath$: $\jmath\jmath$: $\mathbb{U}^{\mathfrak{c}}(\widehat{\mathfrak{sl}}_{\eta})\to\mathbb{U}(\widehat{\mathfrak{sl}}_{\eta})$ defined by

\item[(1)] $k_i\mapsto K_iK_{-i-1}^{-1},\quad (1\leqslant i\leqslant r-1)$,

\item[(2)] $e_i\mapsto E_i+F_{-i-1}K_{i}^{-1},\quad f_i\mapsto E_{-i-1}+F_{i}K_{-i-1}^{-1},\quad (1\leqslant i\leqslant r-1)$,

\item[(3)] $t_0\mapsto E_0+q_0^{-\frac{1}{2}}q_1^{\frac{1}{2}}qF_0K_0^{-1}+\frac{q_0^{\frac{1}{2}}q_1^{\frac{1}{2}}-q_0^{-\frac{1}{2}}q_1^{-\frac{1}{2}}}{q-q^{-1}}K_0^{-1}$,

\item[(4)] $t_0\mapsto E_0+q_0^{-\frac{1}{2}}q_1^{\frac{1}{2}}qF_0K_0^{-1}+\frac{q_0^{\frac{1}{2}}q_1^{\frac{1}{2}}-q_0^{-\frac{1}{2}}q_1^{-\frac{1}{2}}}{q-q^{-1}}K_0^{-1}$.
\end{proposition}

\subsection{Isomorphism \texorpdfstring{$\dot{\mathbb{U}}^\mathfrak{\imath\imath}\simeq\dot{\mathcal{K}}^\mathfrak{\imath\imath}_{\eta}$}{Uii=Knii}}

Let ${}^{\imath\imath}\widetilde{\Xi}_{\eta}^{\mathrm{diag}}$ be the set of all diagonal matrices in $\widetilde{\Xi}_{\eta}^{\imath\imath}$. Denote by $\lambda=\sum_{i}\lambda_iE_{ii}$ a diagonal matrix in ${}^{\imath\imath}\widetilde{\Xi}_{\eta}^{\mathrm{diag}}$. For $\lambda,\lambda'\in\widetilde{\Xi}_{\eta}^{\mathrm{diag}}$ we define the modified quantum algebra $\dot{\mathbb{U}}^{\imath\jmath}$ similarly to the construction of $\dot{\mathbb{U}}^{\mathfrak{c}}$ as follows:
\begin{equation*}
\dot{\mathbb{U}}^{\imath\imath}=\bigoplus\limits_{\lambda,\lambda'\in^{}{}^{\imath\imath}\widetilde{\Xi}_{\eta}^{\mathrm{diag}}}\widetilde{\Xi}_{\eta}^{\mathrm{diag}}{}_{\lambda}\mathbb{U}^{\imath\imath}_{\lambda'}=\sum_{\lambda\in{}^{\imath\imath}\widetilde{\Xi}_{\eta}^{\mathrm{diag}}}\mathbb{U}^{\imath\imath}1_{\lambda}=\sum_{\lambda\in{}^{\imath\imath}\widetilde{\Xi}_{\eta}^{\mathrm{diag}}}1_{\lambda} \mathbb{U}^{\imath\imath},    
\end{equation*}
where ${}_{\lambda}\mathbb{U}^{\imath\imath}_{\lambda'}=\mathbb{U}^{\imath\jmath}/(\sum_{i=1}^{r-1}(k_i-q^{\lambda_i-\lambda_{i+1}})\mathbb{U}^{\imath\imath}+\sum_{i=1}^{r}\mathbb{U}^{\imath\imath}(k_i-q^{\lambda'_i-\lambda'_{i+1}}))$ and $1_{\lambda}\in{}_{\lambda}\mathbb{U}^{\imath\imath}_{\lambda}$ is the canonical projection image of the unit of $\mathbb{U}^{\imath\imath}$.

For $\lambda\in{}^{\imath\imath}\widetilde{\Xi}_{\eta}^{\mathrm{diag}}$ and $i\in[1..r-1]$, we use the following short-hand notations:
\begin{equation*}
\lambda+\alpha_i=\lambda+E_{ii}^{\theta}-E_{i+1,i+1}^{\theta},\quad\lambda-\alpha_i=\lambda-E_{ii}^{\theta}+E_{i+1,i+1}^{\theta}.    
\end{equation*}
We have a presentation of $\dot{\mathbb{U}}^{\imath\imath}$ as a $\mathbb{F}$-algebra generated by the symbols, for $i\in[1..r]$, $\lambda\in{}^{\imath\imath}\widetilde{\Xi}_{\eta}^{\mathrm{diag}}$,
\begin{equation*}
1_{\lambda},\quad t_01_{\lambda},\quad 1_{\lambda}t_0,\quad t_r1_{\lambda},\quad 1_{\lambda}t_r,\quad e_i1_{\lambda},\quad 1_{\lambda}e_i,\quad f_i1_{\lambda},\quad 1_{\lambda}f_i,     
\end{equation*}
subject to the following relations, for $i,j\in[1..r-1]$, $a\in\{0,r\}$, $\lambda,\mu\in{}^{\imath\imath}\widetilde{\Xi}_{\eta}^{\mathrm{diag}}$, $x,y\in\{1,e_i,e_j,f_i,f_j,t_0,t_r\}$:
\begin{gather*}
x1_{\lambda}1_{\mu}y=\delta_{\lambda,\mu}x1_{\lambda}y,\\
e_i1_{\lambda}=1_{\lambda+\alpha_i}e_i,\quad f_i1_{\lambda}=1_{\lambda-\alpha_i}f_i,\quad t_01_{\lambda}=1_{\lambda}t_0,\\
e_i1_{\lambda}f_j=f_j1_{\lambda+\alpha_i+\alpha_j}e_i\quad(i\neq j),\\
(e_if_i-f_ie_i)1_{\lambda}=\llbracket \lambda_i-\lambda_{i+1}\rrbracket1_{\lambda},\\
(e_i^2e_j+e_je_i^2)1_{\lambda}=\llbracket 2\rrbracket e_ie_je_i1_{\lambda},\quad (f_i^2f_j+f_jf_i^2)1_{\lambda}=\llbracket 2\rrbracket f_if_jf_i1_{\lambda}\quad (|i-j|=1),\\
e_it_a1_{\lambda}=t_ae_i1_{\lambda},\quad f_it_a1_{\lambda}=t_af_i1_{\lambda}\quad(|i-a|>1),\\
(e_i^2t_a+t_ae_i^2)1_{\lambda}=\llbracket 2\rrbracket e_it_ae_i1_{\lambda},\quad (f_i^2t_a+t_af_i^2)1_{\lambda}=\llbracket 2\rrbracket f_it_af_i1_{\lambda}\quad (|i-a|=1),\\
(t_a^2e_i+e_it_a^2)1_{\lambda}=\llbracket 2\rrbracket(t_ae_it_a+e_i)1_{\lambda},\quad t_a^2f_i+f_it_a^2=\llbracket 2\rrbracket (t_af_it_a+f_i)1_{\lambda}\quad (|i-a|=1).
\end{gather*}

For all $i\in[1..r]$, $\lambda\in{}^{\imath\imath}\widetilde{\Xi}_{\eta}^{\mathrm{diag}}$, write
\begin{gather*}
\mathbf{e}_i1_{\lambda}=[\lambda-E_{i+1,i+1}^{\theta}+E_{i,i+1}^{\theta}],\quad\mathbf{f}_i1_{\lambda}=[\lambda-E_{ii}^{\theta}+E_{i+1,i}^{\theta}],\\
t_01_{\lambda}=[\lambda-E_{11}^{\theta}+E_{1,-1}^{\theta}]+q^{\lambda_1}\frac{q_0^{\frac{1}{2}}q_1^{\frac{1}{2}}-q_0^{-\frac{1}{2}}q_1^{-\frac{1}{2}}}{q-q^{-1}}[\lambda],\\
t_r1_{\lambda}=[\lambda-E_{rr}^{\theta}+E_{r,r+2}^{\theta}]+q^{\lambda_r}\frac{q_0^{-\frac{1}{2}}q_1^{\frac{1}{2}}-q_0^{\frac{1}{2}}q_1^{-\frac{1}{2}}}{q-q^{-1}}[\lambda]
\end{gather*}
Set ${}_\mathbb{F}\dot{\mathcal{K}}_n^{\imath\imath}=\mathbb{F}\otimes_{\mathbb{A}}\dot{\mathcal{K}}_n^{\imath\imath}$.

\begin{thm}
There is an isomorphism of $\mathbb{F}$-algebras $\aleph: \dot{\mathbb{U}^{\imath\imath}}\to{}_\mathbb{F}\dot{\mathcal{K}}_{\eta}^{\imath\imath}$ such that, for $i\in[1..r-1]$ and $\lambda\in{}^{\imath\imath}\widetilde{\Xi}_{\eta}^{\mathrm{diag}}$,
\begin{equation*}
e_i1_{\lambda}\mapsto\mathbf{e}_i1_{\lambda},\quad f_i1_{\lambda}\mapsto\mathbf{f}_i1_{\lambda},\quad
t_01_{\lambda}\mapsto\mathbf{t}_01_{\lambda},\quad
t_r1_{\lambda}\mapsto\mathbf{t}_r1_{\lambda},\quad
1_{\lambda}\mapsto[\lambda].    
\end{equation*}
\end{thm}

\begin{proof}
The proof is similar to the proof of Theorem \ref{ji}.   
\end{proof}

\vspace{0.5cm}

\appendix
\section{Comparison of multiplication formulas}

\subsection{Affine type C} 
Let $q_0=1,q_1=q^2$. Rewrite $v:=q^{-1}$ and $T'_w:=q_w^{-1}T_w$. We obtain the affine Hecke algebra of type $\mathbf{C}$ with equal parameter (cf. \cite[\S3.3]{FLLLWW20}). 
In this case, Theorem \ref{1.15} is specialized to the following.
\begin{thm}{\cite[Theorem 4.11]{FLLLW23}}\label{2.1}
Let $A$, $B\in\Xi_{n,d}$ with $B$ tridiagonal and $\mathrm{row}_{\mathfrak{c}}(A)=\mathrm{col}_{\mathfrak{c}}(B)$. Then we have
\begin{equation}\label{FM1C}
e_{B}e_{A}
=\sum_{\scalebox{0.7}{$\substack{T\in\Theta_{B,A} \\ S\in \Gamma_{T}}$}}(v^{2}-1)^{n(S)}v^{2(\ell(A,B,S,T)-n(S)-h(S,T))}\llbracket  A;S;T\rrbracket e_{A^{(T-S)}},
\end{equation}
where $\ell(A,B,S,T)=\ell(A)+\ell(B)-\ell(A^{(T-S)})-\ell(w_{A,T})$. 
\end{thm}

There is another multiplication formula with different forms provided by Fan and Li in \cite[Theorem 2.2]{FL19}. In order to state Fan-Li's formula, let us introduce some notation first.

Denote $\check{X}=(x_{i-1,j})_{i,j\in\mathbb{Z}}$ for $X=(x_{ij})_{i,j\in\mathbb{Z}}$. For $A, X, Y\in\Theta_n$, set 
\begin{equation*} 
A_{X,Y}=A+X-Y-(\check{X}-\check{Y}).
\end{equation*} 
Write $X=(x_{ij})$ and $A_{X,Y}=(a_{ij}^*)$. Denote
\begin{equation*} 
\left[\begin{array}{c}
A_{X,Y}\\
X
\end{array}\right]
=\prod\limits_{(i,j)\in\mathscr{I}}
\left[\begin{array}{c}
a_{ij}^*\\
x_{ij}
\end{array}\right]
\left[\begin{array}{c}
a_{ij}^*-x_{ij}\\
x_{-i,-j}
\end{array}\right]
\prod\limits_{\scalebox{0.7}{$\substack{i=0,-r-1\\0\leqslant k\leqslant x_{ii}-1}$}}
\frac{[a_{ii}^*-2k-1]}{[k+1]},
\end{equation*} 
where \begin{equation*} 
\mathscr{I}=(\{-r-1\}\times(-r-1,\infty))\cup([-r,-1]\times\mathbb{Z})\cup(\{0\}\times(-\infty,0))\subset\mathbb{Z}\times\mathbb{Z}.
\end{equation*} 

Given three sequences $\alpha=(\alpha_i)_{i\in\mathbb{Z}}$, $\beta=(\beta_i)_{i\in\mathbb{Z}}$, $\gamma=(\gamma_i)_{i\in\mathbb{Z}}\in\mathbb{N}^{\mathbb{Z}}$ such that their entries are all zeros except at finitely many places, we define
\begin{align*}
&n(\alpha,\gamma,\beta)\\
=&\sum_{\sigma=(\sigma_{kl})}v^{2(\sum_{k,l,p;l>p}\sigma_{kl}\alpha_p+\sum_{k\in\mathbb{Z}}\sigma_{kk}\alpha_k+\sum_{k>p,l<q}\sigma_{kl}\sigma_{pq})}\frac{\prod\limits_{k\in\mathbb{Z}}[\beta_i]^!}{\prod\limits_{k,l\in\mathbb{Z}}[\sigma_{kl}]^!}\prod\limits_{\scalebox{0.7}{$\substack{l\in\mathbb{Z}\\0\leqslant m\leqslant\gamma_j-\sigma_{jj}-1}$}}(v^{2\alpha_l}-v^{2m}),
\end{align*}
where the sum runs over all the upper triangular matrices $\sigma=(\sigma_{kl})\in\mathrm{Mat}_{\mathbb{Z}\times\mathbb{Z}}(\mathbb{N})$ such that $\mathrm{row}_{\mathfrak{c}}(\sigma)=\beta$ and $\mathrm{col}_{\mathfrak{c}}(\sigma)=\gamma$.

Given a sequence $\boldsymbol{a}=(a_i)_{i\in\mathbb{Z}}$, we define the sequence $\boldsymbol{a}^J$ whose $i$-th entry is $a_{-i}$ for all $i\in\mathbb{Z}$. We set 
\begin{equation*} 
n(X,Y)=\prod\limits_{0\leqslant i\leqslant r}n(X_i,Y_i,(X_{-i-1})^J),
\end{equation*} 
where $X_i$ and $Y_i$ are the $i$-th row vectors of $X$ and $Y$, respectively.

For $\alpha=(\alpha_i)_{i\in\mathbb{Z}}\in\mathbb{Z}^{\mathbb{Z}}$, we set  
\begin{equation*} 
\alpha^\#=(\alpha_i^\#)_{i\in\mathbb{Z}},\quad \mbox{where}\quad\alpha_i^\#=\alpha_{-i-1},~\forall i\in\mathbb{Z}.
\end{equation*} 

Given $A\in\Xi_{n,d}$ and $X,Y\in\Theta_n$, We denote
\begin{equation*} 
\xi_{A,X,Y}=\sum_{\scalebox{0.7}{$\substack{-r-1\leqslant i\leqslant r\\j>l}$}}(a_{ij}^*-x_{ij})x_{il}-(\sum_{\scalebox{0.7}{$\substack{-r\leqslant i\leqslant-1\\j>l}$}}+\sum_{\scalebox{0.7}{$\substack{i=-r-1,0\\2i-2l>j>l}$}})x_{-i,-j}x_{il}-\sum_{\scalebox{0.7}{$\substack{i=-r-1,0\\j<i}$}}\frac{x_{ij}(x_{ij}+1)}{2}.
\end{equation*} 

\begin{thm}{\cite[Theorem 2.2]{FL19}}\label{thm:2.2}
Let $\alpha=(\alpha_i)_{i\in\mathbb{Z}}\in\mathbb{N}^{\mathbb{Z}}$ be such that $\alpha_i=\alpha_{i+n}$ for all $i\in\mathbb{Z}$. If $A,B\in\Xi_{n,d}$ satisfy $\mathrm{col}_{\mathfrak{c}}(B)=\mathrm{row}_{\mathfrak{c}}(A)$ and $B-\sum_{1\leqslant i\leqslant n}\alpha_iE_{\theta}^{i,i+1}$ is diagonal, then we have
\begin{equation}\label{FM2C} 
e_{B}e_{A}=\sum_{X,Y}v^{2\xi_{A,X,Y}}n(X,Y)
\left[\begin{array}{c}
A_{X,Y}\\
X
\end{array}\right]
e_{A_{X,Y}},
\end{equation} 
where the sum runs over $X,Y\in\Theta_n$ subject to $x_{ij}+x_{-i-1,-j}=y_{ij}+y_{-i-1,-j}$, $\mathrm{row}_{\mathfrak{c}}(X)=\alpha$, $\mathrm{row}_{\mathfrak{c}}(Y)=\alpha^{\#}$, $A-Y+\check{Y}\in\Theta_n$, $A_{X,Y}\in\Xi_{n,d}$.
\end{thm}

Intrinsically, Formulas \eqref{FM1C} and \eqref{FM2C} must certainly be identical, although this is not immediately apparent. In the following, we shall provide a direct side-by-side comparison to demonstrate their equivalence.
\begin{proposition}\label{prop:compare}
The right-hand sides of \eqref{FM1C} and \eqref{FM2C} are identical.
\end{proposition}


\begin{proof}
Denote $Y=T_{\theta}-T$, $X=S_{\theta}-S+\widehat{T-S}$.
We have
\begin{align*}
\left[\begin{array}{c}
A_{X,Y}\\
X
\end{array}\right]
=&\prod\limits_{(i,j)\in\mathscr{I}}\frac{[a_{ij}^*]^!}{[a_{ij}^*-x_{ij}]^![x_{ij}]^!}\frac{[a_{ij}^*-x_{ij}]^!}{[a_{ij}^*-x_{ij}-x_{-i,-j}]^![x_{-i,-j}]^!}\prod\limits_{\scalebox{0.7}{$\substack{i=0,-r-1\\0\leqslant k\leqslant x_{ii}-1}$}}
\frac{[a_{ii}^*-2k-1]}{[k+1]}\\
=&\prod\limits_{(i,j)\in\mathscr{I}}\frac{[a_{ij}^*]^!}{[a_{ij}-y_{ij}-y_{-i,-j}]^![x_{ij}]^![x_{-i,-j}]^!}\frac{[a_{00}^*-1]\cdots[a_{00}^*-2x_{00}+1]}{[x_{00}]^!}\\
&\times\frac{[a_{-r-1,-r-1}^*-1]\cdots[a_{-r-1,-r-1}^*-2x_{-r-1,-r-1}+1]}{[x_{-r-1,-r-1}]^!}\\
=&\prod\limits_{(i,j)\in\mathscr{I}}\frac{[a_{ij}^*]^!}{[a_{ij}-y_{ij}-y_{-i,-j}]^![x_{ij}]^![x_{-i,-j}]^!}\frac{[a_{00}^*-1]^{!!}}{[a_{00}-2y_{00}-1]^{!!}[x_{00}]^{!!}}\\
&\times\frac{[a_{-r-1,-r-1}^*-1]^{!!}}{[a_{-r-1,-r-1}-2y_{-r-1,-r-1}-1]^{!!}[x_{-r-1,-r-1}]^{!!}}\\
=&\frac{[A^{(T-S)}]_\mathfrak{c}^!}{[X]^![A-T_{\theta}]_\mathfrak{c}^!}.
\end{align*}

We rewrite
\begin{equation*} 
n(X,Y)=\prod\limits_{1\leqslant i\leqslant r+1}n(X_{i-1},Y_{i-1},(X_{-i})^J),
\end{equation*}
and combine the $\sigma=(\sigma_{kl})$ which satisfied $t_{1-i,-j}-s_{1-i,-j}=\sigma_{jj}$ in the summation of the $n(X_{i-1},Y_{i-1},(X_{-i})^J)$. Let $\widetilde{\sigma}=\sigma-\mathrm{diag}(\sigma)$.

Note
\begin{align*} 
\prod\limits_{\scalebox{0.7}{$\substack{l\in\mathbb{Z}\\0\leqslant m\leqslant\gamma_j-\sigma_{jj}-1}$}}(v^{2\alpha_l}-v^{2m})
=&\prod\limits_{\scalebox{0.7}{$\substack{j\in\mathbb{Z}\\0\leqslant m\leqslant s_{1-i,-j}-1}$}}v^{2m}\prod\limits_{\scalebox{0.7}{$\substack{j\in\mathbb{Z}\\0\leqslant m\leqslant s_{1-i,-j}-1}$}}(v^{2(x_{i-1,j}-m)}-1)\\
=&\prod\limits_{\scalebox{0.7}{$\substack{j\in\mathbb{Z}\\0\leqslant m\leqslant s_{1-i,-j}-1}$}}v^{2m}\prod\limits_{\scalebox{0.7}{$\substack{j\in\mathbb{Z}\\0\leqslant m\leqslant s_{1-i,-j}-1}$}}(v^2-1)\prod\limits_{\scalebox{0.7}{$\substack{j\in\mathbb{Z}\\0\leqslant m\leqslant s_{1-i,-j}-1}$}}[x_{i-1,j}-m]\\
=&v^{2\sum_{j\in\mathbb{Z}}\binom{s_{1-i,-j}}{2}}(v^2-1)^{\sum_{j\in\mathbb{Z}}s_{1-i,-j}}\prod\limits_{j\in\mathbb{Z}}\frac{[x_{i-1,j}]^!}{[x_{i-1,j}-s_{1-i,-j}]^!}\\
=&v^{2\sum_{j\in\mathbb{Z}}\binom{s_{1-i,-j}}{2}}(v^2-1)^{\mathrm{row}(s^\dag)_i}\prod\limits_{j\in\mathbb{Z}}\frac{[x_{i-1,j}]^!}{[t_{ij}-s_{ij}]^!},
\end{align*}
\begin{equation*}
\frac{\prod\limits_{k\in\mathbb{Z}}[\beta_i]^!}{\prod\limits_{k,l\in\mathbb{Z}}[\sigma_{kl}]^!}=\frac{\prod\limits_{j\in\mathbb{Z}}[x_{-i,-j}]^!}{\prod\limits_{k,j\in\mathbb{Z}}[\sigma_{kj}]^!}=\frac{\prod\limits_{j\in\mathbb{Z}}[x_{-i,-j}]^!}{\prod\limits_{k\neq j\in\mathbb{Z}}[\sigma_{kj}]^!\prod\limits_{j\in\mathbb{Z}}[\sigma_{jj}]^!}=\frac{1}{\prod\limits_{k\neq j\in\mathbb{Z}}[\sigma_{kj}]^!}\frac{\prod\limits_{j\in\mathbb{Z}}[x_{-i,-j}]^!}{\prod\limits_{j\in\mathbb{Z}}[t_{1-i,-j}-s_{1-i,-j}]^!},
\end{equation*}
\begin{align*} 
&\sum_{\sigma=(\sigma_{kj})}v^{2\sum_{k>p,j<q}\sigma_{kj}\sigma_{pq}}\frac{1}{\prod\limits_{k\neq j\in\mathbb{Z}}[\sigma_{kj}]^!}\\
=&v^{2\sum_{l\leqslant j}\sigma_{jj}(s_{i,l-1}-s_{il}^\dag)}\sum_{\widetilde{\sigma}=(\widetilde{\sigma}_{kj})}v^{2\sum_{k>p,j<q}\widetilde{\sigma}_{kj}\widetilde{\sigma}_{pq}}\frac{\prod\limits_{j\in\mathbb{Z}}[s_{ij}]^!}{\prod\limits_{k,j\in\mathbb{Z}}[\widetilde{\sigma}_{kj}]^!}\frac{1}{\prod\limits_{j\in\mathbb{Z}}[s_{ij}]^!}\\
=&v^{2\sum_{l\leqslant j}\sigma_{jj}(s_{i,l-1}-s_{il}^\dag)}\sum_{\widetilde{\sigma}=(\widetilde{\sigma}_{kj})}v^{2\sum_{k>p,j<q}\widetilde{\sigma}_{kj}\widetilde{\sigma}_{pq}}\prod\limits_{k,j\in\mathbb{Z}}
\left[\begin{array}{c}
s_{ij}-\sum_{l<k}\widetilde{\sigma}_{lj}\\
\widetilde{\sigma}_{kj}
\end{array}\right]
\frac{1}{\prod\limits_{j\in\mathbb{Z}}[s_{ij}]^!}\\
=&v^{2\sum_{l\leqslant j}\sigma_{jj}(s_{i,l-1}-s_{il}^\dag)}\prod\limits_{j\in\mathbb{Z}}
\left[\begin{array}{c}
\sum_{k\leqslant j}(S-S^{\dag})_{ik}\\
S_{i,j+1}^{\dag}
\end{array}\right]
\frac{1}{\prod\limits_{j\in\mathbb{Z}}[s_{ij}]^!}\\
=&v^{2\sum_{l\leqslant j}\sigma_{jj}(s_{i,l-1}-s_{il}^\dag)}\frac{\llbracket  S\rrbracket _i}{\prod\limits_{j\in\mathbb{Z}}[s_{ij}]^!\prod\limits_{j\in\mathbb{Z}}[s^\dag_{i,j+1}]^!}.
\end{align*}

Hence 
\begin{align*} 
n(X,Y)
=&\prod\limits_{i=1}^{r+1}\sum_{\sigma=(\sigma_{kl})}v^{2(\sum_{k,l,p;l>p}\sigma_{kl}\alpha_p+\sum_{k\in\mathbb{Z}}\sigma_{kk}\alpha_k+\sum_{l\leqslant j}\sigma_{jj}(s_{i,l-1}-s_{il}^\dag)+\sum_{j\in\mathbb{Z}}\binom{s_{1-i,-j}}{2})}\\
&\times (v^2-1)^{\mathrm{row}(s^\dag)_i}\frac{\llbracket  S\rrbracket _i}{\prod\limits_{j\in\mathbb{Z}}[s_{ij}]^!\prod\limits_{j\in\mathbb{Z}}[s^\dag_{i,j+1}]^!}\frac{1}{\prod\limits_{k\neq j\in\mathbb{Z}}[\sigma_{kj}]^!}\\
&\times\frac{\prod\limits_{j\in\mathbb{Z}}[x_{-i,-j}]^!}{\prod\limits_{j\in\mathbb{Z}}[t_{1-i,-j}-s_{1-i,-j}]^!}\prod\limits_{j\in\mathbb{Z}}\frac{[x_{i-1,j}]^!}{[t_{ij}-s_{ij}]^!}.
\end{align*} 

Note that
\begin{equation*} 
\prod\limits_{i=1}^{r+1}(v^2-1)^{\mathrm{row}(s^\dag)_i}=\prod\limits_{i=1}^{r+1}(v^2-1)^{\mathrm{row}(s)_i}=(v^2-1)^{\sum_{i=1}^{r+1}\mathrm{row}(s)_i}=(v^2-1)^{n(S)},
\end{equation*}

\begin{equation*} 
\prod\limits_{i=1}^{r+1}\prod\limits_{j\in\mathbb{Z}}[x_{-i,-j}]^!\prod\limits_{j\in\mathbb{Z}}[x_{i-1,j}]^!=\prod\limits_{i=-r-1}^{r}\prod\limits_{j\in\mathbb{Z}}[x_{i,j}]^!=[X]^!,
\end{equation*}

\begin{equation*} 
\prod\limits_{i=1}^{r+1}\prod\limits_{j\in\mathbb{Z}}[t_{ij}-s_{ij}]^!\prod\limits_{j\in\mathbb{Z}}[t_{1-i,-j}-s_{1-i,-j}]^!=\prod\limits_{i=-r}^{r+1}\prod\limits_{j\in\mathbb{Z}}[t_{ij}-s_{ij}]^!=[T-S]^!,
\end{equation*}

\begin{equation*} 
\prod\limits_{i=1}^{r+1}\frac{\llbracket  S\rrbracket _i}{\prod\limits_{j\in\mathbb{Z}}[s_{ij}]^!\prod\limits_{j\in\mathbb{Z}}[s^\dag_{i,j+1}]^!}=\frac{\prod\limits_{i=1}^{r+1}\llbracket  S\rrbracket _i}{\prod\limits_{i=-r}^{r+1}\prod\limits_{j\in\mathbb{Z}}[s_{ij}]^!}=\frac{\llbracket  S\rrbracket }{[S]^!}.
\end{equation*}
Hence 
\begin{align*} 
\prod\limits_{i=1}^{r+1}\frac{\llbracket  S\rrbracket _i}{\prod\limits_{j\in\mathbb{Z}}[s_{ij}]^!\prod\limits_{j\in\mathbb{Z}}[s^\dag_{i,j+1}]^!}\frac{1}{\prod\limits_{k\neq j\in\mathbb{Z}}[\sigma_{kj}]^!}\frac{\prod\limits_{j\in\mathbb{Z}}[x_{-i,-j}]^!}{\prod\limits_{j\in\mathbb{Z}}[t_{1-i,-j}-s_{1-i,-j}]^!}\prod\limits_{j\in\mathbb{Z}}\frac{[x_{i-1,j}]^!}{[t_{ij}-s_{ij}]^!}=\frac{\llbracket  S\rrbracket }{[S]^!}\frac{[X]^!}{[T-S]^!}.
\end{align*} 

Observe that if $x_{ij}$ (resp. $y_{ij}$) is the first non-zero element in the $i$-th row of matrix $X$ (resp. $Y$), then $s_{i+1,-j}=0$ (resp. $s_{i+1,j}=0$). By this observation, we obtain
\begin{align*}
&\ell(A,B,S,T)-n(S)-h(S,T)\\
=&\sum_{i=1}^{r+1}\sum_{k,l;j\geqslant k}x_{-i,-j}x_{kl}
-\sum_{i=1}^{r+1}\sum_{k,l}s_{ik}x_{kl}
+\sum_{i=1}^r\sum_{j>l}x_{-i,-j}x_{il}
-\sum_{\scalebox{0.7}{$\substack{i=-r-1,0\\j<i}$}}\frac{x_{ij}(x_{ij}+1)}{2}\\
&+\sum_{i=1}^n\sum_{j>l}(A-T_{\theta})_{ij}x_{il}
+\sum_{i=1}^{r+1}\sum_{j\geqslant l}(x_{-i,-j}-s_{ij})(s_{i,l-1}-s_{1-i,-l})
+\sum_{\scalebox{0.7}{$\substack{i=0,r+1\\k\geqslant l}$}}x_{ik}x_{il}.
\end{align*}
Thanks to the above equation, we have
\begin{align*}  
e_{B}*e_{A}
=&\sum_{X,Y}v^{2\xi_{A,X,Y}}n(X,Y)
\left[\begin{array}{c}
A_{X,Y}\\
X
\end{array}\right]
e_{A_{X,Y}}\\
=&\sum_{\scalebox{0.7}{$\substack{T\in\Theta_{B,A} \\ S\in \Gamma_{T}}$}}(v^{2}-1)^{n(S)}v^{2(\ell(A,B,S,T)-n(S)-h(S,T))}
\frac{\llbracket  S\rrbracket }{[S]^!}\frac{[X]^!}{[T-S]^!}\frac{[A^{(T-S)}]_\mathfrak{c}^!}{[X]^![A-T_{\theta}]_\mathfrak{c}^!}
e_{A^{(T-S)}}\\
=&\sum_{\scalebox{0.7}{$\substack{T\in\Theta_{B,A} \\ S\in \Gamma_{T}}$}}(v^{2}-1)^{n(S)}v^{2(\ell(A,B,S,T)-n(S)-h(S,T))}\llbracket  A;S;T\rrbracket e_{A^{(T-S)}}.
\end{align*} 
    
\end{proof}

\subsection{Affine type D} 
Let $q_0=q_1=1$, rewrite $v:=q^{-1}$ and $T'_w:=q_w^{-1}T_w$. We obtain the affine Hecke algebra of type $\mathbf{D}$ with equal parameter (cf. \cite[\S3.3]{FLLLWW20}). 
In this case, Theorem~\ref{1.15} is read as follows.
\begin{corollary}\label{1.18}
Let $A$, $B\in\Xi_{n,d}$ with $B$ tridiagonal and $\mathrm{row}_{\mathfrak{c}}(A)=\mathrm{col}_{\mathfrak{c}}(B)$. We have
\begin{equation*}
e_{B}e_{A}
=\sum_{\scalebox{0.7}{$\substack{T\in\Theta_{B,A} \\ S\in \Gamma_{T}}$}}(v^2-1)^{n(S)}v^{\alpha_{\mathfrak{d}}}\frac{[A^{(T-S)}]_\mathfrak{c}^!\llbracket  S\rrbracket}{[A-T_{\theta}]_\mathfrak{c}^![S]^![T-S]^!}e_{A^{(T-S)}},
\end{equation*}
where $\alpha_{\mathfrak{d}}=2(\ell_{\mathfrak{a}}(B)+\ell_{\mathfrak{a}} (w_{A,T})+\ell_{\mathfrak{a}}(A)-\ell_{\mathfrak{a}} (A^{(T-S)})-n(S)-h(S,T))$
\end{corollary}

In particular, if we take $B=\mathrm{diag}(B)+E_{\theta}^{h,h+1}$ or $\mathrm{diag}(B)+E_{\theta}^{h+1,h}$, then we have the following corollary, which is just \cite[Lemma 2.6]{CF24}.
\begin{corollary}
Assume $A$, $B\in\Xi_{n,d}$ with $\mathrm{row}_{\mathfrak{c}}(A)=\mathrm{col}_{\mathfrak{c}}(B)$.

\item[(1)] If $B-E_{\theta}^{h,h+1}$ is diagonal for some $h\in [1,r-1]$, then we have 
\begin{equation*}
e_{B}e_{A}
=\sum_{\scalebox{0.7}{$\substack{p\in\mathbb{Z}\\ a_{hp}\geqslant (E_{\theta}^{h+1,p})_{h+1,p}}$}}v^{2\sum_{j>p}a_{hj}} [a_{hp}+1]e_{A+E_{\theta}^{hp}-E_{\theta}^{h+1,p}}.
\end{equation*}    

\item[(2)] If $B-E_{\theta}^{0,1}$ is diagonal, then we have 
\begin{align*}
e_{B}e_{A}
=&\sum_{\scalebox{0.7}{$\substack{p\leqslant 0\\ a_{0,p}\geqslant (E_{\theta}^{1,p})_{1,p}}$}}v^{2\sum_{j>p}a_{0,j}} [a_{hp}+1]e_{A+E_{\theta}^{0,p}-E_{\theta}^{1,p}}\\&+\sum_{\scalebox{0.7}{$\substack{p>0\\ a_{0,p}\geqslant (E_{\theta}^{1,p})_{1,p}}$}}v^{2(\sum_{j>p}a_{0,j}-1)} [a_{hp}+1]e_{A+E_{\theta}^{0,p}-E_{\theta}^{1,p}}.
\end{align*} 

\item[(3)] If $B-E_{\theta}^{r,r+1}$ is diagonal, then we have 
\begin{align*}
e_{B}e_{A}
=&\sum_{\scalebox{0.7}{$\substack{p\leqslant r\\ a_{rp}\geqslant (E_{\theta}^{rp})_{rp}}$}}v^{2\sum_{j>p}a_{r+1,j}} [a_{r+1,p}+1]e_{A+E_{\theta}^{r+1,p}-E_{\theta}^{rp}}\\&+\sum_{\scalebox{0.7}{$\substack{p<r\\ a_{rp}\geqslant (E_{\theta}^{rp})_{rp}}$}}v^{2\sum_{j>p}a_{r+1,j}} [a_{r+1,p}+1]e_{A+E_{\theta}^{r+1,p}-E_{\theta}^{rp}}.
\end{align*} 

\item[(4)] If $B-E_{\theta}^{h+1,h}$ is diagonal for some $h\in [1,r-1]$, then we have 
\begin{equation*}
e_{B}e_{A}
=\sum_{\scalebox{0.7}{$\substack{p\in\mathbb{Z}\\ a_{hp}\geqslant (E_{\theta}^{hp})_{hp}}$}}v^{2\sum_{j>p}a_{h+1,j}} [a_{h+1,p}+1]e_{A+E_{\theta}^{h+1,p}-E_{\theta}^{hp}}.
\end{equation*}    
\end{corollary}


\end{document}